\documentclass[11pt,reqno,english]{amsart}
\usepackage[T1]{fontenc}
\usepackage[latin9]{inputenc}
\usepackage{color}
\usepackage{babel}
\usepackage{mathrsfs}
\usepackage{mathtools}
\usepackage{bm}
\usepackage{amsbsy}
\usepackage{amstext}
\usepackage{amsthm}
\usepackage{amssymb}
\usepackage{stmaryrd}
\usepackage{geometry}
\usepackage{setspace}
\usepackage{esint}
\usepackage[pdfusetitle,
 bookmarks=true,bookmarksnumbered=false,bookmarksopen=false,
 breaklinks=false,pdfborder={0 0 1},backref=false,colorlinks=false,hidelinks]
 {hyperref}

\makeatletter
\numberwithin{equation}{section}
\numberwithin{figure}{section}
\theoremstyle{plain}
\newtheorem{thm}{\protect\theoremname}[section]
\theoremstyle{definition}
\newtheorem{defn}[thm]{\protect\definitionname}
\theoremstyle{plain}
\newtheorem{assumption}[thm]{\protect\assumptionname}
\theoremstyle{plain}
\newtheorem{prop}[thm]{\protect\propositionname}
\theoremstyle{plain}
\newtheorem{lem}[thm]{\protect\lemmaname}
\theoremstyle{plain}
\newtheorem{cor}[thm]{\protect\corollaryname}
\theoremstyle{remark}
\newtheorem{rem}[thm]{\protect\remarkname}
\theoremstyle{remark}
\newtheorem*{acknowledgement*}{\protect\acknowledgementname}

\usepackage{babel}
\usepackage{mathrsfs}
\usepackage{bm}
\usepackage{amsbsy}
\usepackage{amstext}

\newcommand{\mc}[1]{{\mathcal #1}}
\newcommand{\mb}[1]{{\mathbf #1}}
\newcommand{\mf}[1]{{\mathfrak #1}}

\newcommand{\bb}[1]{{\mathbb #1}}
\newcommand{\ms}[1]{{\mathscr #1}}

\let\SF@@footnote\footnote
\def\footnote{\ifx\protect\@typeset@protect
    \expandafter\SF@@footnote
  \else
    \expandafter\SF@gobble@opt
  \fi
}
\expandafter\def\csname SF@gobble@opt \endcsname{\@ifnextchar[
  \SF@gobble@twobracket
  \@gobble
}
\edef\SF@gobble@opt{\noexpand\protect
  \expandafter\noexpand\csname SF@gobble@opt \endcsname}
\def\SF@gobble@twobracket[#1]#2{}

\usepackage{cite}
\usepackage{bbm}

\providecommand{\leftsquigarrow}{%
  \mathrel{\mathpalette\reflect@squig\relax}%
}
\newcommand{\reflect@squig}[2]{%
  \reflectbox{$\m@th#1\rightsquigarrow$}%
}

\usepackage{tikz}
\usetikzlibrary[patterns]
\usetikzlibrary{arrows.meta}
\usetikzlibrary{mindmap}
\usetikzlibrary{decorations.pathreplacing}
\usepackage[svgnames]{xcolor}

\makeatother

\providecommand{\acknowledgementname}{Acknowledgement}
\providecommand{\assumptionname}{Assumption}
\providecommand{\corollaryname}{Corollary}
\providecommand{\definitionname}{Definition}
\providecommand{\lemmaname}{Lemma}
\providecommand{\propositionname}{Proposition}
\providecommand{\theoremname}{Theorem}
\providecommand{\remarkname}{Remark}

\begin{document}

\title[Gamma Expansion of Large Deviation Rate Functional for
Diffusions]{The Gamma Expansion of the Level Two Large Deviation Rate
Functional for Reversible  Diffusion Processes}

\author{Claudio Landim, Jungkyoung
Lee, Mauro Mariani}

\address{Claudio Landim
  \hfill\break\indent IMPA \hfill\break\indent Estrada Dona Castorina
  110, \hfill\break\indent
J. Botanico, 22460 Rio de Janeiro, Brazil\hfill\break\indent
{\normalfont and} \hfill\break\indent
Univ. Rouen Normandie, \hfill\break\indent
CNRS Normandie Univ,  LMRS UMR 6085, \hfill\break\indent
F-76000 Rouen, France.} 
\email{landim@impa.br}

\address{Jungkyoung Lee
  \hfill\break\indent Department of Mathematics Education,\hfill\break\indent
Inha University, \hfill\break\indent 100, Inha-ro, Michuhol-gu,  \hfill\break\indent
Incheon, 22212, Republic of Korea.} 
\email{jungkyoung@inha.ac.kr}

\address{Mauro Mariani
  \hfill\break\indent Faculty of Mathematics, \hfill\break\indent
National Research University Higher School of Economics, \hfill\break\indent 6 Usacheva St., Moscow, Russia 119048.} 
\email{mmariani@hse.ru}

\begin{abstract}
Fix a smooth Morse function $U\colon \mathbb{R}^{d}\to\mathbb{R}$ with
finitely many critical points, and consider the solution of the stochastic differential
equation 
\begin{equation*}
d\bm{x}_{\epsilon}(t)=-\nabla U(\bm{x}_{\epsilon}(t))\,dt
\,+\,\sqrt{2\epsilon}\, d\bm{w}_{t}\,,
\end{equation*}
where $(\bm{w}_{t})_{t\ge0}$ represents a $d$-dimensional Brownian
motion, and $\epsilon>0$ a small parameter. Denote by
$\mathcal{P}(\mathbb{R}^{d})$ the space of probability measures on
$\bb R^d$, and by
$\mathcal{I}_{\epsilon} \colon
\mathcal{P}(\mathbb{R}^{d})\to[0,\,\infty]$ the Donsker--Varadhan level
two large deviations rate functional. We express $\mc I_\epsilon$ as
$\mc I_\epsilon = \epsilon^{-1} \mc J^{(-1)} + \mc J^{(0)} +
\sum_{1\le p\le \mf q} (1/\theta^{(p)}_\epsilon) \, \mc J^{(p)}$,
where $\mc J^{(p)}\colon \mc P(\bb R^d) \to [0,+\infty]$ stand for
rate functionals independent of $\epsilon$ and $\theta^{(p)}_\epsilon$
for sequences such that $\theta^{(1)}_\epsilon \to\infty$,
$\theta^{(p)}_\epsilon / \theta^{(p+1)}_\epsilon \to 0$ for
$1\le p< \mf q$.  The speeds $\theta^{(p)}_\epsilon$ correspond to the
time-scales at which the diffusion $\bm{x}_{\epsilon}(\cdot)$ exhibits
a metastable behavior, while the functional $\mc J^{(p)}$ represents
the level two  large deviations rate functional of the finite-state,
continuous-time Markov chains which describe the evolution of the
diffusion $\bm{x}_{\epsilon}(\cdot)$ among the wells in the time-scale
$\theta^{(p)}_\epsilon$.
\end{abstract}

\maketitle

\section{Introduction}

The metastable behavior of Markov processes has attracted some
interest in recent years. We refer to the monographs \cite{OV-meta,
BH-Meta, Landim-Gamma}. In this article, we investigate the metastable
behavior of reversible diffusion processes from an analytical
perspective, by showing that the Donsker--Varadhan level 2 large
deviations rate functional encodes the metastable properties of the
process. The main results explain how to extract from these
functionals the metastable time-scales, states and wells.

Consider a family of diffusion processes in $\mathbb{R}^{d}$ defined
by the stochastic differential equation (SDE)
\begin{equation}
d\bm{x}_{\epsilon}(t)=-\nabla U(\bm{x}_{\epsilon}(t))\,dt
\,+\,\sqrt{2\epsilon}\, d\bm{w}_{t}\,,
\label{e: SDE}
\end{equation}
where $U:\mathbb{R}^{d}\to\mathbb{R}$ is a smooth Morse function with
finitely many critical points, $(\bm{w}_{t})_{t\ge0}$ represents a
$d$-dimensional Brownian motion, and $\epsilon>0$ is a small parameter
standing for the temperature. The process
$\{\bm{x}_{\epsilon}(t)\}_{t\ge0}$ has a unique stationary state, the
probability measure $\pi_{\epsilon}$ given by
\begin{equation}
\pi_{\epsilon}(d\boldsymbol{x})\,=\,
\frac{1}{Z_{\epsilon}}\,e^{-U(\boldsymbol{x})/\epsilon}
\,d\boldsymbol{x}\,,
\label{e_Gibbs}
\end{equation}
where $Z_{\epsilon}:=\int_{\mathbb{R}^{d}}e^{-U(\bm{x})/\epsilon}d\bm{x}$
is a normalization constant, which is finite for all $\epsilon>0$
under suitable conditions (cf. \eqref{e: growth}). In particular,
the process $\{\bm{x}_{\epsilon}(t)\}_{t\ge0}$ is reversible with
respect to $\pi_{\epsilon}$.

Suppose that the function $U$ has multiple local minima, so that the
dynamics \eqref{e: SDE} admits multiple equilibria. In the low
temperature regime $\epsilon\to0$, the drift $-\nabla U$ dominates the
system \eqref{e: SDE}, and the process
$\{\bm{x}_{\epsilon}(t)\}_{t\ge0}$ tends to remain near local
minima. However, due to the small random perturbation, metastable
transitions between local minima occur. Such metastable behavior has
been extensively studied from various perspectives: \cite{FW} obtained
lower and upper bounds for the exit of a domain and described the
metastable behavior through the quasi-potential; \cite{BEGK}
established the Eyring--Kramers law, providing sharp asymptotics for
the mean transition times between local minima of $U$; \cite{BGK}
derived sharp asymptotics for the small eigenvalues of the
infinitesimal generator (cf. \eqref{e_generator}); and \cite{RS} analyzed
successive transitions between global minima of $U$, described by a
certain Markov chain. 

When $U$ possesses a complicated structure, the corresponding
metastable transitions exhibit a rich hierarchical structure. A
complete characterization of this hierarchy was obtained in
\cite{LLS-1st,LLS-2nd}, where it was shown that there exist multiple
critical time scales \footnote{In this article, for two positive
sequences $(\alpha_{\epsilon})_{\epsilon>0}$ and
$(\beta_{\epsilon})_{\epsilon>0}$, we denote by
$\alpha_{\epsilon}\prec\beta_{\epsilon}$,
$\beta_{\epsilon}\succ\alpha_{\epsilon}$ if
$\lim_{\epsilon\to0}\alpha_{\epsilon}/\beta_{\epsilon}=0$.}
$\theta_{\epsilon}^{(1)}\prec\cdots\prec\theta_{\epsilon}^{(\mathfrak{q})}$.
At each scale, the finite-dimensional distributions (FDD) of the
rescaled process
$\{\bm{x}_{\epsilon}(\theta_{\epsilon}^{(p)}t)\}_{t\ge0}$ converge to
the FDD of a finite-state Markov chain $\{{\bf y}^{(p)}(t)\}_{t\ge0}$
for $p=1,\,\dots,\,\mathfrak{q}$.

For any topological space $\Omega$, let
\textcolor{blue}{$\mathcal{P}(\Omega)$} denote the space of
probability measures on $\Omega$ endowed with the weak topology. The
\emph{empirical measure} \textcolor{blue}{$L_{\epsilon}(t)$} of the process
$\{\bm{x}_{\epsilon}(t)\}_{t\ge0}$ is defined by
\begin{equation}
\label{e_def_emp}
L_{\epsilon}(t):=\frac{1}{t}\int_{0}^{t}\delta_{\bm{x}_{\epsilon}(s)}ds\,,
\end{equation}
where, for $\bm{x}\in\Omega$, ${\color{blue}\delta_{\bm{x}}}\in\mathcal{P}(\Omega)$
denotes the Dirac measure at $\bm{x}$. Since the
process $\{\bm{x}_{\epsilon}(t)\}_{t\ge0}$ is ergodic, $L_{\epsilon}(t)$
converges to $\pi_{\epsilon}$ as $t\to\infty$. We write \textcolor{blue}{$\mathbb{P}_{\boldsymbol{x}}^{\epsilon}$}
and \textcolor{blue}{$\mathbb{E}_{\boldsymbol{x}}^{\epsilon}$} for
the law and expectation, respectively, of the process $\{\boldsymbol{x}_{\epsilon}(t)\}_{t\ge0}$
starting from $\bm{x}\in\mathbb{R}^{d}$.
The Donsker--Varadhan \cite{DV} large deviation principle (cf. \eqref{e_def_DV})
states that for
any $\bm{x}\in\mathbb{R}^{d}$ and $\mu\in\mathcal{P}(\mathbb{R}^{d})$,
\[
\mathbb{P}_{\bm{x}}^{\epsilon}[L_{\epsilon}(t)\sim\mu]\approx e^{-t\mathcal{I}_{\epsilon}(\mu)}\,,\ \ \text{as}\ t\to\infty\,,
\]
where $\mathcal{I}_{\epsilon}:\mathcal{P}(\mathbb{R}^{d})\to[0,\,\infty]$
is the level two large deviations rate functional defined in
\eqref{e_def_I}. A precise statement is given in the next section.

Our main focus is the behavior of $\mathcal{I}_{\epsilon}$ as
$\epsilon\to0$. In \cite{BGL-22DV}, it was shown that, as
$\epsilon\to0$, $\epsilon\mathcal{I}_{\epsilon}$ converges to the
functional
\[
\mathcal{J}^{(-1)}(\mu):=\frac{1}{4}\int_{\mathbb{R}^{d}}|\nabla U|^{2}d\mu\,.
\]
We extend this result by showing that the functional $\mathcal{I}_{\epsilon}$
admits a full expansion of the form
\[
\mathcal{I}_{\epsilon}=\frac{1}{\epsilon}\mathcal{J}^{(-1)}+\mathcal{J}^{(0)}+\sum_{p=1}^{\mathfrak{q}}\frac{1}{\theta_{\epsilon}^{(p)}}\mathcal{J}^{(p)}\,,\ \ \text{as}\ \epsilon\to0\,,
\]
where $\mathcal{J}^{(0)}$ is the functional introduced in \eqref{e_J^0}
below, and for each $p\in\llbracket1,\,\mathfrak{q}\rrbracket$, $\mathcal{J}^{(p)}:\mathcal{P}(\mathbb{R}^{d})\to[0,\,\infty]$ is the large deviation
rate functional associated with the limiting chain $\{{\bf y}^{(p)}(t)\}_{t\ge0}$.
Their precise definitions are provided in the next section. Since
the convergence is established via $\Gamma$-convergence (cf. Definition
\ref{def_Gamma}), we refer to this as a full $\Gamma$-expansion,
formally defined in Definition \ref{def_Gamma_exp}.

The investigation of the $\Gamma$-expansion of the level two large
deviations rate functional has been initiated in \cite{DiGM} for the diffusion
\eqref{e: SDE} in the case where all wells have different depths. It
has been derived in the context of finite-state Markov chains
\cite{BGL-22AAP , Landim-Gamma} and for random walks on a potential
field \cite{LMS-Gamma}. It has been extended in \cite{KL25} to the
joint current-empirical measure large deviations rate functional.

Our proof relies on tools from the study of metastability.  To
establish the $\Gamma-\liminf$ inequality, we employ the resolvent
equation approach developed in \cite{LMS-res}. For the
$\Gamma-\limsup$ inequality, we construct sequences of measures
converging to the desired limit, making use of test functions
constructed in\cite{LS-22a}, which approximate equilibrium potentials.

\section{Model and result}

\subsection{Model}

Let $U\in C^{3}(\mathbb{R}^{d})$ be a Morse function (cf. \cite[Definition 1.7]{Nic18}), and assume it satisfies the following growth
condition:\footnote{Throughout the article, $|\cdot|$ will denote either the Euclidean norm for vectors or the cardinality for sets, depending on the context.}
\begin{equation}
\begin{gathered}\lim_{n\to\infty}\inf_{|\boldsymbol{x}|\geq n}\frac{U(\boldsymbol{x})}{|\boldsymbol{x}|}\,=\,\infty\;,\quad\lim_{|\boldsymbol{x}|\to\infty}\frac{\boldsymbol{x}}{|\boldsymbol{x}|}\cdot\nabla U(\boldsymbol{x})\,=\,\infty\,,\\
\lim_{|\boldsymbol{x}|\to\infty}\big\{\,|\nabla U(\boldsymbol{x})|\,-\,2\,\Delta U(\boldsymbol{x})\,\big\}\,=\,\infty\,.
\end{gathered}
\label{e: growth}
\end{equation}
The second growth condition in \eqref{e: growth} implies that all critical points of $U$ lie in a bounded set. Moreover, every critical point is isolated because $U$ is Morse. Hence, $U$ has finitely many critical points.
In this article, we allow distinct critical points to have the same potential value (see Figure \ref{fig_ex}).
It is well known (cf. \cite{BEGK}) that by the growth condition \eqref{e: growth},
for all $a\in\mathbb{R}$,
\begin{equation}
\int_{\{\bm{x}\in\mathbb{R}^{d}:U(\bm{x})\ge a\}}e^{-U(\bm{x})/\epsilon}d\bm{x}\,\le\,C_{a}e^{-a/\epsilon}\,,\label{e_U_L1}
\end{equation}
where $C_{a}>0$ is a constant depending on $a$. In particular, $Z_{\epsilon}<\infty$
for all $\epsilon>0$. The process $\{\bm{x}_{\epsilon}(t)\}_{t\ge0}$
driven by the SDE \eqref{e: SDE} is reversible with respect to the
unique invariant distribution $\pi_{\epsilon}$ given by \eqref{e_Gibbs}.
The infinitesimal generator \textcolor{blue}{$\mathscr{L}_{\epsilon}$}
associated with the process $\{\boldsymbol{x}_{\epsilon}(t)\}_{t\ge0}$
acts on a dense subset of $L^{2}(d\pi_{\epsilon})$.
It is defined as the extension of the differential operator $\widetilde{\mathscr{L}}_{\epsilon}$
given by
\begin{equation}
\widetilde{\mathscr{L}}_{\epsilon}f=-\nabla U\cdot\nabla f
+\epsilon\Delta f\ \ ;\ \ f\in C_{c}^{2}(\mathbb{R}^{d})\,.\label{e_generator}
\end{equation}
Let \textcolor{blue}{$D(\mathscr{L}_{\epsilon})$} denote the domain
of the generator $\mathscr{L}_{\epsilon}$, which is a dense subset
of $L^{2}(d\pi_{\epsilon})$.

\subsection*{Large deviations}
Recall from \eqref{e_def_emp} the definition of the empirical measure of the process $\{\bm{x}_{\epsilon}(t)\}_{t\ge0}$. The Donsker--Varadhan
\cite{DV} large deviation principle for diffusion processes reads as
follows: For any compact set $\mathcal{K}\subset\mathbb{R}^{d}$ and
$\mathcal{A}\subset\mathcal{P}(\mathbb{R}^{d})$,
\begin{equation}
\begin{aligned}-\inf_{\mu\in\mathcal{A}^{o}}\mathcal{I}_{\epsilon}(\mu)\le\liminf_{t\to\infty}\inf_{\bm{x}\in\mathcal{K}}\frac{1}{t}\log\mathbb{P}_{\bm{x}}^{\epsilon}[L_{\epsilon}(t)\in\mathcal{A}]\\
\le\limsup_{t\to\infty}\sup_{\bm{x}\in\mathcal{K}}\frac{1}{t}\log\mathbb{P}_{\bm{x}}^{\epsilon}[L_{\epsilon}(t)\in\mathcal{A}] & \le-\inf_{\mu\in\overline{\mathcal{A}}}\mathcal{I}_{\epsilon}(\mu)\,,
\end{aligned}
\label{e_def_DV}
\end{equation}
where ${\color{blue}\mathcal{I}_{\epsilon}}:\mathcal{P}(\mathbb{R}^{d})\to[0,\,+\infty]$
is the \emph{large deviation rate functional} of the process $\{\bm{x}_{\epsilon}(t)\}_{t\ge0}$
defined by
\begin{equation}
\begin{aligned}\mathcal{I}_{\epsilon}(\mu) & :=\sup_{u>0}\int_{\mathbb{R}^{d}}-\frac{\mathscr{L}_{\epsilon}u}{u}d\mu\\
 & =\sup_{H}\int_{\mathbb{R}^{d}}-e^{-H}\mathscr{L}_{\epsilon}e^{H}d\mu\,.
\end{aligned}
\label{e_def_I}
\end{equation}
In this formula, the supremum is taken either over all positive functions
$u\in D(\mathscr{L}_{\epsilon})$ or equivalently over all $H:\mathbb{R}^{d}\to\mathbb{R}$
such that $e^{H}\in D(\mathscr{L}_{\epsilon})$. For any set $\mathcal{A}$
in a topological space, \textcolor{blue}{$\mathcal{A}^{o}$} and \textcolor{blue}{$\overline{\mathcal{A}}$}
represent its interior and closure, respectively.

Since the process $\{\bm{x}_{\epsilon}(t)\}_{t\ge0}$ is reversible with
respect to the invariant distribution $\pi_{\epsilon}$, \cite[Theorem 5]{DV} yields the representation
\begin{equation}
\mathcal{I}_{\epsilon}(\mu)\,=\,\int f_{\epsilon}(-\mathscr{L}_{\epsilon}f_{\epsilon})\,d\pi_{\epsilon}\,=\,\epsilon\int_{\mathbb{R}^{d}}|\nabla f_{\epsilon}|^{2}d\pi_{\epsilon}\,,\label{e_func_max}
\end{equation}
whenever $\mu\in\mathcal{P}(\mathbb{R}^d)$ is absolutely continuous with respect to $\pi_{\epsilon}$
and the square root of the Radon--Nikodym derivative, $f_{\epsilon}=(d\mu/d\pi_{\epsilon})^{1/2}$,
belongs to $D(\mathscr{L}_{\epsilon})$.

\subsection*{$\Gamma$-convergence}

In this article, we study the $\Gamma$-expansion of the large deviation
rate functional $\mathcal{I}_{\epsilon}$ as $\epsilon\to0$ (see
\cite{Mariani}). Since the expansion is formulated in terms of $\Gamma$-convergence, we first recall the definition of \emph{$\Gamma$-convergence}.
\begin{defn}
\label{def_Gamma}Fix a Polish space $X$ and a functional $F:X\to[0,\,+\infty]$.
A sequence $(F_{\epsilon})_{\epsilon>0}$ of functionals $F_{\epsilon}:X\to[0,\,+\infty]$
$\Gamma$-converges to the functional $F$ as $\epsilon\to0$ if and only if the following two conditions hold:
\begin{enumerate}
\item $\Gamma-\liminf$: For each $x\in X$ and each sequence $(x_{\epsilon})_{\epsilon>0}$
such that $\lim_{\epsilon\to0}x_{\epsilon}=x$, $\liminf_{\epsilon\to0}F_{\epsilon}(x_{\epsilon})\ge F(x)$.
\item $\Gamma-\limsup$: For each $x\in X$, there exists a sequence $(x_{\epsilon})_{\epsilon>0}$
in $X$ such that $\lim_{\epsilon\to0}x_{\epsilon}=x$ and $\limsup_{\epsilon\to0}F_{\epsilon}(x_{\epsilon})\le F(x)$.
\end{enumerate}
\end{defn}

The $\Gamma$-convergence of the large deviations rate functional
$\mathcal{I}_{\epsilon}$ as $\epsilon\to0$, in the context of diffusions,
has been examined recently in \cite{BGL-22DV}.

\subsection*{$\Gamma$-expansion}

We now describe a recursive procedure that produces a
\emph{$\Gamma$-expansion} of the large deviation rate functional
$\mathcal{I}_{\epsilon}$.  Suppose that $\mathcal{I}_{\epsilon}$
$\Gamma$-converges to $\mathcal{J}^{(0)}$ as $\epsilon\to0$. If the
$0$-level set of $\mathcal{J}^{(0)}$ is not a singleton (as in the
case when the potential $U$ has multiple local minima), it is natural
to search for a sequence $(\theta_{\epsilon}^{(1)})_{\epsilon>0}$ of
positive numbers such that $1\prec\theta_{\epsilon}^{(1)}$, and the
rescaled functional $\theta_{\epsilon}^{(1)}\mathcal{I}_{\epsilon}$
admits a non-trivial $\Gamma$-limit.

Let $\mathcal{J}^{(1)}$ denote this limit.
Since $\mathcal{J}^{(0)}$ is the $\Gamma$-limit of $\mathcal{I}_{\epsilon}$,
we have:
\begin{itemize}
\item if $\mathcal{J}^{(1)}(\mu)<\infty$ for some $\mu\in\mathcal{P}(\mathbb{R}^{d})$,
then necessarily $\mu$ belongs to the $0$-level set of $\mathcal{J}^{(0)}$,
\item conversely, if $\mu\in\mathcal{P}(\mathbb{R}^{d})$ belongs to the
$0$-level set of $\mathcal{J}^{(0)}$, then $\mathcal{J}^{(1)}(\mu)<\infty$.
\end{itemize}
If this is not the case, there exists a sequence $(\theta_{\epsilon}')_{\epsilon>0}$
of positive numbers such that $1\prec\theta_{\epsilon}'\prec\theta_{\epsilon}^{(1)}$
and $\theta_{\epsilon}'\mathcal{I}_{\epsilon}$ admits a non-trivial
$\Gamma$-limit.

If the $0$-level set of $\mathcal{J}^{(1)}$ is a singleton, the
procedure stops. Otherwise, we repeat the same process to obtain a second scale. This procedure terminates
once we find a sequence $(\theta_{\epsilon}^{(\mathfrak{q})})_{\epsilon>0}$
and a rate functional $\mathcal{J}^{(\mathfrak{q})}$ whose
$0$-level set is a singleton.

We now consider the reverse direction. If, for every sequence $(\varrho_{\epsilon})_{\epsilon>0}$
of positive numbers such that $\varrho_{\epsilon}\prec1$, the rescaled functional $\varrho_{\epsilon}\mathcal{I}_{\epsilon}$
$\Gamma$-converges to $0$ as $\epsilon\to0$, the expansion is complete.
Otherwise, we can search for a suitable sequence $(\theta_{\epsilon}^{(-1)})_{\epsilon>0}$
of positive numbers such that $\lim_{\epsilon\to0}\theta_{\epsilon}^{(-1)}=0$
and $\theta_{\epsilon}^{(-1)}\mathcal{I}_{\epsilon}$ $\Gamma$-converges
to a functional $\mathcal{J}^{(-1)}$ as $\epsilon\to0$ satisfying
\[
\mathcal{J}^{(-1)}(\mu)=0 \quad \Longleftrightarrow \quad \mathcal{J}^{(0)}(\mu)<\infty \,.
\]
This procedure is iterated until we find a sequence $(\theta_{\epsilon}^{(-\mathfrak{r})})_{\epsilon>0}$
such that $\varrho_{\epsilon}\mathcal{I}_{\epsilon}$ $\Gamma$-converges
to $0$ as $\epsilon\to0$ for all sequences $(\varrho_{\epsilon})_{\epsilon>0}$ of
positive numbers such that $\varrho_{\epsilon}\prec\theta_{\epsilon}^{(-\mathfrak{r})}$.

Based on the previous discussion, we now define the notion of a full $\Gamma$-expansion
of a sequence $(\mathcal{I}_{\epsilon})_{\epsilon>0}$ of functionals
$\mathcal{I}_{\epsilon}:\mathcal{P}(\mathbb{R}^{d})\to[0,\,\infty]$.
\begin{defn}
\label{def_Gamma_exp}Consider a sequence $(\mathcal{I}_{\epsilon})_{\epsilon>0}$
of functionals $\mathcal{I}_{\epsilon}:\mathcal{P}(\mathbb{R}^{d})\to[0,\,\infty]$.
A full $\Gamma$-expansion of $(\mathcal{I}_{\epsilon})_{\epsilon>0}$
is given by the speeds $(\theta_{\epsilon}^{(p)})_{\epsilon>0}$ and
the functionals $\mathcal{J}^{(p)}:\mathcal{P}(\mathbb{R}^{d})\to[0,\,+\infty]$,
$-\mathfrak{r}\le p\le\mathfrak{q}$, if:
\begin{enumerate}
\item The speeds $\theta_{\epsilon}^{(-\mathfrak{r})},\,\dots,\,\theta_{\epsilon}^{(\mathfrak{q})}$
are sequences such that $\theta_{\epsilon}^{(p)}\prec\theta_{\epsilon}^{(p+1)}$,
$-\mathfrak{r}\le p\le\mathfrak{q}-1$.
\item For each $-\mathfrak{r}\le p\le\mathfrak{q}$, $\theta_{\epsilon}^{(p)}\mathcal{I}_{\epsilon}$
$\Gamma$-converges to $\mathcal{J}^{(p)}$ as $\epsilon\to0$.
\item For $-\mathfrak{r}\le p\le\mathfrak{q}-1$, $\mathcal{J}^{(p+1)}(\mu)$
is finite if, and only if, $\mu$ belongs to the $0$-level set of $\mathcal{J}^{(p)}$.
\item For every sequence $(\varrho_{\epsilon})_{\epsilon>0}$ of positive
numbers such that $\varrho_{\epsilon}\prec\theta_{\epsilon}^{(-\mathfrak{r})}$,
$\varrho_{\epsilon}\mathcal{I}_{\epsilon}$ $\Gamma$-converges to
$0$ as $\epsilon\to0$.
\item The $0$-level set of $\mathcal{J}^{(\mathfrak{q})}$ is a singleton.
\end{enumerate}
\end{defn}

The concept of $\Gamma$-expansion for large deviation rate functionals
has recently been established in various settings: reversible and
non-reversible finite state Markov chains \cite{BGL-22AAP,
Landim-Gamma, KL25}, random walks in a potential field
\cite{LMS-Gamma}, and diffusion processes under generic conditions
\cite{DiGM}.

\subsection{Assumption}

In this subsection, we present the main assumptions. Recall that
$U$ is a Morse function satisfying \eqref{e: growth}. We further
assume that there exists $\epsilon_{0}>0$ such that
\begin{equation}
|\nabla U|^{2},\,\Delta U\in L^{2}(d\pi_{\epsilon})\,\text{for all}\,\epsilon\in(0,\,\epsilon_{0})\,.\label{e_L2}
\end{equation}
In Lemma \ref{l_assuL2}, we show that the above assumption is not
restrictive.

Let \textcolor{blue}{$\mathcal{C}_{0}$} denote the set of critical points
of $U$, and let $\nabla^{2}U(\bm{x})$ be the Hessian of $U$ at $\bm{x}\in\mathbb{R}^{d}$.
Denote by \textcolor{blue}{$\mathcal{M}_{0}$} the set of local minima
of $U$, and assume that $|\mathcal{M}_{0}|\ge2$.

For distinct $\bm{c}_{1},\,\bm{c}_{2}\in\mathcal{C}_{0}$,
a \textit{\textcolor{blue}{heteroclinic orbit}} $\phi$ from $\bm{c}_{1}$
to $\bm{c}_{2}$ is a smooth path $\phi:\mathbb{R}\to\mathbb{R}^{d}$
satisfying
\[
\dot{\phi}(t)=-\nabla U(\phi(t)) \quad \text{for all } t \in\mathbb{R} \, ,
\]
together with the boundary conditions
\[
\lim_{t\to-\infty}\phi(t)\,=\,\bm{c}_{1}\ ,\ \ \lim_{t\to+\infty}\phi(t)\,=\,\bm{c}_{2}\,.
\]

Let $\mathcal{S}_{0}$ be the set of saddle points of $U$. Since $U$ is
a Morse function, $\mathcal{S}_{0}$ consists precisely of those
critical points $\bm{\sigma}\in\mathcal{C}_{0}$ whose Hessian
$\nabla^{2}U(\bm{\sigma})$ has one negative eigenvalue and $d-1$
positive eigenvalues. In particular, by the Hartman-Grobman theorem
(cf. \cite[Section 2.8]{Perko}), for every
$\bm{\sigma}\in\mathcal{S}_0$, there exist exactly two heteroclinic
orbits $\phi$ satisfying $\lim_{t\to-\infty}\phi(t)=\bm{\sigma}$.

The following is the main assumption as in \cite{LLS-1st,LLS-2nd}.
\begin{assumption}
\label{assu: hetero}Fix $\boldsymbol{\sigma}\in\mathcal{S}_{0}$
and let $\phi_{\pm}$ be the two heteroclinic orbits satisfying $\lim_{t\to-\infty}\phi_{\pm}(t)=\bm{\sigma}$.
Then, $\lim_{t\to+\infty}\phi_{\pm}(t)\in\mathcal{M}_{0}$.
\end{assumption}

\begin{rem}
We present an example of a potential $U$ that violates Assumption \ref{assu: hetero}.
Let $U:\mathbb{R}^{2}\to\mathbb{R}$ be defined by
\[
U(x,y)=(x^{2}-y^{2}-1)^{2}
+\left(y^{2}-\frac{1}{2}\right)^{2}.
\]
Its gradient and Hessian are given by
\[
\nabla U(x,y)
=
\left(
4x(x^{2}-y^{2}-1),\,
4y\left(-x^{2}+2y^{2}+\frac{1}{2}\right)
\right)
\]
and
\[
\nabla^{2}U(x,y)
=
\begin{pmatrix}
12x^{2}-4y^{2}-4 & -8xy\\
-8xy & -4x^{2}+24y^{2}+2
\end{pmatrix},
\]
respectively. Thus, $(0,0)$ and $(\pm1,0)$ are saddle points.
Moreover, since
\[
\nabla U(x,0)=(4x^{3}-4x,\,0),
\]
the $x$-axis is invariant under the negative gradient flow,
which reduces on this axis to
\[
\dot{x}(t)=4x(t)-4x(t)^{3},
\qquad
y(t)=0 \quad ; \quad t\in \mathbb{R}.
\]
The two heteroclinic orbits emanating from $(0,0)$
may therefore be parametrized as
\[
\phi_{\pm}(t)
=
\left(\frac{\pm1}{\sqrt{1+e^{-8t}}},\,0\right)
\quad ; \quad t\in\mathbb{R}.
\]
They satisfy
\[
\lim_{t\to-\infty}\phi_{\pm}(t)=(0,0),
\quad \text{and} \quad
\lim_{t\to+\infty}\phi_{\pm}(t)=(\pm1,0).
\]
Since both limiting points are saddles rather than local minima,
Assumption \ref{assu: hetero} fails.
\end{rem}

Finally,  to simplify the computations, we normalize $U$ by adding a constant so that
\[
\min_{\bm{x}\in\mathbb{R}^d} U(\bm{x})=0.
\]

\subsection{Metastability}

\subsubsection{\label{subsec_tree}Tree structure}

We now introduce the \emph{tree structure} associated with the metastable behavior
of the process $\{\bm{x}_{\epsilon}(t)\}_{t\ge0}$. This structure
consists of a positive integer $\mathfrak{q}\in\mathbb{N}$ and a family of quintuples:\footnote{In this article, for $a<b$, $\llbracket a,\,b\rrbracket:=[a,\,b]\cap\mathbb{Z}$.}
\[
{\color{blue}\Lambda^{(n)}}\,:=\,\left(d^{(n)},\,\mathscr{V}^{(n)},\,\mathscr{N}^{(n)},\,\mathbf{\widehat{y}}^{(n)},\,\mathbf{y}^{(n)}\right)\ \text{for}\ n\in\llbracket1,\,\mathfrak{q}\rrbracket\,.
\]
A rigorous definition is provided in Section \ref{sec_tree_rig}.
\begin{defn}[Tree structure]
\label{def:tree}
A tree structure is specified by:
\begin{enumerate}
\item A positive integer ${\color{blue}\mathfrak{q}\ge1}$ denoting the
number of time scales. 
\item A finite sequence of depths $0<d^{(1)}<\cdots<d^{(\mathfrak{q})}<\infty$
and time-scales
\[
{\color{blue}\theta_{\epsilon}^{(p)}}\,:=\,\exp\frac{d^{(p)}}{\epsilon}\;\;;\ \;p\in\llbracket1,\,\mathfrak{q}\rrbracket\,.
\]
\item A finite sequence of partitions $\mathscr{V}^{(p)}\cup\mathscr{N}^{(p)}$,
$p\in\llbracket1,\,\mathfrak{q}\rrbracket$, of $\mathcal{M}_{0}$.
\item A finite sequence of continuous-time Markov chains $\{\widehat{{\bf y}}^{(p)}(t)\}_{t\ge0}$
and $\{{\bf y}^{(p)}(t)\}_{t\ge0}$, $p\in\llbracket1,\,\mathfrak{q}\rrbracket$,
on $\mathscr{V}^{(p)}\cup\mathscr{N}^{(p)}$ and $\mathscr{V}^{(p)}$,
respectively.
\end{enumerate}
\end{defn}

At the first scale\footnote{In this article, we sometimes write
$\bm{m}$ for $\{\bm{m}\}$.},
\begin{equation}
{\color{blue}\mathscr{V}^{(1)}}\,:=\,\left\{ \{\bm{m}\}:\bm{m}\in\mathcal{M}_{0}\right\} \;,\;\;{\color{blue}\mathscr{N}^{(1)}}\,:=\,\varnothing\ ,\ \ {\color{blue}\mathscr{S}^{(1)}}:=\mathscr{V}^{(1)}\cup\mathscr{N}^{(1)}\,.\label{e_V_1}
\end{equation}
Let \textcolor{blue}{$d^{(1)}$} be the first depth (precisely defined
below display \eqref{e_def_Xi}), and
\textcolor{blue}{$\{{\bf y}^{(1)}(t)\}_{t\ge0}=\{\widehat{{\bf
y}}^{(1)}(t)\}_{t\ge0}$} be the $\mathscr{V}^{(1)}$-valued Markov
chain defined in Section \ref{subsec: MC1}.  This defines
\textcolor{blue}{$\Lambda^{(1)}$}.

Denote by ${\color{blue}\mathscr{R}_{1}^{(1)},\dots,\mathscr{R}_{\mathfrak{n}_{1}}^{(1)}}$ the closed irreducible classes of the Markov chain $\{{\bf y}^{(1)}(t)\}_{t\ge0}$,
and by ${\color{blue}\mathscr{T}^{(1)}}$ the set of its transient states.
If $\mathfrak{n}_{1}=1$, then $\mathfrak{q}=1$ and the
construction terminates. If $\mathfrak{n}_{1}>1$, we add a new
layer to the tree, as explained below.

Suppose that the quintuples $\Lambda^{(1)},\,\dots,\,\Lambda^{(p)}$
have already been constructed. Let \textcolor{blue}{$\mathscr{R}_{1}^{(p)},\dots,\mathscr{R}_{\mathfrak{n}_{p}}^{(p)}$ and
$\mathscr{T}^{(p)}$} denote the closed irreducible classes and the set of transient
states of the Markov chain $\{{\bf y}^{(p)}(t)\}_{t\ge0}$ on $\mathscr{V}^{(p)}$,
respectively. Define
\begin{equation}
{\color{blue}\mathcal{M}_{i}^{(p+1)}}\,:=\,\bigcup_{\mathcal{M}\in\mathscr{R}_{i}^{(p)}}\mathcal{M}\;\;;\;i\in\llbracket1,\,\mathfrak{n}_{p}\rrbracket\,.\label{e_M_p+1}
\end{equation}
If $\mathfrak{n}_{p}=1$, the procedure stops
and $\mathfrak{q}=p$. If $\mathfrak{n}_{p}\ge2$, we construct a new layer by setting
\begin{equation}
{\color{blue}\mathscr{V}^{(p+1)}}\,:=\,\{\mathcal{M}_{1}^{(p+1)},\,\dots,\,\mathcal{M}_{\mathfrak{n}_{p}}^{(p+1)}\}\;,\ \ {\color{blue}\mathscr{N}^{(p+1)}}\,:=\,\mathscr{N}^{(p)}\cup\mathscr{T}^{(p)}\ ,\ \ {\color{blue}\mathscr{S}^{(p+1)}}:=\mathscr{V}^{(p+1)}\cup\mathscr{N}^{(p+1)}\,.\label{e_V_p+1}
\end{equation}
It follows immediately that if
$\mathscr{S}^{(p)}=\mathscr{V}^{(p)}\cup\mathscr{N}^{(p)}$ is a
partition of $\mathcal{M}_{0}$, then so is $\mathscr{S}^{(p+1)}$.  Let
\textcolor{blue}{$d^{(p+1)}$} be the $(p+1)$-th depth, defined in
display \eqref{e_Xi}, let
\textcolor{blue}{$\{\widehat{{\bf y}}^{(p+1)}(t)\}_{t\ge0}$} be the
$\mathscr{S}^{(p+1)}$-valued Markov chain defined in Section
\ref{subsec: MC2}, and let
\textcolor{blue}{$\{{\bf y}^{(p+1)}(t)\}_{t\ge0}$} denote its trace
process on $\mathscr{V}^{(p+1)}$. This defines
\textcolor{blue}{$\Lambda^{(p+1)}$.}  As
$\mathfrak{n}_{p+1}<\mathfrak{n}_{p}$, this procedure terminates after
finitely many steps. Denote by $\mathfrak{q}$ the total number of
constructed quintuples $\Lambda^{(p)}$.

\subsubsection{Metastability}

For $H\in\mathbb{R}$, define the level sets
\begin{equation}
{\color{blue}\{U<H\}}:=\left\{ \boldsymbol{x}\in\mathbb{R}^{d}:U(\boldsymbol{x})<H\right\} \;\text{and}\ \ {\color{blue}\{U\le H\}}:=\left\{ \boldsymbol{x}\in\mathbb{R}^{d}:U(\boldsymbol{x})\le H\right\} \,.\label{eq:level}
\end{equation}
For each $\boldsymbol{m}\in\mathcal{M}_{0}$ and $r>0$, denote by
${\color{blue}\mathcal{W}^{r}(\boldsymbol{m}}\mathclose{\color{blue})}$
the connected component of $\{U\le U(\boldsymbol{m})+r\}$
containing $\boldsymbol{m}$. Take ${\color{blue}r_{0}}>0$ small
enough so that the conditions (a)-(e) in Appendix \ref{subsec_valley}
hold. In particular, $\boldsymbol{m}$ is the unique critical
point of $U$ in $\mathcal{W}^{3r_{0}}(\boldsymbol{m})$.

Define the\emph{
valley }around $\boldsymbol{m}$ as 
\begin{equation}
{\color{blue}\mathcal{E}(\boldsymbol{m}}):=\mathcal{W}^{r_{0}}(\boldsymbol{m})\,.\label{e_Em}
\end{equation}
For $\mathcal{M}\subset\mathcal{M}_{0}$, write $\mathcal{E}(\mathcal{M})$
for the union of the valleys around local minima of $\mathcal{M}$:
\begin{equation}
{\color{blue}\mathcal{E}(\mathcal{M}}\mathclose{\color{blue})}:=\bigcup_{\boldsymbol{m}\in\mathcal{M}}\mathcal{E}(\boldsymbol{m})\,,\label{e_E_M}
\end{equation}
and define
\[
{\color{blue}\mathcal{E}^{(p)}}:=\bigcup_{\mathcal{M}\in\mathscr{V}^{(p)}}\mathcal{E}(\mathcal{M})\ ;\ p\in\llbracket1,\,\mathfrak{q}\rrbracket\,.
\]
For $\mathcal{M}\in\mathscr{V}^{(p)}$, denote by \textcolor{blue}{$\mathcal{Q}_{\mathcal{M}}^{(p)}$}
the law of $\{{\bf y}^{(p)}(t)\}_{t\ge0}$ starting from $\mathcal{M}$
and the corresponding expectation.

The following theorem is the main result of \cite{LLS-1st,LLS-2nd}.

\begin{thm}
Fix $p\in\llbracket1,\,\mathfrak{q}\rrbracket$ and $\mathcal{M}\in\mathscr{V}^{(p)}$.
Then, for all $t>0$, $\bm{x}\in\mathcal{E}(\mathcal{M})$, and $\mathcal{M}'\in\mathscr{V}^{(p)}$,
\[
\lim_{\epsilon\to0}\mathbb{P}_{\bm{x}}^{\epsilon}\left[\bm{x}_{\epsilon}(\theta_{\epsilon}^{(p)}t)\in\mathcal{E}(\mathcal{M}')\right]=\mathcal{Q}_{\mathcal{M}}^{(p)}\left[{\bf y}^{(p)}(t)=\mathcal{M}'\right]\,.
\]
In other words, the behavior of
$\{\bm{x}_{\epsilon}(\theta_{\epsilon}^{(p)}t)\}_{t\ge0}$ in the time
scale $\theta_{\epsilon}^{(p)}$ is described by the Markov chain
$\{{\bf y}^{(p)}(t)\}_{t\ge0}$.
\end{thm}

\subsection{Measures}

For each $\bm{m}\in\mathcal{M}_{0}$ and $\mathcal{M}\subset\mathcal{M}_{0}$,
define
\begin{equation}
{\color{blue}\nu(\bm{m}}\mathclose{\color{blue})}:=\frac{1}{\sqrt{\det\nabla^{2}U(\bm{m})}}\,,\ {\color{blue}\nu(\mathcal{M}}\mathclose{\color{blue})}:=\sum_{\bm{m}'\in\mathcal{M}}\nu(\bm{m}')\,,\ {\color{blue}\nu_{\star}}:=\nu(\mathcal{M}_{\star})\,,\label{e_def_nu}
\end{equation}
where \textcolor{blue}{$\mathcal{M}_{\star}$} denotes the set of
global minima of $U$.

Recall that for each $p\in\llbracket1,\,\mathfrak{q}\rrbracket$,
$\mathscr{R}_{1}^{(p)},\,\dots,\,\mathscr{R}_{\mathfrak{n}_{p}}^{(p)}$
are the closed irreducible classes of the Markov chain $\{{\bf y}^{(p)}(t)\}_{t\ge0}$.
For $i\in\llbracket1,\mathfrak{n}_{p}\rrbracket$, define the probability
measure ${\color{blue}\nu_{i}^{(p)}}\in\mathcal{P}(\mathscr{R}_{i}^{(p)})$
by
\begin{equation}
\nu_{i}^{(p)}(\mathcal{M}):=\frac{\nu(\mathcal{M})}{\nu(\mathcal{M}_{i}^{(p+1)})}\quad;\quad\mathcal{M}\in\mathscr{R}_{i}^{(p)}\,,\label{e_nu_i}
\end{equation}
where $\mathcal{M}_{i}^{(p+1)}$ is defined in \eqref{e_M_p+1}.
By \cite[Proposition 12.7]{LLS-2nd}, $\nu_{i}^{(p)}$ is the unique
stationary distribution of $\{{\bf y}^{(p)}(t)\}_{t\ge0}$ restricted
to $\mathscr{R}_{i}^{(p)}$. Moreover, since $\{{\bf y}^{(p)}(t)\}_{t\ge0}$ has only finitely many
closed irreducible classes $\mathscr{R}_{1}^{(p)},\,\dots,\,\mathscr{R}_{\mathfrak{n}_{p}}^{(p)}$,
all stationary distributions of $\{{\bf y}^{(p)}(t)\}_{t\ge0}$ are convex combinations of $\nu_{1}^{(p)},\,\dots,\,\nu_{\mathfrak{n}_{p}}^{(p)}$.

For $p\in\llbracket1,\,\mathfrak{q}\rrbracket$ and $\mathcal{M}\in\mathscr{V}^{(p)}$,
define a probability measure ${\color{blue}\pi_{\mathcal{M}}}\in\mathcal{P}(\mathcal{M})\subset\mathcal{P}(\mathbb{R}^{d})$
by
\[
\pi_{\mathcal{M}}:=\sum_{\bm{m}\in\mathcal{M}}\frac{\nu(\bm{m})}{\nu(\mathcal{M})}\delta_{\bm{m}}\,.
\]
Note that $\pi_{\bm{m}}=\delta_{\bm{m}}$ for
$\bm{m}\in\mathscr{V}^{(1)}$.  Clearly, for
$p\in\llbracket2,\,\mathfrak{q}\rrbracket$ and
$\mathcal{M}\in\mathscr{V}^{(p)}$,
\begin{equation}
\pi_{\mathcal{M}}=\sum_{\mathcal{M}'\in\mathscr{R}^{(p-1)}(\mathcal{M})}\frac{\nu(\mathcal{M}')}{\nu(\mathcal{M})}\pi_{\mathcal{M}'}\,,
\label{e_pi}
\end{equation}
where \textcolor{blue}{$\mathscr{R}^{(p-1)}(\mathcal{M})$} is the
closed irreducible class of $\{{\bf y}^{(p-1)}(t)\}_{t\ge0}$ such that
\begin{equation}
\mathcal{M}=\bigcup_{\mathcal{M}'\in\mathscr{R}^{(p-1)}(\mathcal{M})}\mathcal{M}'\,.
\label{eq: R^p_M}
\end{equation}

\subsection{Main result}

To state the main result, we start by defining the limiting functionals.
For each $p\in\llbracket1,\,\mathfrak{q}\rrbracket$, let
$\mathfrak{L}^{(p)}$ denote the infinitesimal generator of the Markov
chain $\{{\bf y}^{(p)}(t)\}_{t\ge0}$. The level two large deviation
rate functional
${\color{blue}\mathfrak{J}^{(p)}}:\mathcal{P}(\mathscr{V}^{(p)})\to[0,\,\infty]$
associated with the chain $\{{\bf y}^{(p)}(t)\}_{t\ge0}$ is defined by
\[
\mathfrak{J}^{(p)}(\omega):=\sup_{{\bf u}>0}\sum_{\mathcal{M}\in\mathscr{V}^{(p)}}-\omega(\mathcal{M})\frac{\mathfrak{L}^{(p)}{\bf u}(\mathcal{M})}{{\bf u}(\mathcal{M})}\,,
\]
where the supremum is taken over all positive functions ${\bf u}:\mathscr{V}^{(p)}\to(0,\,\infty)$.
The lifting ${\color{blue}\mathcal{J}^{(p)}}:\mathcal{P}(\mathbb{R}^{d})\to[0,\,+\infty]$
of the functional $\mathfrak{J}^{(p)}$ on $\mathcal{P}(\mathbb{R}^{d})$
is defined by
\begin{equation}
\mathcal{J}^{(p)}(\mu):=\begin{cases}
\mathfrak{J}^{(p)}(\omega) &
\text{if}\;\; \mu=\sum_{\mathcal{M}\in\mathscr{V}^{(p)}}
\omega(\mathcal{M})\pi_{\mathcal{M}}\ ,\
\omega\in\mathcal{P}(\mathscr{V}^{(p)})\,,\\
\infty & \text{otherwise}\,.
\end{cases}\label{e_J^p}
\end{equation}

For $\bm{x}\in\mathcal{C}_{0}$, define
\begin{equation}
\zeta(\bm{x}):=\sum_{k=1}^{d}-\min\{\,\lambda_{k}(\bm{x}),\,0\,\}\,,\label{e_zeta}
\end{equation}
where $\lambda_{1}(\bm{x}),\,\dots,\,\lambda_{d}(\bm{x})$ are the
eigenvalues of $\nabla^{2}U(\bm{x})$.
Equivalently, $\zeta(\bm{x})$ is the sum of the absolute values of the negative eigenvalues of $\nabla^{2}U(\bm{x})$; positive eigenvalues do not contribute.
Define ${\color{blue}\mathcal{J}^{(0)}}:\mathcal{P}(\mathbb{R}^{d})\to[0,\,\infty]$
by
\begin{equation}
\mathcal{J}^{(0)}(\mu):=\begin{cases}
\sum_{\bm{x}\in\mathcal{C}_{0}}\omega(\bm{x})\,\zeta(\bm{x}) & \text{if}\;\; \mu=\sum_{\bm{x}\in\mathcal{C}_{0}}\omega(\bm{x})\delta_{\bm{x}}\ ,\
\omega\in\mathcal{P}(\mathcal{C}_0)\,,\\
\infty & \text{otherwise}\,.
\end{cases}\label{e_J^0}
\end{equation}

Finally, define ${\color{blue}\mathcal{J}^{(-1)}}:\mathcal{P}(\mathbb{R}^{d})\to[0,\,+\infty]$
by
\[
\mathcal{J}^{(-1)}(\mu):=\frac{1}{4}\int_{\mathbb{R}^{d}}|\nabla U|^{2}d\mu\,.
\]

Set ${\color{blue}\theta_{\epsilon}^{(-1)}}:=\epsilon$ and ${\color{blue}\theta_{\epsilon}^{(0)}}:=1$,
and recall from Definition \ref{def:tree} the time
scales $\theta_{\epsilon}^{(p)}$ for $1\le p\le\mathfrak{q}$. The
main result of this article reads as follows.

\begin{thm}
\label{t_main}
Assume that conditions \eqref{e: growth}, \eqref{e_L2}, and Assumption
\ref{assu: hetero} are in force. Then, the full $\Gamma$-expansion of
$(\mathcal{I}_{\epsilon})_{\epsilon>0}$, as in Definition
\ref{def_Gamma_exp}, is given by the speeds
$(\theta_{\epsilon}^{(p)},\,\epsilon>0)$ and the functionals
$\mathcal{J}^{(p)}:\mathcal{P}(\mathbb{R}^{d})\to[0,\,+\infty]$,
$-1\le p\le\mathfrak{q}$.
\end{thm}

\begin{rem}
As noted above, the functional $\mc J^{(p)}$ represents the level two
large deviations rate functional of the Markov chain
$\{{\bf y}^{(p)}(t)\}_{t\ge0}$ which describes the evolution of the
diffusion $\bm{x}_{\epsilon}(\cdot)$ among the wells in the time-scale
$\theta^{(p)}_\epsilon$.  According to \cite[Corollary 5.3]{KL25}, it
is possible to recover the generator of a reversible, finite-state,
continuous time Markov chain from its level two large deviations rate
functional. Therefore, the large deviations rate functional
$\mc I_\epsilon(\cdot)$ encapsulates not only the large deviations
rate functional of $\{{\bf y}^{(p)}(t)\}_{t\ge0}$, but also its
generator.
\end{rem}

From now on in this article, we always assume that conditions
\eqref{e: growth}, \eqref{e_L2}, and Assumption \ref{assu: hetero} are
satisfied.

\subsection{Example}
\begin{figure}
	\begin{tikzpicture}[scale=1.5]
		\node at (-0.2,3.3) {$U$};
	
	\draw[thick] (-0.5,3.5) sin (0,0);
	\draw[thick] (0,0) cos (0.25,0.5); \draw[thick] (0.25,0.5) sin (0.5,1);
	\draw[thick] (0.5,1) cos (0.75,0.5); \draw[thick] (0.75,0.5) sin (1,0);
	\draw[thick] (1,0) cos (1.5,1); \draw[thick] (1.5,1) sin (2,2);
	\draw[thick] (2,2) cos (2.25,1.5); \draw[thick] (2.25,1.5) sin (2.5,1);
	\draw[thick] (2.5,1) cos (2.75,1.5); \draw[thick] (2.75,1.5) sin (3,2);
	\draw[thick] (3,2) cos (3.5,1); \draw[thick] (3.5,1) sin (4,0);
	\draw[thick] (4,0) cos (5,3.5);
	
	\filldraw (0,0) circle (2pt);
	\node[below] at (0,0) {$\bm{m}_{1}$};
	\filldraw (1,0) circle (2pt);
	\node[below] at (1,0) {$\bm{m}_{2}$};
	\filldraw (2.5,1) circle (2pt);
	\node[below] at (2.5,1) {$\bm{m}_{3}$};
	\filldraw (4,0) circle (2pt);
	\node[below] at (4,0) {$\bm{m}_{4}$};
	
	\draw[latex-latex,red] (0.5,0)--(0.5,1);
	\draw[latex-latex,red] (3,0)--(3,2);
	\draw[latex-latex,red] (2,1)--(2,2);
	
	\node at (1,1.5) {$d^{(1)}$};
	\draw [-latex, dash pattern=on 1.5pt off 1pt] (1,1.2) to (0.6,0.5);
	\draw [-latex, dash pattern=on 1.5pt off 1pt] (1.3,1.5) to (1.9,1.5);
	\node at (2.2,2.5) {$d^{(2)}$};
	\draw [-latex, dash pattern=on 1.5pt off 1pt] (2.3,2.3) to (2.9,1.5);
	
	\draw [dashed] (-0.5,0) to (5,0);
	\draw [dashed] (-0.5,2) to (5,2);
	\end{tikzpicture}
	
	\caption{Graph of the potential $U$ with four local minima}
	\label{fig_ex}
\end{figure}
\subsubsection{Tree structure}
In this subsection, we present an example illustrating the tree structure and the $\Gamma$-expansion.

Let $U:\mathbb{R} \to \mathbb{R}$ be a smooth function with four local minima $\bm{m}_1,\bm{m}_2,\bm{m}_3,\bm{m}_4$ and three local maxima $\bm{\sigma}_1,\bm{\sigma}_2,\bm{\sigma}_3$, ordered as
\[
\bm{m}_1<\bm{\sigma}_1<\bm{m}_2<\bm{\sigma}_2<\bm{m}_3<\bm{\sigma}_3<\bm{m}_4 \,.
\]
Assume that these are the only critical points of $U$. Then,
\[
\mathcal{C}_0 = \{\bm{m}_1,\,\bm{\sigma}_1,\,\bm{m}_2,\,\bm{\sigma}_2,\,\bm{m}_3,\,\bm{\sigma}_3,\,\bm{m}_4\} \quad \text{and} \quad \mathcal{M}_0 = \{\bm{m}_1,\,\bm{m}_2,\,\bm{m}_3,\,\bm{m}_4\}\,.
\]
See Figure \ref{fig_ex} for the graph of $U$.
For the energy landscape shown in the figure, there are two depths  $d^{(1)}<d^{(2)}$.
At the first scale,
\[
	\mathscr{S}^{(1)} = \mathscr{V}^{(1)} = \left\{\{\bm{m}_1\},\{\bm{m}_2\},\{\bm{m}_3\},\{\bm{m}_4\}\right\} \quad \text{and} \quad \mathscr{N}^{(1)} = \varnothing \,.
\]
The first depth $d^{(1)}$ is the smallest energy barrier, attained, in particular, by the barriers separating $\bm{m}_1$ and $\bm{m}_2$.
On the first metastable time-scale $e^{d^{(1)}/\epsilon}$, the first limiting Markov chain $\widehat{\mathbf{y}}^{(1)}=\mathbf{y}^{(1)}$ has the following possible transitions:
\[
	\{\bm{m}_1\} \longleftrightarrow \{\bm{m}_2\} \longleftarrow \{\bm{m}_3\} \longrightarrow \{\bm{m}_4\}\,.
\]
In particular, $\mathbf{y}^{(1)}$ can jump from $\bm{m}_3$ to $\bm{m}_2$ or $\bm{m}_4$, but cannot return to $\bm{m}_3$. It can also jump from $\bm{m}_1$ to $\bm{m}_2$ and vice versa, whereas $\bm{m}_4$ is an absorbing state.
Here and below, we identify each singleton $\{\bm{m}\}$ with $\bm{m}$ when referring to the states of $\mathbf{y}^{(1)}$.
Thus, it has two closed irreducible classes and one transient state:
\[
	\mathscr{R}^{(1)}_1 = \left\{ \{\bm{m}_1\}, \{\bm{m}_2\} \right\}, \quad \mathscr{R}^{(1)}_2 = \left\{ \{\bm{m}_4\} \right\}\quad \text{and} \quad\mathscr{T}^{(1)}=\left\{\{\bm{m}_3\}\right\} \,.
\]

Since there is more than one closed irreducible class, we proceed to the next step.
Merging the states within each of the two closed irreducible classes gives the two elements of $\mathscr{V}^{(2)}$, while the transient state becomes negligible:
\[
	\mathscr{V}^{(2)}=\left\{\{\bm{m}_1,\bm{m}_2\},\,\{\bm{m}_4\}\right\}, \quad \mathscr{N}^{(2)}=\left\{\{\bm{m}_3\}\right\}\quad \text{and} \quad \mathscr{S}^{(2)}=\mathscr{V}^{(2)}\cup\mathscr{N}^{(2)}\,.
\]
The second depth $d^{(2)}$ is the energy barrier separating $\{\bm{m}_1,\bm{m}_2\}$ and $\{\bm{m}_4\}$.
At the second metastable time-scale $e^{d^{(2)}/\epsilon}$, the Markov chain $\widehat{\mathbf{y}}^{(2)}$ has the following possible transitions:
\[
	\{\bm{m}_1,\bm{m}_2\} \longleftrightarrow \{\bm{m}_3\} \longleftrightarrow \{\bm{m}_4\}\,.
\]
Its trace process $\mathbf{y}^{(2)}$ on $\mathscr{V}^{(2)}$ therefore has the possible transitions
\[
	\{\bm{m}_1,\bm{m}_2\} \longleftrightarrow \{\bm{m}_4\}\,.
\]
Thus, at the last time-scale, $\mathbf{y}^{(2)}$ has a unique closed irreducible class,
\[
	\mathscr{R}_1^{(2)} = \left\{\{\bm{m}_1,\bm{m}_2\},\{\bm{m}_4\}\right\}\,,
\]
and the construction terminates.

\subsubsection{$\Gamma$-expansion}
Let $\mathfrak{L}^{(1)}$ and $\mathfrak{L}^{(2)}$ denote the infinitesimal generators of $\mathbf{y}^{(1)}$ and $\mathbf{y}^{(2)}$, respectively.
For $p=1,2$, the corresponding level two large deviation rate functionals $\mathfrak{J}^{(p)}:\mathcal{P}(\mathscr{V}^{(p)}) \to [0,\infty]$ are defined by
\[
	\mathfrak{J}^{(p)}(\omega) = \sup_{\mathbf{u}>0}\left\{-\sum_{\mathcal{M} \in \mathscr{V}^{(p)}} \omega(\mathcal{M}) \frac{\mathfrak{L}^{(p)}\mathbf{u}(\mathcal{M})}{\mathbf{u}(\mathcal{M})}\right\}\,.
\]
For $\mu\in\mathcal{P}(\mathbb{R})$, define the four rate functionals $\mathcal{J}^{(p)}$, $p=-1,0,1,2$, by
\[
\begin{aligned}
	\mathcal{J}^{(-1)} (\mu) & = \frac{1}{4}\int_{\mathbb{R}} |U'|^2 d\mu\,,\\
	\mathcal{J}^{(0)} (\mu) & = \begin{cases}
		\sum_{\bm{x}\in\mathcal{C}_{0}}\omega(\bm{x})\,\zeta(\bm{x}) & \text{if}\;\; \mu=\sum_{\bm{x}\in\mathcal{C}_{0}}\omega(\bm{x})\delta_{\bm{x}}\ ,\
		\omega\in\mathcal{P}(\mathcal{C}_0)\,,\\
		\infty & \text{otherwise}\,,
	\end{cases}\\
	\mathcal{J}^{(1)} (\mu) & = \begin{cases}
		\mathfrak{J}^{(1)}(\omega) &
		\text{if}\;\; \mu=\sum_{\bm{m}\in\mathcal{M}_0}
		\omega(\bm{m})\pi_{\bm{m}}\ ,\
		\omega\in\mathcal{P}(\mathcal{M}_0)\,,\\
		\infty & \text{otherwise}\,,
	\end{cases}\\
	\mathcal{J}^{(2)} (\mu) & = \begin{cases}
		\mathfrak{J}^{(2)}(\omega) &
		\text{if}\;\; \mu=\sum_{\mathcal{M}\in\mathscr{V}^{(2)}}
		\omega(\mathcal{M})\pi_{\mathcal{M}}\ ,\
		\omega\in\mathcal{P}(\mathscr{V}^{(2)})\,,\\
		\infty & \text{otherwise}\,.
	\end{cases}
\end{aligned}
\]
Here, $\zeta:\mathcal{C}_{0} \to \mathbb{R}$ is defined in \eqref{e_zeta}, and $\pi_{\bm{m}}=\delta_{\bm{m}}$.
Under the singleton identification above, $\mathscr{V}^{(1)}=\mathcal{M}_0$.

Let $\mathcal{I}_{\epsilon}$ denote the large deviation rate functional of $\bm{x}_\epsilon$ defined in \eqref{e_def_I}. The four time-scales are
\[
	\theta_\epsilon^{(-1)} = \epsilon\ \prec \ \theta_\epsilon^{(0)} = 1\ \prec \ \theta_\epsilon^{(1)} = e^{d^{(1)}/\epsilon}\ \prec \ \theta_\epsilon^{(2)} = e^{d^{(2)}/\epsilon}\,.
\]
At the first time-scale, $(\theta_\epsilon^{(-1)}\mathcal{I}_\epsilon)_{\epsilon>0}$ $\Gamma$-converges to $\mathcal{J}^{(-1)}$.
By the proof of Corollary \ref{c43}, the $0$-level set of $\mathcal{J}^{(-1)}$ consists of the probability measures concentrated on $\mathcal{C}_0$.

At the second time-scale, $(\theta_\epsilon^{(0)}\mathcal{I}_\epsilon)_{\epsilon>0}$ $\Gamma$-converges to $\mathcal{J}^{(0)}$.
By definition, $\mathcal{J}^{(0)}(\mu)<\infty$ if and only if $\mu$ is concentrated on $\mathcal{C}_0$, or equivalently, $\mu$ belongs to the $0$-level set of $\mathcal{J}^{(-1)}$.
Moreover, the $0$-level set of $\mathcal{J}^{(0)}$ consists of the probability measures concentrated on $\mathcal{M}_0$.

At the first metastable time-scale $\theta_\epsilon^{(1)}$, $(\theta_\epsilon^{(1)}\mathcal{I}_\epsilon)_{\epsilon>0}$ $\Gamma$-converges to $\mathcal{J}^{(1)}$.
By the definition of $\mathcal{J}^{(1)}$ and Lemma \ref{l_MC_DV_finite}, $\mathcal{J}^{(1)}(\mu)<\infty$ if and only if $\mu$ is concentrated on $\mathcal{M}_0$, or equivalently, $\mu$ belongs to the $0$-level set of $\mathcal{J}^{(0)}$.

By the definition of the Markov chain $\mathbf{y}^{(1)}$ in Section \ref{sec_tree_rig}, the jump rates between $\bm{m}_1$ and $\bm{m}_2$ are
\[
r^{(1)}(\bm{m}_1,\,\bm{m}_2) = \frac{\omega(\bm{\sigma}_1)}{\nu(\bm{m}_1)} \quad \text{and}
\quad
r^{(1)}(\bm{m}_2,\,\bm{m}_1) = \frac{\omega(\bm{\sigma}_1)}{\nu(\bm{m}_2)} \,.
\]
Since $\bm{m}_1$ and $\bm{m}_2$ form a closed irreducible class of $\mathbf{y}^{(1)}$, the probability measure
\[
\pi_{\{\bm{m}_1,\,\bm{m}_2\}} = \frac{\nu(\bm{m}_1)}{\nu(\bm{m}_1)+\nu(\bm{m}_2)}\delta_{\bm{m}_1} + \frac{\nu(\bm{m}_2)}{\nu(\bm{m}_1)+\nu(\bm{m}_2)}\delta_{\bm{m}_2}
\]
is a stationary distribution of $\mathbf{y}^{(1)}$.
Moreover, since $\bm{m}_4$ is an absorbing state, $\pi_{\bm{m}_4}=\delta_{\bm{m}_4}$ is another stationary distribution.
All stationary distributions of $\mathbf{y}^{(1)}$ are convex combinations of $\pi_{\{\bm{m}_1,\,\bm{m}_2\}}$ and $\pi_{\bm{m}_4}$. Hence, by Lemma \ref{l_MC_DV_irred}, the $0$-level set of $\mathcal{J}^{(1)}$ consists precisely of these convex combinations.

Finally, at the last time-scale $\theta_\epsilon^{(2)}$, $(\theta_\epsilon^{(2)}\mathcal{I}_\epsilon)_{\epsilon>0}$ $\Gamma$-converges to $\mathcal{J}^{(2)}$.
By the definition of $\mathcal{J}^{(2)}$ and Lemma \ref{l_MC_DV_finite}, $\mathcal{J}^{(2)}(\mu)<\infty$ if and only if $\mu$ is a convex combination of $\pi_{\{\bm{m}_1,\,\bm{m}_2\}}$ and $\pi_{\bm{m}_4}$, or equivalently, $\mu$ belongs to the $0$-level set of $\mathcal{J}^{(1)}$.
By construction, the second Markov chain $\mathbf{y}^{(2)}$ has a unique closed irreducible class $\mathscr{R}_1^{(2)}$ and hence a unique stationary distribution.
Therefore, by Lemma \ref{l_MC_DV_irred}, the $0$-level set of $\mathcal{J}^{(2)}$ is a singleton, and this is the last term of the $\Gamma$-expansion.

\subsection{Outline of the article }

We prove Theorem \ref{t_main} in Section \ref{sec_pf_t_main}, assuming
Proposition \ref{p_Gamma} on the $\Gamma$-convergence. Sections
\ref{sec_pre} and \ref{sec_meta} establish $\Gamma$-convergence for
the pre-metastable ($p=-1,\,0$) and metastable time scales
($p\in\llbracket1,\,\mathfrak{q}\rrbracket$), respectively.  The proof
of the $\Gamma-\limsup$ inequality in the metastable time scales
relies on constructing a family of density functions of probability
measures converging to the target measure. This construction, stated
in Proposition \ref{p_test}, is carried out in Section
\ref{sec_pf_equi_pot}.  For this purpose, we recall several notions
from \cite{LLS-1st,LLS-2nd} and provide the rigorous construction of
the tree structure in Section \ref{sec_tree_rig}.  Finally, Section
\ref{sec_pf_last} contains the proofs of Propositions \ref{p_rev} and
\ref{p_cap}, which involve technical arguments.

\section{\label{sec_pf_t_main}Proof of Theorem \ref{t_main}}

In this section, we prove Theorem \ref{t_main} assuming Proposition \ref{p_Gamma} below together with some general properties of the Donsker--Varadhan rate functionals, recalled in Appendix \ref{app_DV}.
\begin{prop}
\label{p_Gamma} We have that
\begin{enumerate}
\item For any sequence $(\varrho_{\epsilon})_{\epsilon>0}$ of positive
numbers satisfying $\varrho_{\epsilon}\prec\epsilon$, $\varrho_{\epsilon}\mathcal{I}_{\epsilon}$
$\Gamma$-converges to $0$ as $\epsilon\to0$.
\item $\epsilon\mathcal{I}_{\epsilon}$ $\Gamma$-converges to $\mathcal{J}^{(-1)}$ as $\epsilon\to0$.
\item $\mathcal{I}_{\epsilon}$ $\Gamma$-converges to $\mathcal{J}^{(0)}$ as $\epsilon\to0$.
\item For $p\in\llbracket1,\,\mathfrak{q}\rrbracket$, $\theta_{\epsilon}^{(p)}\mathcal{I}_{\epsilon}$
$\Gamma$-converges to $\mathcal{J}^{(p)}$ as $\epsilon\to0$.
\end{enumerate}
\end{prop}

The proof is presented in Sections \ref{sec_pre} and \ref{sec_meta}.

The following lemma shows that $\mathcal{J}^{(p)}$ is finite precisely on convex combinations of the measures $\pi_{\mathcal{M}}$, and its zero level set corresponds to the convex combinations of the next level. This result plays a key role in establishing Proposition \ref{p_Gamma}
and hence Theorem \ref{t_main}.
\begin{lem}
\label{l42}We have that
\begin{enumerate}
\item Fix $p\in\llbracket1,\,\mathfrak{q}\rrbracket$ and $\mu\in\mathcal{P}(\mathbb{R}^{d})$.
Then, $\mathcal{J}^{(p)}(\mu)<\infty$ if and only if $\mu=\sum_{\mathcal{M}\in\mathscr{V}^{(p)}}\omega(\mathcal{M})\,\pi_{\mathcal{M}}$
for some $\omega\in\mathcal{P}(\mathscr{V}^{(p)})$.
\item Fix $p\in\llbracket1,\,\mathfrak{q}-1\rrbracket$ and $\mu\in\mathcal{P}(\mathbb{R}^{d})$.
Then, $\mathcal{J}^{(p)}(\mu)=0$ if and only if $\mu=\sum_{\mathcal{M}\in\mathscr{V}^{(p+1)}}\omega(\mathcal{M})\,\pi_{\mathcal{M}}$
for some $\omega\in\mathcal{P}(\mathscr{V}^{(p+1)})$.
\end{enumerate}
\end{lem}

\begin{proof}
By Lemma \ref{l_MC_DV_finite}, $\mathfrak{J}^{(p)}(\omega)<\infty$
for all $p\in\llbracket1,\,\mathfrak{q}\rrbracket$ and $\omega\in\mathcal{P}(\mathscr{V}^{(p)})$.
Thus, the first assertion follows directly from the definition
\eqref{e_J^p} of $\mathcal{J}^{(p)}$.

We now prove the second assertion.
Let $\mu\in\mathcal{P}(\mathbb{R}^{d})$ satisfy $\mathcal{J}^{(p)}(\mu)=0$.
By definition \eqref{e_J^p} of $\mathcal{J}^{(p)}$,
\[
\mu=\sum_{\mathcal{M}\in\mathscr{V}^{(p)}}\omega(\mathcal{M})\,\pi_{\mathcal{M}} \, ,
\]
for some $\omega\in\mathcal{P}(\mathscr{V}^{(p)})$ such that $\mathfrak{J}^{(p)}(\omega)=0$.
By Lemma \ref{l_MC_DV_irred}, 
\[
\omega=\sum_{i\in\llbracket1,\,\mathfrak{n}_{p}\rrbracket}a_{i}\nu_{i}^{(p)} \,,
\]
for some coefficients $(a_{i})_{i\in\llbracket1,\,\mathfrak{n}_{p}\rrbracket}$
such that $a_{i}\ge0$ and $\sum_{i\in\llbracket1,\,\mathfrak{n}_{p}\rrbracket}a_{i}=1$.
Therefore,
\[
\begin{aligned}\mu & =\sum_{\mathcal{M}\in\mathscr{V}^{(p)}}\omega(\mathcal{M})\,\pi_{\mathcal{M}}\\
 & =\sum_{\mathcal{M}\in\mathscr{V}^{(p)}}\sum_{i\in\llbracket1,\,\mathfrak{n}_{p}\rrbracket}a_{i}\nu_{i}^{(p)}(\mathcal{M})\,\pi_{\mathcal{M}}\\
 & =\sum_{i\in\llbracket1,\,\mathfrak{n}_{p}\rrbracket}a_{i}\sum_{\mathcal{M}\in\mathscr{R}_{i}^{(p)}}\nu_{i}^{(p)}(\mathcal{M})\,\pi_{\mathcal{M}}\,,
\end{aligned}
\]
where the last equality holds since the support of $\nu_{i}^{(p)}$
is $\mathscr{R}_{i}^{(p)}$. By the definition \eqref{e_nu_i} of $\nu_{i}^{(p)}$,
the last term is equal to
\[
\begin{aligned} & \sum_{i\in\llbracket1,\,\mathfrak{n}_{p}\rrbracket}a_{i}\sum_{\mathcal{M}\in\mathscr{R}_{i}^{(p)}}\frac{\nu(\mathcal{M})}{\nu(\mathcal{M}_{i}^{(p+1)})}\sum_{\bm{m}\in\mathcal{M}}\frac{\nu(\bm{m})}{\nu(\mathcal{M})}\delta_{\bm{m}}\\
 & =\sum_{i\in\llbracket1,\,\mathfrak{n}_{p}\rrbracket}a_{i}\sum_{\bm{m}\in\mathcal{M}_{i}^{(p+1)}}\frac{\nu(\bm{m})}{\nu(\mathcal{M}_{i}^{(p+1)})}\delta_{\bm{m}}\\
 & =\sum_{i\in\llbracket1,\,\mathfrak{n}_{p}\rrbracket}a_{i}\pi_{\mathcal{M}_{i}^{(p+1)}}\,.
\end{aligned}
\]
In other words,
\[
\mu=\sum_{\mathcal{M}\in\mathscr{V}^{(p+1)}}\alpha(\mathcal{M})\,\pi_{\mathcal{M}}
\]
where $\alpha(\mathcal{M}_{i}^{(p+1)})=a_{i}$ for each $i\in\llbracket1,\,\mathfrak{n}_{p}\rrbracket$.

Conversely, suppose $\mu=\sum_{\mathcal{M}\in\mathscr{V}^{(p+1)}}\omega(\mathcal{M})\,\pi_{\mathcal{M}}$
for some $\omega\in\mathcal{P}(\mathscr{V}^{(p+1)})$. By \eqref{e_M_p+1}, \eqref{e_pi}, and \eqref{eq: R^p_M},
\[
\begin{aligned}\mu & =\sum_{\mathcal{M}\in\mathscr{V}^{(p+1)}}\omega(\mathcal{M})\sum_{\mathcal{M}'\in\mathscr{R}^{(p)}(\mathcal{M})}\frac{\nu(\mathcal{M}')}{\nu(\mathcal{M})}\,\pi_{\mathcal{M}'}\\
 & =\sum_{i\in\llbracket1,\,\mathfrak{n}_{p}\rrbracket}\omega(\mathcal{M}_{i}^{(p+1)})\sum_{\mathcal{M}'\in\mathscr{R}_{i}^{(p)}}\frac{\nu(\mathcal{M}')}{\nu(\mathcal{M}_{i}^{(p+1)})}\,\pi_{\mathcal{M}'}\\
 & =\sum_{i\in\llbracket1,\,\mathfrak{n}_{p}\rrbracket}\sum_{\mathcal{M}'\in\mathscr{R}_{i}^{(p)}}\omega(\mathcal{M}_{i}^{(p+1)})\,\nu_{i}^{(p)}(\mathcal{M}')\,\pi_{\mathcal{M}'}\,.
\end{aligned}
\]
Therefore,
\[
\mu=\sum_{\mathcal{M}'\in\mathscr{V}^{(p)}}\alpha(\mathcal{M}')\,\pi_{\mathcal{M}'} \, ,
\]
where
\[
\alpha=\sum_{i\in\llbracket1,\,\mathfrak{n}_{p}\rrbracket}\omega(\mathcal{M}_{i}^{(p+1)})\,\nu_{i}^{(p)}\in\mathcal{P}(\mathscr{V}^{(p)}) \,.
\]
By Lemma \ref{l_MC_DV_irred}, we conclude $\mathcal{J}^{(p)}(\mu)=\mathfrak{J}^{(p)}(\alpha)=0$.
\end{proof}

The following corollary proves the third condition of the definition of $\Gamma$-expansion.

\begin{cor}
\label{c43}For $p\in\llbracket-1,\,\mathfrak{q}-1\rrbracket$,
\[
\mathcal{J}^{(p)}(\mu)=0\ \ \text{if and only if}\ \ \mathcal{J}^{(p+1)}(\mu)<\infty\,.
\]
\end{cor}

\begin{proof}
For $p=-1$, note that $\mathcal{J}^{(-1)}(\mu)=0$ if $\mu=\sum_{\bm{c}\in\mathcal{C}_{0}}a_{\bm{c}}\delta_{\bm{c}}$
for some $(a_{\bm{c}})_{\bm{c}\in\mathcal{C}_{0}}$ such that $a_{\bm{c}}\ge0$
and $\sum_{\bm{c}\in\mathcal{C}_{0}}a_{\bm{c}}=1$.
We prove the only if part.
Suppose that $\mathcal{J}^{(-1)}(\mu)=0$, i.e.,
\[
\int_{\mathbb{R}^d}|\nabla U|^2 d\mu=0 \,.
\]
Let $\mathcal{A}_n := \{\bm{x}\in\mathbb{R}^d:|\nabla U(\bm{x})|^2>1/n\}$.
Then, the preceding equality implies that
\[
0=\int_{\mathbb{R}^d}|\nabla U|^2 d\mu \ge \frac{1}{n}\mu(A_n) \,,
\]
so that $\mu(A_n)=0$ for all $n \ge 1$.
Since $\mathcal{A}_n\uparrow\mathcal{C}_0^c$ as $n\to\infty$, we obtain $\mu(\mathbb{R}^d\setminus\mathcal{C}_0)=0$.
Therefore,  $\mathcal{J}^{(-1)}(\mu)=0$ is a necessary and sufficient condition for $\mathcal{J}^{(0)}(\mu)<\infty$.

For $p=0$, observe that $\mathcal{J}^{(0)}(\mu)=0$ if and only if $\mu$
is supported on $\mathcal{M}_{0}$. Since the state space of the
first limiting Markov chain $\{{\bf y}^{(1)}(t)\}_{t\ge0}$ is $\mathscr{V}^{(1)}=\mathcal{M}_{0}$,
this is a necessary and sufficient condition for $\mathcal{J}^{(1)}(\mu)<\infty$
by Lemma \ref{l42}-(1).

For $p\ge1$, let $\mu\in\mathcal{P}(\mathbb{R}^{d})$. By Lemma
\ref{l42}-(2), $\mathcal{J}^{(p)}(\mu)=0$ if and only if
$\mu=\sum_{\mathcal{M}\in\mathscr{V}^{(p+1)}}\omega(\mathcal{M})\,\pi_{\mathcal{M}}$
for some $\omega\in\mathcal{P}(\mathscr{V}^{(p+1)})$. By Lemma
\ref{l42}-(1), this is equivalent to
$\mathcal{J}^{(p+1)}(\mu)<\infty$.
\end{proof}
We are now ready to prove Theorem \ref{t_main}.
\begin{proof}[Proof of Theorem \ref{t_main}]
The first condition of Definition \ref{def_Gamma_exp}
follows immediately from the definitions of time scales. The second and fourth conditions
are direct consequences of Proposition \ref{p_Gamma}. The third
condition is exactly Corollary \ref{c43}. For the last condition, suppose $\mu\in\mathcal{P}(\mathbb{R}^{d})$
satisfies $\mathcal{J}^{(\mathfrak{q})}(\mu)=0$. By definition
\eqref{e_J^p} of $\mathcal{J}^{(\mathfrak{q})}$, $\mu=\sum_{\mathcal{M}\in\mathscr{V}^{(\mathfrak{q})}}\omega(\mathcal{M})\,\pi_{\mathcal{M}}$
for some $\omega\in\mathcal{P}(\mathscr{V}^{(\mathfrak{q})})$ such
that $\mathfrak{J}^{(\mathfrak{q})}(\omega)=0$. By Lemma \ref{l_MC_DV_irred},
$\omega$ must be a stationary distribution of the chain $\{{\bf y}^{(\mathfrak{q})}(t)\}_{t\ge0}$.
Since this chain has a unique closed irreducible
class, it has a unique stationary distribution. Hence, there exists exactly
one $\mu\in\mathcal{P}(\mathbb{R}^{d})$ satisfying $\mathcal{J}^{(\mathfrak{q})}(\mu)=0$.
\end{proof}

\section{\label{sec_pre}Pre-metastable scale}

In this section, we prove the $\Gamma$-convergence of $\theta_{\epsilon}^{(p)}\mathcal{I}_{\epsilon}$ as $\epsilon\to0$
for $p=-1,\,0$.

We begin with a lemma showing that certain functions belong to the domain $D(\mathscr{L}_{\epsilon})$
of the infinitesimal generator.
\begin{lem}
\label{l_gen}Constant functions and $C_{c}^{2}$ functions belong to $D(\mathscr{L}_{\epsilon})$. Moreover, for every $a>2$,
\[
e^{U/(a\epsilon)}\in D(\mathcal L_\epsilon)
\qquad
\text{whenever}\quad
0<\epsilon<\frac{a-2}{a}\epsilon_0.
\]
\end{lem}

\begin{proof}
By Proposition \ref{p_gen}-(2), constant functions and $C_{c}^{2}$
functions lie in $D(\mathscr{L}_{\epsilon})$.

Now fix $a>2$. It follows from \eqref{e_U_L1} that
$e^{U/(a\epsilon)}\in L^{2}(d\pi_{\epsilon})$ for all $\epsilon>0$.
Recall the definition of the differential operator
$\widetilde{\ms L_\epsilon}$ introduced in display
\eqref{e_generator}. By \eqref{e_L2},
\[
\widetilde{\mathscr{L}_{\epsilon}}
\left(e^{U/(a\epsilon)}\right)=e^{U/(a\epsilon)}\left(\frac{1-a}{a^{2}\epsilon}|\nabla U|^{2}+\frac{1}{a}\Delta U\right)\in L^{2}(d\pi_{\epsilon}) \, , \ 
\text{for } \epsilon\in\left(0,\,\frac{a-2}{a}\epsilon_{0}\right)\,.
\]
Thus by Proposition \ref{p_gen}-(2),
$e^{U/(a\epsilon)}\in D(\mathscr{L}_{\epsilon})$ for $\epsilon\in(0,\,((a-2)/a)\epsilon_0)$.
\end{proof}

\subsection{First pre-metastable scale}

We first establish the $\Gamma$-convergence at the time scale $\theta_{\epsilon}^{(-1)}=\epsilon$.

\subsubsection{$\Gamma-\liminf$}
\begin{proof}[Proof of $\Gamma-\liminf$ for Proposition \ref{p_Gamma}-(2)]
Fix $\mu\in\mathcal{P}(\mathbb{R}^{d})$ and let
$(\mu_{\epsilon})_{\epsilon>0}$ be a sequence in
$\mathcal{P}(\mathbb{R}^{d})$ such that $\mu_{\epsilon}\to\mu$ weakly.
For $f\in C^{2}(\mathbb{R}^{d})$, a direct computation yields
\begin{equation*}
\begin{aligned}\nabla e^{f/\epsilon} & =\frac{1}{\epsilon}e^{f/\epsilon}\nabla f\,,\\
\Delta e^{f/\epsilon} & =e^{f/\epsilon}(\frac{1}{\epsilon^{2}}|\nabla f|^{2}+\frac{1}{\epsilon}\Delta f)\,,
\end{aligned}
\end{equation*}
so that
\begin{equation}
\mathscr{L}_{\epsilon}e^{f/\epsilon}=e^{f/\epsilon}(-\frac{1}{\epsilon}\nabla U\cdot\nabla f+\frac{1}{\epsilon}|\nabla f|^{2}+\Delta f)\,.
\label{e_e^f}
\end{equation}

Let 
\[
D_{+}(\mathscr{L}_{\epsilon}):=\{f\in D(\mathscr{L}_{\epsilon}):f>0\}\,.
\]
For $f\in C_{c}^{2}(\mathbb{R}^{d})$, note that $e^{f/\epsilon}-1\in C_{c}^{2}(\mathbb{R}^{d})\subset D(\mathscr{L}_{\epsilon})$. Moreover, since $1\in D(\mathscr{L}_{\epsilon})$ by Lemma
\ref{l_gen}, $e^{f/\epsilon}\in D_{+}(\mathscr{L}_{\epsilon})$.
Therefore, since $\mu_{\epsilon}\to\mu$, for all $f\in C_{c}^{2}(\mathbb{R}^{d})$,
\[
\begin{aligned}\liminf_{\epsilon\to0}\epsilon\mathcal{I}_{\epsilon}(\mu_{\epsilon}) & =\liminf_{\epsilon\to0}\sup_{u\in D_{+}(\mathscr{L}_{\epsilon})}-\epsilon\int_{\mathbb{R}^{d}}\frac{\mathscr{L}_{\epsilon}u}{u}d\mu_{\epsilon}\\
 & \ge\liminf_{\epsilon\to0}-\epsilon\int_{\mathbb{R}^{d}}\frac{\mathscr{L}_{\epsilon}e^{f/\epsilon}}{e^{f/\epsilon}}d\mu_{\epsilon}\\
 & =\liminf_{\epsilon\to0}\int_{\mathbb{R}^{d}}\Big(\nabla U\cdot\nabla f-|\nabla f|^{2}-\epsilon\Delta f\Big)d\mu_{\epsilon}\\
 & =\int_{\mathbb{R}^{d}}\Big(\nabla U\cdot\nabla f-|\nabla f|^{2}\Big)d\mu\\
 & =\frac{1}{4}\int_{\mathbb{R}^{d}}|\nabla U|^{2}d\mu-\int_{\mathbb{R}^{d}}|\nabla f-\frac{1}{2}\nabla U|^{2}d\mu\,.
\end{aligned}
\]
Optimizing over $f\in C_{c}^{2}(\mathbb{R}^{d})$ gives
\[
\liminf_{\epsilon\to0}\epsilon\mathcal{I}_{\epsilon}(\mu_{\epsilon})\ge\frac{1}{4}\int_{\mathbb{R}^{d}}|\nabla U|^{2}d\mu=\mathcal{J}^{(-1)}(\mu)\,.
\]
\end{proof}

\subsubsection{$\Gamma-\limsup$}

We begin by constructing a sequence of measures that approximate a Dirac measure.
\begin{lem}
\label{l_41}For $\bm{x}\in\mathbb{R}^{d}$, let $\delta_{\bm{x}}$
denote the Dirac measure at $\bm{x}$. Then, there exists a sequence
$(\mu_{\epsilon}^{\bm{x}})_{\epsilon>0}$ in $\mathcal{P}(\mathbb{R}^{d})$
such that $\mu_{\epsilon}^{\bm{x}}\to\delta_{\bm{x}}$ as $\epsilon\to0$
and
\[
\lim_{\epsilon\to0}\epsilon\mathcal{I}_{\epsilon}(\mu_{\epsilon}^{\bm{x}})=\mathcal{J}^{(-1)}(\delta_{\bm{x}})\,.
\]
\end{lem}

\begin{proof}
Fix $\bm{x}\in\mathbb{R}^{d}$. Let $V_{\bm{x}}\in C^{2}(\mathbb{R}^{d})$
satisfy 
\begin{itemize}
\item $V_{\bm{x}}(\bm{x})=0$ and $V_{\bm{x}}(\bm{y})>0$ for all $\bm{y}\ne\bm{x}$.
\item There exists $a>0$ such that for all $|\bm{y}-\bm{x}|\le a$,
\[
V_{\bm{x}}(\bm{y})=|\bm{y}-\bm{x}|^{2}\,.
\]
\item For $|\bm{y}|$ large enough,
\[
V_{\bm{x}}(\bm{y})\ge|\bm{y}|^{2}+|\nabla U(\bm{y})|^{2}+|\Delta U(\bm{y})|\,.
\]
\item $\nabla\left(e^{-V_{\bm{x}}(\bm{y})}\right)\in L^{2}(d\bm{y})$.
\end{itemize}
The existence of such a function is ensured by Step 1 in the proof of
(2.13) on page 3066 of \cite{DiGM}.

Define
\[
\mu_{\epsilon}^{\bm{x}}(d\bm{y}):=\frac{1}{\int_{\mathbb{R}^{d}}e^{-V_{\bm{x}}(\bm{z})/\epsilon}d\bm{z}}e^{-V_{\bm{x}}(\bm{y})/\epsilon}d\bm{y}\,.
\]
Since $V_{\bm{x}}(\bm{y})\ge|\bm{y}|^{2}$ for $|\bm{y}|$ large enough,
$\int_{\mathbb{R}^{d}}e^{-V_{\bm{x}}(\bm{z})/\epsilon}d\bm{z}<\infty$ and $\mu_{\epsilon}^{\bm{x}}\to\delta_{\bm{x}}$ as $\epsilon\to0$.

It remains to estimate $\epsilon\mathcal{I}_{\epsilon}(\mu_{\epsilon}^{\bm{x}})$.
By definition,
\[
\frac{d\mu_{\epsilon}^{\bm{x}}}{d\pi_{\epsilon}}(\bm{y})=\frac{Z_{\epsilon}}{\int_{\mathbb{R}^{d}}e^{-V_{\bm{x}}(\bm{z})/\epsilon}d\bm{z}}e^{-[V_{\bm{x}}(\bm{y})-U(\bm{y})]/\epsilon} \, ,
\]
so that
\[
\nabla\sqrt{\frac{d\mu_{\epsilon}^{\bm{x}}}{d\pi_{\epsilon}}}=-\sqrt{\frac{Z_{\epsilon}}{\int_{\mathbb{R}^{d}}e^{-V_{\bm{x}}(\bm{z})/\epsilon}d\bm{z}}}e^{-[V_{\bm{x}}-U]/2\epsilon}\frac{1}{2\epsilon}(\nabla V_{\bm{x}}-\nabla U)\,.
\]
Therefore, by \eqref{e_Gibbs} and \eqref{e_func_max},
\[
\begin{aligned}\epsilon\mathcal{I}_{\epsilon}(\mu_{\epsilon}^{\bm{x}}) & =\epsilon^{2}\int_{\mathbb{R}^{d}}\left|\nabla\sqrt{\frac{d\mu_{\epsilon}^{\bm{x}}}{d\pi_{\epsilon}}}\right|^{2}d\pi_{\epsilon}\\
 & =\frac{Z_{\epsilon}}{4\int_{\mathbb{R}^{d}}e^{-V_{\bm{x}}(\bm{z})/\epsilon}d\bm{z}}\int_{\mathbb{R}^{d}}|\nabla V_{\bm{x}}(\bm{y})-\nabla U(\bm{y})|^{2}e^{-[V_{\bm{x}}(\bm{y})-U(\bm{y})]/\epsilon}d\pi_{\epsilon}\\
 & =\frac{1}{4\int_{\mathbb{R}^{d}}e^{-V_{\bm{x}}(\bm{z})/\epsilon}d\bm{z}}\int_{\mathbb{R}^{d}}|\nabla V_{\bm{x}}(\bm{y})-\nabla U(\bm{y})|^{2}e^{-V_{\bm{x}}(\bm{y})/\epsilon}d\bm{y}\,.
\end{aligned}
\]
Since $\bm{x}$ is the unique minimizer of $V_{\bm{x}}$, Laplace's method yields
\[
\lim_{\epsilon\to0}\epsilon\mathcal{I}_{\epsilon}(\mu_{\epsilon}^{\bm{x}})=\frac{1}{4}|\nabla U(\bm{x})|^{2}=\frac{1}{4}\int_{\mathbb{R}^{d}}|\nabla U|^{2}\,d\delta_{\bm{x}}=\mathcal{J}^{(-1)}(\delta_{\bm{x}})\,,
\]
as claimed.
\end{proof}
To apply the previous lemma, we need the following auxiliary lemma.
\begin{lem}
\label{l_42}Let $(X,\, d)$ be a metric space.
Let $g:X\to[0,\,+\infty]$ and let $(f_{\epsilon})_{\epsilon>0}$ be a family
of functions $f_{\epsilon}:X\to[0,\,+\infty]$. Let $a\in X$
and a sequence  $(x_{n})_{n\ge1}$ be such that
\begin{equation}
\lim_{n\to\infty}x_{n}=a\ ,\ \ \limsup_{n\to\infty}g(x_{n})\le g(a)\,.\label{e_l42-1}
\end{equation}
Suppose that for each $n\in\mathbb{N}$, there exists a sequence $(y_{n,\,\epsilon})_{\epsilon>0}$
in $X$ such that
\begin{equation}
\lim_{\epsilon\to0}y_{n,\,\epsilon}=x_{n}\ ,\ \ \limsup_{\epsilon\to0}f_{\epsilon}(y_{n,\,\epsilon})\le g(x_{n})\,.\label{e_l42-2}
\end{equation}
Then, there exists a sequence $(z_{\epsilon})_{\epsilon>0}$ in $X$
such that
\begin{equation}
\lim_{\epsilon\to0}z_{\epsilon}=a\ ,\ \ \limsup_{\epsilon\to0}f_{\epsilon}(z_{\epsilon})\le g(a)\,.\label{e_l42-3}
\end{equation}
\end{lem}

\begin{proof}
If $g(a)=\infty$, we can take $z_{\epsilon}=a$ for all $\epsilon>0$. Hence assume
$g(a)<\infty$. By \eqref{e_l42-1}, for each $k\in\mathbb{N}$, there
exists $N_{k}\in\mathbb{N}$ such that
\begin{equation}
n\ge N_{k}\ \Longrightarrow\ d(x_{n},\,a)\,,\ g(x_{n})-g(a)\,\le\,\frac{1}{k}\,.\label{e_pf_l42-1}
\end{equation}
By \eqref{e_l42-2} we may choose $M_{k}\in\mathbb{N}$ such that $M_{k}<M_{k+1}$,
$k\in\mathbb{N}$, and
\begin{equation}
\epsilon\le\frac{1}{M_{k}}\ \Longrightarrow \ d(y_{N_{k},\,\epsilon},\,x_{N_{k}})\,,\ f_{\epsilon}(y_{N_{k},\,\epsilon})-g(x_{N_{k}})\,\le\,\frac{1}{k}\,.\label{e_pf_l42-2}
\end{equation}
Define
\[
z_{\epsilon}:=y_{N_{k},\,\epsilon}\ ;\ \ \epsilon\in\left(\frac{1}{M_{k+1}},\,\frac{1}{M_{k}}\right]\,.
\]

We claim that the sequence $(z_{\epsilon})_{\epsilon>0}$ satisfies \eqref{e_l42-3}. 
Let $\delta>0$ and pick $k\in\mathbb{N}$ satisfying $k\ge2/\delta$.
For $\epsilon\in\left(0,\,\frac{1}{M_{k}}\right]$, since $\epsilon\in\left(\frac{1}{M_{l+1}},\,\frac{1}{M_{l}}\right]$
for some $l\ge k$, by \eqref{e_pf_l42-1}
and \eqref{e_pf_l42-2},
\[
d(z_{\epsilon},\,a)=d(y_{N_{l},\,\epsilon},\,a)\le d(y_{N_{l},\,\epsilon},\,x_{N_{l}})+d(x_{N_{l}},a)\le\frac{2}{l}\le\delta\,,
\]
which shows $\lim_{\epsilon\to0}z_{\epsilon}=a$. Similarly, by \eqref{e_pf_l42-1} and \eqref{e_pf_l42-2},
\[
f_{\epsilon}(z_{\epsilon})=f_{\epsilon}(y_{N_{l},\,\epsilon})=f_{\epsilon}(y_{N_{l},\,\epsilon})-g(x_{N_{l}})+g(x_{N_{l}})\le g(a)+\frac{2}{l}\le g(a)+\delta\,.
\]
Taking $\limsup_{\epsilon\to0}$ yields $\limsup_{\epsilon\to0}f_{\epsilon}(z_{\epsilon})\le g(a)$,
as claimed.
\end{proof}

For $r>0$ and $\bm{x}\in\mathbb{R}^{d}$, denote by $B_{r}(\bm{x})$
the closed ball with radius $r$ centered at $\bm{x}$:
\[
{\color{blue}B_{r}(\bm{x}}\mathclose{\color{blue})}:=\{\bm{y}:|\bm{x}-\bm{y}|\le r\}\,.
\]
When the center is the origin, we simply write ${\color{blue}B_{r}}:=B_r(\bm{0})$.

We are now ready to prove $\Gamma-\limsup$ of the sequence $(\epsilon\mathcal{I}_{\epsilon})_{\epsilon>0}$.

\begin{proof}[Proof of $\Gamma-\limsup$ for Proposition \ref{p_Gamma}-(2)]
Let $\mu\in\mathcal{P}(\mathbb{R}^{d})$.

\noindent\textbf{Step 1.} Dirac measure

If $\mu=\delta_{\bm{x}}$
for some $\bm{x}\in\mathbb{R}^{d}$, we can take the sequence $(\mu_{\epsilon}^{\bm{x}})_{\epsilon>0}$
introduced in Lemma \ref{l_41}.

\noindent\textbf{Step 2.} Finite convex combinations of Dirac measures

Suppose that
\[
\mu=\sum_{\bm{x}\in\mathcal{A}}a_{\bm{x}}\delta_{\bm{x}}
\]
for some finite set $\mathcal{A}\subset\mathbb{R}^{d}$ and positive
weights $a_{\bm{x}}$ such that $\sum_{\bm{x}\in\mathcal{A}}a_{\bm{x}}=1$.
Let $\mu_{\epsilon}=\sum_{\bm{x}\in\mathcal{A}}a_{\bm{x}}\mu_{\epsilon}^{\bm{x}}$.
By convexity of $\mathcal{I}_{\epsilon}$ and Lemma \ref{l_41},
\[
\limsup_{\epsilon\to0}\epsilon\mathcal{I}_{\epsilon}(\mu_{\epsilon})\le\sum_{\bm{x}\in\mathcal{A}}a_{\bm{x}}\limsup_{\epsilon\to0}\epsilon\mathcal{I}_{\epsilon}(\mu_{\epsilon}^{\bm{x}})\le\sum_{\bm{x}\in\mathcal{A}}\frac{a_{\bm{x}}}{4}|\nabla U(\bm{x})|^{2}=\frac{1}{4}\int_{\mathbb{R}^{d}}|\nabla U|^{2}d\mu\,.
\]

\noindent\textbf{Step 3.} General measures

If $\mathcal{J}^{(-1)}(\mu)=\infty$,
there is nothing to prove. Assume, therefore, that $\nabla U\in L^{2}(d\mu)$,
i.e.,
\[
\mathcal{J}^{(-1)}(\mu)=\frac{1}{4}\int_{\mathbb{R}^{d}}|\nabla U|^{2}d\mu<\infty\,.
\]
For $n\in\mathbb{N}$, let $\mu_{n}$ be the measure $\mu$ conditioned on
$B_{n}$. Clearly,
\[
\mu_{n}\to\mu\ ,\ \ \mathcal{J}^{(-1)}(\mu_{n})\to\mathcal{J}^{(-1)}(\mu)\ \text{as}\ n\to\infty\,.
\]

Since $B_{n}$ is compact, the space of finite convex combinations of Dirac
measures supported on $B_{n}$ is dense in $\mathcal{P}(B_{n})$.
Hence, there exist finite sets $\mathcal{A}(n,\,k)$ and
coefficients $a_{\bm{x}}^{n,\,k}\ge0$
such that measures
\[
\nu_{n,\,k}:=\sum_{\bm{x}\in\mathcal{A}(n,\,k)}a_{\bm{x}}^{n,\,k}\delta_{\bm{x}}\in\mathcal{P}(B_{n})
\]
satisfy $\nu_{n,\,k}\to\mu_{n}$
and $\mathcal{J}^{(-1)}(\nu_{n,\,k})\to\mathcal{J}^{(-1)}(\mu_{n})$
as $k\to\infty$.

Since the weak-$\ast$ topology on $\mathcal{P}(\mathbb{R}^{d})$ is metrizable,
the diagonal argument yields a subsequence $(\nu_{n,\,k(n)})_{n\ge1}$
such that
\[
\nu_{n,\,k(n)}\to\mu \ \text{and} \  \mathcal{J}^{(-1)}(\nu_{n,\,k(n)})\to\mathcal{J}^{(-1)}(\mu), \ 
\text{as}\ n\to\infty \, .
\]
For each $n\in\mathbb{N}$, let $(\mu_{n,\,\epsilon})_{\epsilon>0}$
be the sequence of measures constructed in Step 2 such that 
\[
\mu_{n,\,\epsilon}\to\nu_{n,\,k(n)}\ \text{as } \epsilon\to0\, , \quad
\limsup_{\epsilon\to0}\epsilon\mathcal{I}_{\epsilon}(\mu_{n,\,\epsilon})\le\frac{1}{4}\int_{\mathbb{R}^{d}}|\nabla U|^{2}d\nu_{n,\,k(n)}\,.
\]
Finally, apply Lemma \ref{l_42} with
\[
X=\mathcal{P}(\mathbb{R}^{d}) \, ,\ 
f_{\epsilon}=\epsilon\mathcal{I}_{\epsilon} \, ,\ g=\mathcal{J}^{(-1)} \, ,\ 
a=\mu \, , \ x_{n}=\nu_{n,\,k(n)} \,, \  \text{and } \  y_{n,\,\epsilon}=\mu_{n,\,\epsilon}\,.
\]
\end{proof}

The next result shows that $\theta_{\epsilon}^{(-1)}=\epsilon$ is
the first time scale in the $\Gamma$-expansion of $\mathcal{I}_{\epsilon}$. 

\begin{proof}[Proof of Proposition \ref{p_Gamma}-(1)]
It suffices to consider the $\Gamma-\limsup$. Let
$(\varrho_{\epsilon})_{\epsilon>0}$ be a sequence of positive numbers such that
\[
\lim_{\epsilon\to0}\frac{\varrho_{\epsilon}}{\epsilon}=0 \,.
\]
By Lemma \ref{l_41}, for every $\bm{x}\in\mathbb{R}^d$, 
\[
\limsup_{\epsilon\to0}\varrho_{\epsilon}\mathcal{I}_{\epsilon}(\mu_{\epsilon}^{\bm{x}})=0\,,
\]
where $\mu_{\epsilon}^{\bm{x}}$ is the measure constructed in the
proof of Lemma \ref{l_41}.

Now fix $\mu\in\mathcal{P}(\mathbb{R}^{d})$
and apply Lemma \ref{l_42} with $g=0$.
By the same argument as in the proof of the $\Gamma-\limsup$
of Proposition \ref{p_Gamma}-(2), we conclude that there exists a sequence $(\mu_{\epsilon})_{\epsilon>0}$ in $\mathcal{P}(\mathbb{R}^d)$
such that 
\[
\mu_{\epsilon}\to\mu\ \text{as}\ \epsilon\to 0 \,, \ \text{and}\ \limsup_{\epsilon\to0}\varrho_{\epsilon}\mathcal{I}_{\epsilon}(\mu_{\epsilon})=0\,.
\]
This completes the proof.
\end{proof}

\subsection{Second pre-metastable scale}

In this subsection, we prove the $\Gamma$-convergence in time scale $\theta_\epsilon^{(0)}=1$.
Recall from \eqref{e_zeta} and \eqref{e_J^0} the definitions of
$\zeta:\mathcal{C}_{0}\to\mathbb{R}$ and $\mathcal{J}^{(0)}:\mathcal{P}(\mathbb{R}^{d})\to[0,\,+\infty]$.

\subsubsection{$\Gamma-\liminf$}

Fix $R_{0}>0$ such that

\begin{equation}
|\nabla U(\bm{x})|>1\,,\ |\nabla U(\boldsymbol{x})|-2\Delta U(\boldsymbol{x})\ge0\ \ \text{for all}\ |\bm{x}|\ge R_{0}\,.\label{e_R0}
\end{equation}
The existence of such $R_{0}$ follows from the growth condition \eqref{e: growth}.
For $|\bm{x}|\ge R_{0}$ and $\epsilon\in(0,\,1)$, we distinguish two cases:

\smallskip
\noindent
$\bullet$ If $\Delta U(\bm{x})\ge0$, then since $|\nabla U(\bm{x})|>1$ and $\epsilon<\frac{4}{3}$, by the second inequality in \eqref{e_R0},
\[
\frac{2}{9}|\nabla U(\bm{x})|^{2}-\frac{\epsilon}{3}\Delta U(\bm{x}) > \frac{2}{9}|\nabla U(\bm{x})|-\frac{4}{9}\Delta U(\bm{x}) \ge 0\,.
\]

\noindent
$\bullet$ If $\Delta U(\bm{x})<0$, then
\[
\frac{2}{9}|\nabla U(\bm{x})|^{2}-\frac{\epsilon}{3}\Delta U(\bm{x})\ge\frac{2}{9}|\nabla U(\bm{x})|^{2}>0\,.
\]

\smallskip
In summary,
\begin{equation}
\frac{2}{9}|\nabla U(\bm{x})|^{2}-\frac{\epsilon}{3}\Delta U(\bm{x})>0\ ;\ \epsilon\in(0,\,1)\,,\ |\bm{x}|\ge R_{0}\,.\label{e_growth-2}
\end{equation}
By enlarging $R_{0}>0$ if necessary, we may assume $R_{0}\ge1$
and that there is no critical point $\bm{c}\in\mathcal{C}_{0}$ with
$|\bm{c}|\ge R_{0}/2$.

\medskip
For $\mathcal{A}\subset\mathbb{R}^{d}$ and $r>0$, define 
\[
{\color{blue}B_{r}(\mathcal{A}}\mathclose{\color{blue})}:=\bigcup_{\bm{x}\in\mathcal{A}}B_{r}(\bm{x})\,.
\]

By Lemma \ref{l_gen}, $e^{U/(a\epsilon)}\in D(\mathscr{L}_{\epsilon})$
for all $a>2$ and small $\epsilon>0$. Hence, by \eqref{e_e^f},
for every $\mu\in\mathcal{P}(\mathbb{R}^{d})$,
\[
\begin{aligned}\mathcal{I}_{\epsilon}(\mu) & \ge-\int_{\mathbb{R}^{d}}\frac{\mathscr{L}_{\epsilon}e^{U/(3\epsilon)}}{e^{U/(3\epsilon)}}d\mu\\
 & =\int_{\mathbb{R}^{d}}\left(\frac{2}{9\epsilon}|\nabla U|^{2}-\frac{1}{3}\Delta U\right)d\mu\,.
\end{aligned}
\]
Therefore, by \eqref{e_growth-2}, for all $R>R_{0}$, $\mu\in\mathcal{P}(\mathbb{R}^{d})$, and small $\epsilon>0$,
\begin{equation}
\mathcal{I}_{\epsilon}(\mu)\ge\int_{B_{R}}\left(\frac{2}{9\epsilon}|\nabla U|^{2}-\frac{1}{3}\Delta U\right)d\mu\,.\label{e_bound_I}
\end{equation}

The next lemma provides the key estimate needed for the proof of the $\Gamma-\liminf$
of the sequence $(\mathcal{I}_{\epsilon})_{\epsilon>0}$.
\begin{lem}
\label{l_pre2}
Fix $\mu\in\mathcal{P}(\mathbb{R}^{d})$ and let $(\mu_{\epsilon})_{\epsilon>0}$
be a sequence in $\mathcal{P}(\mathbb{R}^{d})$ such that $\mu_{\epsilon}\to\mu$
and
\begin{equation}
\liminf_{\epsilon\to0}\mathcal{I}_{\epsilon}(\mu_{\epsilon})<\infty\,.\label{e_l_pre2}
\end{equation}
Then,
\begin{equation}
\limsup_{r \to 0}\liminf_{\epsilon\to0}\sum_{\bm{c}\in\mathcal{C}_{0}}\frac{1}{\epsilon}\left[\int_{B_{2r}(\bm{c})\setminus{B_{r}(\bm{c})}}|\bm{x}-\bm{c}|^{2}d\mu_{\epsilon} + \int_{B_{r}(\bm{c})}|\bm{x}-\bm{c}|^{3}d\mu_{\epsilon} \right]=0\,.\label{e_l_pre2-1}
\end{equation}
\end{lem}

\begin{proof}
Fix $R>R_0$. Since $\mathcal C_0$ is finite and every critical point is
nondegenerate, there exist $a,\,b>0$ such that the balls
$B_{2a}(\bm{c})$, $\bm{c}\in\mathcal C_0$, are pairwise disjoint,
are contained in $B_R$, and satisfy
\[
|\nabla U(\bm{x})|^2\ge b|\bm{x}-\bm{c}|^2
\quad
\text{for all}\ \bm{x}\in B_{2a}(\bm{c}),\quad \bm{c}\in\mathcal C_0\,.
\]
Define
\[
S_\epsilon := \sum_{c\in\mathcal C_0}
\frac{1}\epsilon
\int_{B_{2a}(\bm{c})}|\bm{x}-\bm{c}|^2\,d\mu_\epsilon \,.
\]
By \eqref{e_bound_I}, for all sufficiently small $\epsilon>0$,
\[
S_\epsilon
\le
\frac{1}{b\epsilon}
\int_{B_R}|\nabla U|^2\,d\mu_\epsilon
\le
\frac{9}{2b}
\left[
\mathcal I_\epsilon(\mu_\epsilon)
+
\frac13\|\Delta U\|_{L^\infty(B_R)}
\right].
\]
Therefore, by \eqref{e_l_pre2},
\[
K:=\liminf_{\epsilon\to0}S_\epsilon<\infty.
\]

For $j\ge0$, set $\rho_j:=a2^{-j}$.
For $j\ge1$ and $\bm{c}\in\mathcal C_0$, define
\[
A_j(\bm{c}):=
B_{\rho_{j-1}}(\bm{c})\setminus B_{\rho_{j+1}}(\bm{c}),
\]
and let
\[
T_{\epsilon,j}
:=
\sum_{\bm{c}\in\mathcal C_0}
\frac1\epsilon
\int_{A_j(c)}|\bm{x}-\bm{c}|^2\,d\mu_\epsilon,
\qquad
V_{\epsilon,j}:=T_{\epsilon,j}+\rho_jS_\epsilon.
\]
For each fixed $\bm{c}$, every point belongs to at most two
of the annuli $A_j(\bm{c})$, $j\ge1$.
Moreover,
\[
\sum_{j=1}^\infty\rho_j=a.
\]
Consequently, for every $N\ge1$,
\[
\sum_{j=1}^N V_{\epsilon,j}
\le
(2+a)S_\epsilon.
\]
Set
\[
v_j:=\liminf_{\epsilon\to0}V_{\epsilon,j}.
\]
Since all terms are nonnegative,
\[
\sum_{j=1}^N v_j
\le
\liminf_{\epsilon\to0}
\sum_{j=1}^N V_{\epsilon,j}
\le
(2+a)K.
\]
Letting $N\to\infty$ yields
\[
\sum_{j=1}^\infty v_j<\infty,
\qquad
\text{and hence}\qquad
v_j\longrightarrow0.
\]

For $0<r\le\rho_1$, choose $j\ge1$ such that
\[
\rho_{j+1}<r\le\rho_j.
\]
Then,
\[
B_{2r}(\bm{c})\setminus B_{r}(\bm{c})\subset A_j(\bm{c}),
\]
and
\[
\int_{B_r(\bm{c})}|\bm{x}-\bm{c}|^3\,d\mu_\epsilon
\le
\rho_j\int_{B_{a}(\bm{c})}|\bm{x}-\bm{c}|^2\,d\mu_\epsilon.
\]
Thus,
\[
\begin{aligned}
&\sum_{\bm{c}\in\mathcal C_0}\frac1\epsilon
\left[
\int_{B_{2r}(\bm{c})\setminus B_r(\bm{c})}
|\bm{x}-\bm{c}|^2\,d\mu_\epsilon
+
\int_{B_r(\bm{c})}
|\bm{x}-\bm{c}|^3\,d\mu_\epsilon
\right]
\\
&\hspace{3cm}\le
T_{\epsilon,j}+\rho_jS_\epsilon
=
V_{\epsilon,j}.
\end{aligned}
\]
Taking $\liminf_{\epsilon\to0}$ bounds the left-hand side
by $v_j$.
As $r\downarrow0$, the corresponding index $j$ tends to
infinity, and $v_j\to0$, which completes the proof.
\end{proof}

We now prove the $\Gamma-\liminf$. For $\mathcal{A}\subset\mathbb{R}^{d}$,
denote by \textcolor{blue}{$\mathcal{X}_{\mathcal{A}}$} the indicator function
on $\mathcal{A}$.

\begin{proof}[Proof of $\Gamma-\liminf$ for Proposition \ref{p_Gamma}-(3)]
Suppose that $\mu\in\mathcal{P}(\mathbb{R}^{d})$ is not a convex combination of $\delta_{\bm{c}}$, $\bm{c}\in\mathcal{C}_{0}$.
Let $(\mu_{\epsilon})_{\epsilon>0}$ be a sequence in $\mathcal{P}(\mathbb{R}^{d})$
such that $\mu_{\epsilon}\to\mu$. By Proposition \ref{p_Gamma}-(2),
\[
\liminf_{\epsilon\to0}\epsilon\mathcal{I}_{\epsilon}(\mu_{\epsilon}) \ge \mathcal{J}^{(-1)}(\mu)>0\,,
\]
so that
\[
\liminf_{\epsilon\to0}\mathcal{I}_{\epsilon}(\mu_{\epsilon})=\infty\,.
\]

Now assume that
\begin{equation}
\mu=\sum_{\bm{c}\in\mathcal{C}_{0}}\omega(\bm{c})\delta_{\bm{c}} \quad \text{for some}\quad \omega\in\mathcal{P}(\mathcal{C}_{0})\,.
\label{e_pf_p_Gamma2-1}
\end{equation}
Let $(\mu_{\epsilon})_{\epsilon>0}$ be a sequence in $\mathcal{P}(\mathbb{R}^{d})$
such that $\mu_{\epsilon}\to\mu$. Since $\mathcal{J}^{(0)}(\mu)<\infty$,
we may assume that the condition \eqref{e_l_pre2} holds; otherwise, there is nothing to prove.
We now divide the proof into three parts.

\medskip
\noindent\textbf{Step 1.} Test function

For each $\bm{c}\in\mathcal{C}_{0}$, let ${\color{blue}\mathbb{H}^{\bm{c}}}:=\nabla^{2}U(\bm{c})$
and let $\lambda_{1}^{\bm{c}},\,\dots,\,\lambda_{d}^{\bm{c}}$ be
the eigenvalues of $\mathbb{H}^{\bm{c}}$.
Define
\[
{\color{blue}\mathbb{D}^{\bm{c}}}:={\rm diag}(\lambda_{1}^{\bm{c}},\,\dots,\,\lambda_{d}^{\bm{c}})\,,
\]
and let \textcolor{blue}{$\mathbb{U}^{\bm{c}}$} be the orthogonal matrix such that
\[
\mathbb{H}^{\bm{c}}=\mathbb{U}^{\bm{c}}\mathbb{D}^{\bm{c}}(\mathbb{U}^{\bm{c}})^{-1}\,.
\]
Let $\widetilde{\mathbb{D}^{\bm{c}}}$
be the diagonal matrix defined by
\[
\widetilde{\mathbb{D}^{\bm{c}}}:={\rm diag}(\varLambda_{1}^{\bm{c}},\,\dots,\,\Lambda_{d}^{\bm{c}})\,,
\]
where $\Lambda_{i}^{\bm{c}}:=\min\{\lambda_{i}^{\bm{c}},\,0\}$. Note
that only the negative eigenvalues are present in $\widetilde{\mathbb{D}^{\bm{c}}}$.
Define the quadratic form
\[
{\color{blue}G_{\bm{c}}(\bm{x})}:=\frac{1}{2}(\bm{x}-\bm{c})\cdot\widetilde{\mathbb{H}^{\bm{c}}}(\bm{x}-\bm{c})\,,
\]
where ${\color{blue}\widetilde{\mathbb{H}^{\bm{c}}}}:=\mathbb{U}^{\bm{c}}\widetilde{\mathbb{D}^{\bm{c}}}(\mathbb{U}^{\bm{c}})^{-1}$.
A direct computation gives
\begin{equation}
\Delta G_{\bm{c}}(\bm{c})={\rm Tr}\widetilde{\mathbb{H}^{\bm{c}}}={\rm Tr}\widetilde{\mathbb{D}^{\bm{c}}}=-\zeta(\bm{c})\,,\label{e_2nd_zeta}
\end{equation}
where $\zeta:\mathcal{C}_{0}\to\mathbb{R}$ was introduced in
\eqref{e_zeta}.

Fix $a_{0}>0$ so small that $B_{3a_{0}}(\bm{c})\cap B_{3a_{0}}(\bm{c}')=\varnothing$ whenever $\bm{c},\,\bm{c}'\in\mathcal{C}_{0}$ are distinct.
 For $r\in(0,\,a_{0})$, let $\psi_{r}\in C_{c}^{\infty}(\mathbb{R}^{d})$
be such that
\begin{equation}
\mathcal{X}_{B_{r}}\le\psi_{r} \le \mathcal{X}_{B_{2r}}\,,\ \text{and}\ \|\nabla\psi_{r}\|_{L^{\infty}(B_{2r})}\le\frac{C^{(1)}}{r}\,,\label{e_smoothcut}
\end{equation}
for some constant $C^{(1)}>0$ independent of $r>0$. Define the localized test function
\[
{\color{blue}F_{r}(\bm{x})}:=\sum_{\bm{c}\in\mathcal{C}_{0}}\psi_{r}(\bm{x}-\bm{c})G_{\bm{c}}(\bm{x})\,.
\]
The supports of the terms in this sum are pairwise disjoint.
In particular,
\[
F_r=G_{\bm{c}}\ \text{on}\ B_r(\bm{c})\,,
\quad \text{and} \quad
\Delta F_r(\bm{c})=-\zeta(\bm{c})\,.
\]

\smallskip
\noindent\textbf{Step 2.} Lower bound

Set
\[
L:=\liminf_{\epsilon\to0} \mathcal I_\epsilon(\mu_\epsilon)<\infty.
\]
Choose a decreasing sequence $\epsilon_k \to 0$ such that
\[
\mathcal I_{\epsilon_k}(\mu_{\epsilon_k})\longrightarrow L.
\]

By \eqref{e_e^f}, for all $r\in(0,\,a_{0})$ and $\epsilon>0$,
\begin{equation}
\begin{aligned}\mathcal{I}_{\epsilon}(\mu_{\epsilon}) & \ge-\int_{\mathbb{R}^{d}}e^{-F_{r}(\bm{x})/\epsilon}\mathscr{L}_{\epsilon}e^{F_{r}(\bm{x})/\epsilon}d\mu_{\epsilon}\\
 & =\frac{1}{\epsilon}\int_{\mathbb{R}^{d}}\nabla F_{r}\cdot(\nabla U-\nabla F_{r})d\mu_{\epsilon}-\int_{\mathbb{R}^{d}}\Delta F_{r}d\mu_{\epsilon}\\
 & =\sum_{\bm{c}\in\mathcal{C}_{0}}\left\{ \frac{1}{\epsilon}\int_{B_{2r}(\bm{c})}\nabla F_{r}\cdot(\nabla U-\nabla F_{r})d\mu_{\epsilon}-\int_{B_{2r}(\bm{c})}\Delta F_{r}d\mu_{\epsilon}\right\} \,.
\end{aligned}
\label{e_414}
\end{equation}

Consider first the second term. Note that $\Delta F_{r}=\Delta\psi_{r}G_{\bm{c}}+2\nabla\psi_{r}\cdot\nabla G_{\bm{c}}+\psi_{r}\Delta G_{\bm{c}}$
is continuous on $B_{2r}(\bm{c})$. Since $\mu_{\epsilon} \to \sum_{\bm{c}\in\mathcal{C}_0}\omega(\bm{c})\delta_{\bm{c}}$
and $G_{\bm{c}}(\bm{c})=\nabla G_{\bm{c}}(\bm{c})=0$, it follows from \eqref{e_2nd_zeta} that
\begin{equation}
\lim_{k\to\infty}\int_{B_{2r}(\bm{c})}\Delta F_{r}d\mu_{\epsilon_{k}}=\omega(\bm{c})\Delta G_{\bm{c}}(\bm{c})=-\omega(\bm{c})\zeta(\bm{c})\,.\label{e_415}
\end{equation}

We turn to the first term. We claim that
\begin{equation}
\limsup_{r\to0}\liminf_{k\to\infty}\sum_{\bm{c}\in\mathcal{C}_0}\frac{1}{\epsilon_{k}}\int_{B_{2r}(\bm{c})}\left|\nabla F_{r}\cdot(\nabla U-\nabla F_{r})\right|d\mu_{\epsilon_{k}}=0\,.\label{e_pf_p_Gamma2-2}
\end{equation}

For $\bm{x}\in B_{2r}(\bm{c})$, $F_{r}(\bm{x})=\psi_{r}(\bm{x}-\bm{c})G_{\bm{c}}(\bm{x})$
and $\nabla F_{r}(\bm{x})=\nabla\psi_{r}(\bm{x}-\bm{c})G_{\bm{c}}(\bm{x})+\psi_{r}(\bm{x}-\bm{c})\nabla G_{\bm{c}}(\bm{x})$.
Also, there exists $C^{(2)}>0$ such that for all $\bm{x}\in B_{2r}(\bm{c})$,
\begin{equation}
|\nabla U(\bm{x})|,\,|\nabla G_{\bm{c}}(\bm{x})|\le C^{(2)}r\,,\ |G_{\bm{c}}(\bm{x})|\le C^{(2)}r^{2}\,.\label{e_pf_p_Gamma2-3}
\end{equation}
Therefore, by \eqref{e_smoothcut}, for $\bm{x}\in B_{2r}(\bm{c})$,
\[
|\nabla F_{r}(\bm{x})|\le C^{(3)}r\,,
\]
for some $C^{(3)}>0$.

The proof splits into two regions $B_{2r}(\bm{c})\setminus B_{r}(\bm{c})$ and $B_{r}(\bm{c})$.
First, we consider the integration on $B_{2r}(\bm{c})\setminus B_{r}(\bm{c})$.
By the previous observation, for $\bm{x}\in B_{2r}(\bm{c})\setminus B_{r}(\bm{c})$,
\[
\left|\nabla F_{r}\cdot(\nabla U-\nabla F_{r})\right|\le C^{(4)}r^{2}\le C^{(4)}|\bm{x}-\bm{c}|^{2}\,,
\]
for some constant $C^{(4)}>0$ independent of $r>0$ and $\bm{c}$.

We turn to the integration on $B_{r}(\bm{c})$. For $\bm{x}\in B_{r}(\bm{c})$,
since $F_{r}(\bm{x})=G_{\bm{c}}(\bm{x})$,
\[
\nabla F_{r}\cdot(\nabla U-\nabla F_{r})=\nabla G_{\bm{c}}\cdot(\nabla U-\nabla G_{\bm{c}})\,.
\]
By the Taylor expansion, there exists $C^{(5)}>0$ such that for $\bm{x}\in B_{r}(\bm{c})$,
\[
|\nabla U(\bm{x})-\mathbb{H}^{\bm{c}}(\bm{x}-\bm{c})|\le C^{(5)}|\bm{x}-\bm{c}|^{2}\,.
\]
Hence, by the definition of $G_{\bm{c}}$ and the previous bound,
\[
\nabla G_{\bm{c}}(\bm{x})\cdot\left(\nabla U(\bm{x})-\nabla G_{\bm{c}}(\bm{x})\right)=\nabla G_{\bm{c}}(\bm{x})\cdot\left(\mathbb{H}^{\bm{c}}(\bm{x}-\bm{c})-\nabla G_{\bm{c}}(\bm{x})\right)+R_{\bm{c}}(\bm{x})
\]
where $R_{\bm{c}}$ satisfies for some $C^{(6)}>0$ independent of $r>0$ and $\bm{c}$,
\[
|R_{\bm{c}}(\bm{x})|\le C^{(6)}|\bm{x}-\bm{c}|^{3}\,.
\]
By the definitions of $G_{\bm{c}}$, $\mathbb{H}^{\bm{c}}$, and $\widetilde{\mathbb{H}^{\bm{c}}}$,
\[
\nabla G_{\bm{c}}(\bm{x})\cdot\left(\mathbb{H}^{\bm{c}}(\bm{x}-\bm{c})-\nabla G_{\bm{c}}(\bm{x})\right)=\widetilde{\mathbb{H}^{\bm{c}}}(\bm{x}-\bm{c})\cdot(\mathbb{H}^{\bm{c}}(\bm{x}-\bm{c})-\widetilde{\mathbb{H}^{\bm{c}}}(\bm{x}-\bm{c}))=0\,.
\]
Consequently, for some constant $C>0$ independent of $\epsilon$ and $r$,
\[
\begin{aligned}
& \sum_{\bm{c}\in\mathcal C_0}\frac{1}{\epsilon}\int_{B_{2r}(\bm{c})}\left|\nabla F_{r}\cdot(\nabla U-\nabla F_{r})\right|d\mu_{\epsilon} \\ 
& \qquad \le
C\sum_{\bm{c}\in\mathcal C_0}\frac1\epsilon \left[ \int_{B_{2r}(\bm{c})\setminus B_r(\bm{c})} |\bm{x}-\bm{c}|^2\,d\mu_\epsilon + \int_{B_r(\bm{c})}|\bm{x}-\bm{c}|^3\,d\mu_\epsilon \right].
\end{aligned}
\]
Applying Lemma 4.4 along the chosen sequence proves \eqref{e_pf_p_Gamma2-2}.

\smallskip
\noindent\textbf{Step 3.} Conclusion

Finally, the inequality \eqref{e_414} gives
\[
-\int_{\mathbb R^d}\Delta F_r\,d\mu_{\epsilon_{k}} - \mathcal I_{\epsilon_{k}}(\mu_{\epsilon_{k}}) \le\sum_{\bm{c}\in\mathcal{C}_0}\frac{1}{\epsilon_{k}}\int_{B_{2r}(\bm{c})}\left|\nabla F_{r}\cdot(\nabla U-\nabla F_{r})\right|d\mu_{\epsilon_{k}}.
\]
For each fixed $r\in(0,a_0)$, the left-hand side converges to
$\mathcal{J}^{(0)}(\mu)-L$. Hence,
\[
\mathcal{J}^{(0)}(\mu)-L \le \liminf_{k\to\infty}\sum_{\bm{c}\in\mathcal{C}_0}\frac{1}{\epsilon_{k}}\int_{B_{2r}(\bm{c})}\left|\nabla F_{r}\cdot(\nabla U-\nabla F_{r})\right|d\mu_{\epsilon_{k}}.
\]
Letting $r \to 0$ and using \eqref{e_pf_p_Gamma2-2} yields
\[
\mathcal{J}^{(0)}(\mu) \le L = \liminf_{\epsilon\to0}\mathcal I_\epsilon(\mu_\epsilon),
\]
where the last liminf is taken over the original sequence.
This completes the proof.
\end{proof}

\subsubsection{$\Gamma-\limsup$}

The proof is based on the next elementary lemma.
\begin{lem}
\label{l_D1}Let $F\in C(\mathbb{R}^{d})\cap L^{1}(d\bm{x})$.
Let $(\varrho_{\epsilon}^{(1)})_{\epsilon>0}$ and
$(\varrho_{\epsilon}^{(2)})_{\epsilon>0}$
be sequences of positive numbers such that $\varrho_{\epsilon}^{(1)}\prec1\prec\varrho_{\epsilon}^{(2)}$,
$\varrho_{\epsilon}^{(1)}\varrho_{\epsilon}^{(2)}\le1$. Then, for any
$f,\,g\in C(\mathbb{R}^{d})$,
\[
\lim_{\epsilon\to0}\int_{B_{\varrho_{\epsilon}^{(2)}}}F(\bm{x})\,g(\frac{\bm{x}}{\varrho_{\epsilon}^{(2)}})\,f(\varrho_{\epsilon}^{(1)}\bm{x})\,d\bm{x}\,=\,g(\bm{0})\,f(\bm{0})\,\int_{\mathbb{R}^{d}}F(\bm{x})\,d\bm{x}\,.
\]
\end{lem}

\begin{proof}
Let $(a_{\epsilon})_{\epsilon>0}$ be a sequence of positive numbers
satisfying $1\prec a_{\epsilon}\prec\varrho_{\epsilon}^{(2)}$ so
that $a_{\epsilon}\varrho_{\epsilon}^{(1)}\prec1$. Then,
\begin{equation}
\begin{aligned} & \left|\int_{B_{\varrho_{\epsilon}^{(2)}}}F(\bm{x})\,g(\frac{\bm{x}}{\varrho_{\epsilon}^{(2)}})\,f(\varrho_{\epsilon}^{(1)}\bm{x})\,d\bm{x}-g(\bm{0})\,f(\bm{0})\,\int_{\mathbb{R}^{d}}F(\bm{x})\,d\bm{x}\right|\\
 & \le\left|\int_{B_{a_{\epsilon}}}F(\bm{x})\,g(\frac{\bm{x}}{\varrho_{\epsilon}^{(2)}})\,f(\varrho_{\epsilon}^{(1)}\bm{x})\,d\bm{x}-g(\bm{0})\,f(\bm{0})\,\int_{\mathbb{R}^{d}}F(\bm{x})\,d\bm{x}\right|\\
 & \ \ \ +\int_{B_{\varrho_{\epsilon}^{(2)}}\setminus B_{a_{\epsilon}}} \left| F(\bm{x})\,g(\frac{\bm{x}}{\varrho_{\epsilon}^{(2)}})\,f(\varrho_{\epsilon}^{(1)}\bm{x}) \right| \,d\bm{x}\,.
\end{aligned}
\label{e_l_D1-1}
\end{equation}
Since $\varrho_{\epsilon}^{(1)}\varrho_{\epsilon}^{(2)}\le1$, and
$f,\,g$ are bounded in $B_{1}$, there exists a constant $C_{1}>0$ such that
\begin{equation}
\left|\int_{B_{\varrho_{\epsilon}^{(2)}}\setminus B_{a_{\epsilon}}}F(\bm{x})\,g(\frac{\bm{x}}{\varrho_{\epsilon}^{(2)}})\,f(\varrho_{\epsilon}^{(1)}\bm{x})d\bm{x}\right|\le C_{1}\int_{B_{\varrho_{\epsilon}^{(2)}}\setminus B_{a_{\epsilon}}}|F(\bm{x})|\,d\bm{x}\,.\label{e_l_D1-2}
\end{equation}
This expression converges to zero as $\epsilon\to0$ because $F\in L^{1}(d\bm{x})$.

We turn to the first term of the right-hand side of \eqref{e_l_D1-1}.
Fix $\eta>0$. By continuity, there exists $\gamma>0$ such that
\[
\bm{x}\in B_{\gamma}\ \Rightarrow\ |g(\bm{x})-g(\bm{0})|\,,\ |f(\bm{x})-f(\bm{0})|\,\le\,\eta\,.
\]
 Fix $\epsilon_{1}>0$ such that for all $\epsilon\in(0,\,\epsilon_{1})$,
$a_{\epsilon}/\varrho_{\epsilon}^{(2)}<\gamma$ and $a_{\epsilon}\varrho_{\epsilon}^{(1)}<\gamma$.
Then, for all $\epsilon\in(0,\,\epsilon_{1})$,
\[
\bm{x}\in B_{a_{\epsilon}}\ \Rightarrow\ |\,g(\frac{\bm{x}}{\varrho_{\epsilon}^{(2)}})-g(\bm{0})\,|\,,\ |\,f(\varrho_{\epsilon}^{(1)}\bm{x})-f(\bm{0})\,|\,\le\,\eta\,.
\]
Therefore, there exists a constant $C_{2}>0$ such that for $\epsilon\in(0,\,\epsilon_{1})$,
\begin{equation}
\begin{aligned} & \left|\int_{B_{a_{\epsilon}}}F(\bm{x})\,g(\frac{\bm{x}}{\varrho_{\epsilon}^{(2)}})\,f(\varrho_{\epsilon}^{(1)}\bm{x})\,d\bm{x}-g(\bm{0})\,f(\bm{0})\,\int_{\mathbb{R}^{d}}F(\bm{x})\,d\bm{x}\right|\\
 & \le\left|g(\bm{0})\,f(\bm{0})\right|\left|\int_{B_{a_{\epsilon}}}F(\bm{x})\,d\bm{x}-\int_{\mathbb{R}^{d}}F(\bm{x})\,d\bm{x}\right|+C_{2}(\eta+\eta^{2})\int_{B_{a_{\epsilon}}}|\,F(\bm{x})\,|\,d\bm{x}\,.
\end{aligned}
\label{e_l_D1-3}
\end{equation}
 Since $F\in L^{1}(d\bm{x})$, by \eqref{e_l_D1-1}--\eqref{e_l_D1-3},
\[
\begin{aligned} & \limsup_{\epsilon\to0}\left|\int_{B_{\varrho_{\epsilon}^{(2)}}}F(\bm{x})\,g(\frac{\bm{x}}{\varrho_{\epsilon}^{(2)}})\,f(\varrho_{\epsilon}^{(1)}\bm{x})\,d\bm{x}-g(\bm{0})\,f(\bm{0})\,\int_{\mathbb{R}^{d}}F(\bm{x})\,d\bm{x}\right|\\
 & \ \ \ \le C_{2}(\eta+\eta^{2})\int_{\mathbb{R}^{d}}|\,F(\bm{x})\,|\,d\bm{x}\,.
\end{aligned}
\]
As $\eta>0$ can be arbitrarily small and $F\in L^{1}(d\bm{x})$,
the proof is complete.
\end{proof}

\begin{cor}
\label{c_D1}Let $\mathbb{A}\in\mathbb{R}^{d\times d}$ be a positive-definite
symmetric matrix and let $\delta=\delta(\epsilon)$ satisfy $\epsilon^{1/2}\prec\delta\le1$.
Then,
\begin{enumerate}
\item For all $f,\,g\in C(\mathbb{R}^{d})$,
\[
\lim_{\epsilon\to0}\,\frac{1}{(2\pi\epsilon)^{d/2}}\,\int_{B_{\delta}}e^{-\frac{1}{2\epsilon}\bm{x}\cdot\mathbb{A}\bm{x}}\,g(\frac{\bm{x}}{\delta})\,f(\bm{x})\,d\bm{x}\,=\,\frac{g(\bm{0})\,f(\bm{0})}{\sqrt{\det\mathbb{A}}}\,.
\]
\item For every nonnegative-definite
symmetric matrix $\mathbb{B}\in\mathbb{R}^{d\times d}$ and $g\in C(\mathbb{R}^d)$,
\[
\lim_{\epsilon\to0}\,\frac{1}{\epsilon(2\pi\epsilon)^{d/2}}\,\int_{B_{\delta}}e^{-\frac{1}{2\epsilon}\bm{x}\cdot\mathbb{A}\bm{x}}\,g(\frac{\bm{x}}{\delta})\,\bm{x}\cdot\mathbb{B}\bm{x}\,d\bm{x}\,=\,\frac{g(\bm{0})\,{\rm Tr}(\mathbb{B}\mathbb{A}^{-1})}{\sqrt{\det\mathbb{A}}}\,.
\]
\end{enumerate}
\end{cor}

\begin{proof}
By the change of variables $\bm{x}=\sqrt{\epsilon}\bm{y}$,
\[
\begin{aligned}\frac{1}{(2\pi\epsilon)^{d/2}}\,\int_{B_{\delta}}e^{-\frac{1}{2\epsilon}\bm{x}\cdot\mathbb{A}\bm{x}}\,g(\frac{\bm{x}}{\delta})\,f(\bm{x})\,d\bm{x} & =\,\frac{1}{(2\pi)^{d/2}}\,\int_{B_{\delta/\sqrt{\epsilon}}}e^{-\frac{1}{2}\bm{y}\cdot\mathbb{A}\bm{y}}\,g(\frac{\sqrt{\epsilon}\bm{y}}{\delta})\,f(\sqrt{\epsilon}\bm{y})\,d\bm{y}\,,\\
\frac{1}{\epsilon(2\pi\epsilon)^{d/2}}\,\int_{B_{\delta}}e^{-\frac{1}{2\epsilon}\bm{x}\cdot\mathbb{A}\bm{x}}\,g(\frac{\bm{x}}{\delta})\,\bm{x}\cdot\mathbb{B}\bm{x}\,d\bm{x} & =\,\frac{1}{(2\pi)^{d/2}}\,\int_{B_{\delta/\sqrt{\epsilon}}}e^{-\frac{1}{2}\bm{y}\cdot\mathbb{A}\bm{y}}\,g(\frac{\sqrt{\epsilon}\bm{y}}{\delta})\,\bm{y}\cdot\mathbb{B}\bm{y}\,d\bm{y}\,.
\end{aligned}
\]
The first identity implies the first assertion by Lemma \ref{l_D1} with $\varrho_{\epsilon}^{(1)}=\sqrt{\epsilon}$,
$\varrho_{\epsilon}^{(2)}=\delta/\sqrt{\epsilon}$, and $F(\bm{x})=(2\pi)^{-d/2}e^{-\frac{1}{2}\bm{x}\cdot\mathbb{A}\bm{x}}$. 

We turn to the second assertion. Let $X$ be a centered Gaussian random vector with
covariance matrix $\mathbb{A}^{-1}$. Then, $\mathbb{E}[X\cdot\mathbb{B}X]={\rm Tr}(\mathbb{B}\mathbb{A}^{-1})$
so that 
\[
\lim_{\epsilon\to0}\,\frac{1}{(2\pi)^{d/2}}\,\int_{B_{\delta/\sqrt{\epsilon}}}e^{-\frac{1}{2}\bm{y}\cdot\mathbb{A}\bm{y}}\,\bm{y}\cdot\mathbb{B}\bm{y}\,d\bm{y}\,=\,\frac{\mathbb{E}[X\cdot\mathbb{B}X]}{\sqrt{\det\mathbb{A}}}\,=\,\frac{{\rm Tr}(\mathbb{B}\mathbb{A}^{-1})}{\sqrt{\det\mathbb{A}}}\,.
\]
Therefore, the second assertion follows from Lemma \ref{l_D1}
with $\varrho_{\epsilon}^{(1)}=\sqrt{\epsilon}$, $\varrho_{\epsilon}^{(2)}=\delta/\sqrt{\epsilon}$,
and $F(\bm{x})=(2\pi)^{-d/2}e^{-\frac{1}{2}\bm{x}\cdot\mathbb{A}\bm{x}}\,\bm{x}\cdot\mathbb{B}\bm{x}$,
and $f\equiv1$.
\end{proof}

Let $h\in C^{1}(\mathbb{R}^{d})$ satisfy $\nabla h(\bm{0})=\bm{0}$.
Applying Corollary \ref{c_D1}-(1), by letting $f\equiv1$
and $g(\bm{x})=|\nabla h(\bm{x})|^{2}$ gives
\begin{equation}
\lim_{\epsilon\to0}\,\frac{\epsilon}{\delta^{2}(2\pi\epsilon)^{d/2}}\,\int_{B_{\delta}}e^{-\frac{1}{2\epsilon}\bm{x}\cdot\mathbb{A}\bm{x}}\,|\nabla h(\frac{\bm{x}}{\delta})|^{2}\,d\bm{x}\,=\,0\,,\label{e_c_D1}
\end{equation}
since $\lim_{\epsilon\to0}\epsilon/\delta^{2}=0$.

We now proceed to the proof of the $\Gamma$-limsup in Proposition \ref{p_Gamma}-(3).

\begin{proof}[Proof of $\Gamma-\limsup$ for Proposition \ref{p_Gamma}-(3)]
By convexity of $\mathcal{I}_{\epsilon}$ and linearity of
$\mathcal{J}^{(0)}$, it suffices to prove the $\Gamma-\limsup$ for
$\mu=\delta_{\bm{c}}$, $\bm{c}\in\mathcal{C}_{0}$. Without loss of generality,
assume that $\bm{c}=\bm{0}$ and $U(\bm{0})=0$. We divide the proof into three steps.

\smallskip
\noindent\textbf{Step 1.} Construction of measures
\smallskip

As in the proof
of the $\Gamma-\liminf$, we can write
\begin{equation}
\nabla^{2}U(\bm{0})=\mathbb{U}\mathbb{D}(\mathbb{U})^{-1}\label{e_422-1}
\end{equation}
for some orthogonal matrix \textcolor{blue}{$\mathbb{U}$} and diagonal matrix
\[
{\color{blue}\mathbb{D}}:={\rm diag}(\lambda_{1},\,\dots,\,\lambda_{d})
\]
where $\lambda_{1},\,\dots,\,\lambda_{d}$ are eigenvalues of $\nabla^{2}U(\bm{0})$.
Define the diagonal matrix \textcolor{blue}{$\widetilde{\mathbb{D}}$} as
\[
\widetilde{\mathbb{D}}:={\rm diag}(\Lambda_{1},\,\dots,\,\Lambda_{d})\,,
\]
where ${\color{blue}\Lambda_{i}}:=\min\{\lambda_{i},\,0\}$. Let ${\color{blue}G}\in C(\mathbb{R}^{d})$
be given by
\[
G(\bm{x})=\bm{x}\cdot\widetilde{\mathbb{H}}\bm{x}\,,
\]
where
\begin{equation}
{\color{blue}\widetilde{\mathbb{H}}}:=\mathbb{U}\widetilde{\mathbb{D}}(\mathbb{U})^{-1}\,.\label{e_422-2}
\end{equation}

Let ${\color{blue}\varphi}\in C_{c}^{\infty}(\mathbb{R}^{d})$ be such that
\[
\mathcal{X}_{B_{1/2}}\le\varphi\le\mathcal{X}_{B_{1}}\,,
\]
and define ${\color{blue}\varphi_{\epsilon}(\bm{x})}:=\varphi(\bm{x}/\delta)$ for
some $\epsilon^{1/2}\prec\delta=\delta(\epsilon)\prec\epsilon^{1/3}$.
Clearly,
\[
\mathcal{X}_{B_{\sqrt{\epsilon}}}\le\mathcal{X}_{B_{\delta/2}}\le\varphi_{\epsilon}\le\mathcal{X}_{B_{\delta}}\,.
\]
 Let
\[
{\color{blue}g_{\epsilon}(\bm{x})}:=e^{\frac{1}{2\epsilon}G(\bm{x})}\varphi_{\epsilon}(\bm{x})\,.
\]
For $\epsilon>0$, define probability measures
\[
{\color{blue}\mu_{\epsilon}(d\bm{x})}\,:=\,\frac{1}{A_{\epsilon}}\left(g_{\epsilon}(\bm{x})\right)^{2}d\pi_{\epsilon}(d\bm{x})
\]
where
\[
A_{\epsilon}\,:=\,\int_{\mathbb{R}^{d}}\,g_{\epsilon}^{2}\,d\pi_{\epsilon}\,.
\]

\smallskip
\noindent\textbf{Step 2.} Weak convergence of the sequence of measures
\smallskip

By the Taylor expansion, for $\bm{x}\in B_{\delta}$,
\[
U(\bm{x})=U(\bm{0})+\nabla U(\bm{0})\cdot\bm{x}+\frac{1}{2}\bm{x}\cdot\nabla^2 U(\bm{0})\bm{x}+O(\delta^{3})\,.
\]
As $U(\bm{0})=\nabla U(\bm{0})=0$,
\[
\begin{aligned}\exp\left\{ \frac{1}{\epsilon}(-U(\bm{x})+\bm{x}\cdot\widetilde{\mathbb{H}}\bm{x})\right\}  & =\exp\left\{ \frac{1}{\epsilon}\left[-\frac{1}{2}\bm{x}\cdot\nabla^{2} U(\bm{0})\bm{x}+\bm{x}\cdot\widetilde{\mathbb{H}}\bm{x}+O(\delta^{3})\right]\right\} \\
 & =\exp\left\{ -\frac{1}{2\epsilon}\bm{x}\cdot(\nabla^{2} U(\bm{0})-2\widetilde{\mathbb{H}})\bm{x}+O(\frac{\delta^{3}}{\epsilon})\right\} \,.
\end{aligned}
\]
Since $e^{x}=1+O(x)$ as $x\to0$ and $\delta^{3}/\epsilon\prec1$,
for $\bm{x}\in B_{\delta}$,
\begin{equation}
\exp\left\{ \frac{1}{\epsilon}(-U(\bm{x})+\bm{x}\cdot\widetilde{\mathbb{H}}\bm{x})\right\} =\exp\left\{ -\frac{1}{2\epsilon}\bm{x}\cdot(\nabla^{2} U(\bm{0})-2\widetilde{\mathbb{H}})\bm{x}\right\} \left(1+O(\frac{\delta^{3}}{\epsilon})\right)\,.\label{e_422-3}
\end{equation}
Since $\varphi(\bm{0})=1$ and the matrix $\nabla^{2} U(\bm{0})-2\widetilde{\mathbb{H}}=\mathbb{U}(\mathbb{D}-2\widetilde{\mathbb{D}})(\mathbb{U})^{-1}$
is positive-definite, by the first assertion of Corollary \ref{c_D1},
for all $f\in C(\mathbb{R}^{d})$,
\[
\lim_{\epsilon\to0}\int_{\mathbb{R}^{d}}fd\mu_{\epsilon}=\lim_{\epsilon\to0}\frac{\int_{B_{\delta}}\exp\left\{ -\frac{1}{2\epsilon}\bm{x}\cdot(\nabla^{2} U(\bm{0})-2\widetilde{\mathbb{H}})\bm{x}\right\} (\varphi_{\epsilon}(\bm{x}))^{2}f(\bm{x})d\bm{x}}{\int_{B_{\delta}}\exp\left\{ -\frac{1}{2\epsilon}\bm{x}\cdot(\nabla^{2} U(\bm{0})-2\widetilde{\mathbb{H}})\bm{x}\right\} (\varphi_{\epsilon}(\bm{x}))^{2}d\bm{x}}=f(\bm{0})\,,
\]
so that $\mu_{\epsilon}\to\delta_{\bm{0}}$ as $\epsilon\to0$.

\smallskip
\noindent\textbf{Step 3.} $\Gamma-\limsup$ inequality
\smallskip

By definition of $g_{\epsilon}$ and $\widetilde{\mathbb{H}}$,
\[
\begin{aligned}\nabla g_{\epsilon}(\bm{x}) & =\frac{1}{2\epsilon}e^{\frac{1}{2\epsilon}G(\bm{x})}\varphi_{\epsilon}(\bm{x})\nabla G(\bm{x})+e^{\frac{1}{2\epsilon}G(\bm{x})}\nabla\varphi_{\epsilon}(\bm{x})\\
 & =\frac{1}{\epsilon}e^{\frac{1}{2\epsilon}G(\bm{x})}\varphi_{\epsilon}(\bm{x})\widetilde{\mathbb{H}}\bm{x}+e^{\frac{1}{2\epsilon}G(\bm{x})}\nabla\varphi_{\epsilon}(\bm{x})\,,
\end{aligned}
\]
so that by \eqref{e_func_max},
\[
\begin{aligned}\mathcal{I}_{\epsilon}(\mu_{\epsilon}) & =\epsilon\int_{\mathbb{R}^{d}}\left|\nabla\sqrt{\frac{d\mu_{\epsilon}}{d\pi_{\epsilon}}}\right|^{2}d\pi_{\epsilon}\\
 & =\frac{\epsilon}{A_{\epsilon}}\int_{\mathbb{R}^{d}}|\nabla g_{\epsilon}|^{2}d\pi_{\epsilon}\\
 & =\Phi_{\epsilon}^{(1)}+\Phi_{\epsilon}^{(2)}+\Phi_{\epsilon}^{(3)}
\end{aligned}
\]
where
\[
\begin{aligned}\Phi_{\epsilon}^{(1)} & =\frac{1}{\epsilon A_{\epsilon}}\int_{\mathbb{R}^{d}}e^{\frac{1}{\epsilon}G(\bm{x})}|\varphi_{\epsilon}(\bm{x})|^{2}|\widetilde{\mathbb{H}}\bm{x}|^{2}\pi_{\epsilon}(d\bm{x})\,,\\
\Phi_{\epsilon}^{(2)} & =\frac{\epsilon}{A_{\epsilon}}\int_{\mathbb{R}^{d}}e^{\frac{1}{\epsilon}G(\bm{x})}|\nabla\varphi_{\epsilon}(\bm{x})|^{2}\pi_{\epsilon}(d\bm{x})\,,\\
\Phi_{\epsilon}^{(3)} & =\frac{2}{A_{\epsilon}}\int_{\mathbb{R}^{d}}e^{\frac{1}{\epsilon}G(\bm{x})}\varphi_{\epsilon}(\bm{x})\nabla\varphi_{\epsilon}(\bm{x})\cdot\widetilde{\mathbb{H}}\bm{x}\,\pi_{\epsilon}(d\bm{x})\,.
\end{aligned}
\]
By \eqref{e_422-3} and Corollary \ref{c_D1},
\[
\begin{aligned}\lim_{\epsilon\to0}\Phi_{\epsilon}^{(1)} & =\lim_{\epsilon\to0}\frac{\int_{B_{\delta}}\exp\left\{ -\frac{1}{2\epsilon}\bm{x}\cdot(\nabla^{2} U(\bm{0})-2\widetilde{\mathbb{H}})\bm{x}\right\} (\varphi_{\epsilon}(\bm{x}))^{2}\bm{x}\cdot(\widetilde{\mathbb{H}})^{2}\bm{x}d\bm{x}}{\epsilon\int_{B_{\delta}}\exp\left\{ -\frac{1}{2\epsilon}\bm{x}\cdot(\nabla^{2} U(\bm{0})-2\widetilde{\mathbb{H}})\bm{x}\right\} (\varphi_{\epsilon}(\bm{x}))^{2}d\bm{x}}\\
 & ={\rm Tr}\left((\widetilde{\mathbb{H}})^{2}(\nabla^{2} U(\bm{0})-2\widetilde{\mathbb{H}})^{-1}\right)\,.
\end{aligned}
\]
Using \eqref{e_422-1} and \eqref{e_422-2}, this equals
\[
\lim_{\epsilon\to0}\Phi_{\epsilon}^{(1)}={\rm Tr}\left((\widetilde{\mathbb{D}})^{2}(\mathbb{D}-2\widetilde{\mathbb{D}})^{-1}\right)=\zeta(\bm{0})=\mathcal{J}^{(0)}(\delta_{\bm{0}})\,.
\]
For the second term $\Phi_\epsilon^{(2)}$, \eqref{e_422-3} and \eqref{e_c_D1} give
\[
\lim_{\epsilon\to0}\Phi_{\epsilon}^{(2)}=\lim_{\epsilon\to0}\frac{\epsilon}{\delta^{2}A_{\epsilon}}\int_{\mathbb{R}^{d}}e^{\frac{1}{\epsilon}G(\bm{x})}|\nabla\varphi(\frac{\bm{x}}{\delta})|^{2}\pi_{\epsilon}(d\bm{x})=0\,.
\]
Finally, H\"older's inequality, together with the previous estimate, implies
$\lim_{\epsilon\to0}\Phi_{\epsilon}^{(3)}=0$. Thus,
\[
\lim_{\epsilon\to0}\mathcal{I}_{\epsilon}(\mu_{\epsilon})=\mathcal{J}^{(0)}(\delta_{\bm{0}}) \,,
\]
which completes the proof.
\end{proof}

\section{\label{sec_meta}Metastable scale}

In this section, we prove Proposition \ref{p_Gamma}-(4), namely the $\Gamma$-convergences
at the metastable scales $\theta_{\epsilon}^{(p)}$, $p\in\llbracket1,\,\mathfrak{q}\rrbracket$. 

\subsection{$\Gamma-\liminf$}

Our approach to the $\Gamma-\liminf$ is based on the resolvent approach developed in \cite{LMS-res}.

\subsubsection{Resolvent equation}

For $\lambda>0$, $p\in\llbracket1,\,\mathfrak{q}\rrbracket$, and
$\boldsymbol{g}\colon\mathscr{V}^{(p)}\rightarrow\mathbb{R}$,
Proposition \ref{p_gen}-(1) ensures there exists a unique solution ${\color{blue}F_{\epsilon}=F_{\epsilon}^{p,\boldsymbol{g},\lambda}}\in D(\mathscr{L}_{\epsilon})\subset L^{2}(d\pi_{\epsilon})$
to the resolvent equation 
\begin{equation}
\left(\lambda-\theta_{\epsilon}^{(p)}\mathscr{L}_{\epsilon}\right)\,F_{\epsilon}\,=\,\sum_{\mathcal{M}\in\mathscr{V}^{(p)}}\boldsymbol{g}(\mathcal{M})\,\mathcal{X}_{_{\mathcal{E}(\mathcal{M})}}\,,\label{e_res}
\end{equation}
where $\mathcal E(\mathcal M)$ is defined in \eqref{e_E_M}.

The following theorem, due to \cite{LLS-1st,LLS-2nd}, provides the asymptotic behavior of $F_{\epsilon}$.
\begin{thm}[{\cite[Theorem 2.14]{LLS-2nd}}]
\label{t_res}Fix a constant $\lambda>0$, $p\in\llbracket1,\,\mathfrak{q}\rrbracket$
and $\boldsymbol{g}\colon\mathscr{V}^{(p)}\rightarrow\mathbb{R}$.
Then, for all $\mathcal{M}\in\mathscr{V}^{(p)}$, the solution $F_{\epsilon}$
to the resolvent equation \eqref{e_res} satisfies 
\[
\lim_{\epsilon\rightarrow0}\,\sup_{\boldsymbol{x}\in\mathcal{E}(\mathcal{M})}\,\bigg|\,F_{\epsilon}(\boldsymbol{x})\,-\,\boldsymbol{f}(\mathcal{M})\,\bigg|\ =\ 0\;,
\]
where $\boldsymbol{f}:\mathscr{V}^{(p)}\rightarrow\mathbb{R}$ denotes
the unique solution of the reduced resolvent equation 
\begin{equation}
\left(\lambda-\mathfrak{L}^{(p)}\right)\,\boldsymbol{f}\ =\ \boldsymbol{g}\;.\label{eq:res_y}
\end{equation}
\end{thm}

It is well known from \cite[Section 6.5]{Friedman} that $F_{\epsilon}$
admits the probabilistic representation
\begin{equation}
F_{\epsilon}(\boldsymbol{x})=\mathbb{E}_{\boldsymbol{x}}^{\epsilon}\left[\,\int_{0}^{\infty}e^{-\lambda s}\,G(\boldsymbol{x}_{\epsilon}(\theta_{\epsilon}^{(p)}s))\,ds\,\right]\,,\label{eq:probex}
\end{equation}
where
\[
G:=\sum_{\mathcal{M}\in\mathscr{V}^{(p)}}\boldsymbol{g}(\mathcal{M})\,\mathcal{X}_{_{\mathcal{E}(\mathcal{M})}}\,.
\]

\subsubsection{Main lemma}

Throughout the article, $o_{\epsilon}(1)$ denotes a remainder term
which vanishes as $\epsilon\to0$. The next result establishes the
$\Gamma-\liminf$ of the sequence
$(\theta_{\epsilon}^{(p)}\mathcal{I}_{\epsilon})_{\epsilon>0}$,
$p\in\llbracket1,\,\mathfrak{q}\rrbracket$, for convex combinations of
the measures $\pi_{\mathcal{M}}$, $\mathcal{M}\in\mathscr{V}^{(p)}$.
The proof of the full $\Gamma-\liminf$ will be given at the end of
this section and relies on the next result.

\begin{lem}
\label{l_liminf}
Fix $p\in\llbracket1,\,\mathfrak{q}\rrbracket$ and
let $\mu=\sum_{\mathcal{M}\in\mathscr{V}^{(p)}}\omega(\mathcal{M})\,\pi_{\mathcal{M}}\in\mathcal{P}(\mathbb{R}^{d})$
for some $\omega\in\mathcal{P}(\mathscr{V}^{(p)})$. Then, for every
sequence $(\mu_{\epsilon})_{\epsilon>0}$ in $\mathcal{P}(\mathbb{R}^{d})$
converging to $\mu$,
\[
\liminf_{\epsilon\to0}\,\theta_{\epsilon}^{(p)}\,\mathcal{I}_{\epsilon}(\mu_{\epsilon})\,\ge\,\mathfrak{J}^{(p)}(\omega)\,.
\]
\end{lem}

\begin{proof}
Let $\bm{h}:\mathscr{V}^{(p)}\to(0,\,\infty)$ be a positive function.
Since $\mathscr{V}^{(p)}$ is finite, we may choose
\[
\lambda > \max \left\{ 0, \, \max_{\mathcal{M}\in\mathscr{V}^{(p)}} \frac{\mathfrak{L}^{(p)}\bm{h}(\mathcal{M})}{\bm{h}(\mathcal{M})} \right\} \,.
\]
Define $\bm{g}:\mathscr{V}^{(p)}\to\mathbb{R}$ by
\[
\bm{g}:=(\lambda-\mathfrak{L}^{(p)})\bm{h}\,,
\]
so that $\bm{g}>0$.
Define $G:\mathbb{R}^{d}\to\mathbb{R}$ by 
\[
G:=\sum_{\mathcal{M}\in\mathscr{V}^{(p)}}\bm{g}(\mathcal{M})\,\mathcal{X}_{\mathcal{E}(\mathcal{M})}\,,
\]
and let $F_{\epsilon}=F_{\epsilon}^{\lambda,\,\bm{g}}$ be the solution
to \eqref{e_res}. Since $G\ge0$, representation \eqref{eq:probex} gives $F_{\epsilon}\ge0$.

Fix $a<0$. By Lemma \ref{l_gen}, $F_{\epsilon}+e^{a/\epsilon}\in D(\mathscr{L}_{\epsilon})$,
and since $F_{\epsilon}+e^{a/\epsilon}>0$,
\[
\theta_{\epsilon}^{(p)}\mathcal{I}_{\epsilon}(\mu_{\epsilon}) = \sup_{u>0}\int_{\mathbb{R}^{d}}-\frac{\theta_{\epsilon}^{(p)}\mathscr{L}_{\epsilon}u}{u}d\mu_{\epsilon}
\ge\int_{\mathbb{R}^{d}}-\frac{\theta_{\epsilon}^{(p)}\mathscr{L}_{\epsilon}F_{\epsilon}}{F_{\epsilon}+e^{a/\epsilon}}d\mu_{\epsilon}\,.
\]
Since $F_{\epsilon}$ is the solution to \eqref{e_res}, the last
term is equal to
\[
\int_{\mathbb{R}^{d}}\frac{G-\lambda F_{\epsilon}}{F_{\epsilon}+e^{a/\epsilon}}d\mu_{\epsilon}=-\lambda\int_{\mathbb{R}^{d}}\frac{F_{\epsilon}}{F_{\epsilon}+e^{a/\epsilon}}d\mu_{\epsilon}+\int_{\mathbb{R}^{d}}\frac{G}{F_{\epsilon}+e^{a/\epsilon}}d\mu_{\epsilon}\,.
\]
Since $\frac{F_{\epsilon}}{F_{\epsilon}+e^{a/\epsilon}}\le1$, $G\ge0$,
and $G=\bm{g}(\mathcal{M})$ on $\mathcal{E}(\mathcal{M})$, the last expression is bounded below by
\begin{equation}
-\lambda+\sum_{\mathcal{M}\in\mathscr{V}^{(p)}}\int_{\mathcal{E}(\mathcal{M})}\,\frac{\bm{g}(\mathcal{M})}{F_{\epsilon}+e^{a/\epsilon}}\,d\mu_{\epsilon}\,.\label{e_res_inf}
\end{equation}

By Theorem \ref{t_res}, since $a<0$, 
\[
\lim_{\epsilon\to0}\,\sup_{\mathcal{M}\in\mathscr{V}^{(p)}}\,\|F_{\epsilon}+e^{a/\epsilon}-\bm{h}(\mathcal{M})\|_{L^{\infty}(\mathcal{E}(\mathcal{M}))}\,=\,0\,.
\]
Hence \eqref{e_res_inf} is bounded below by
\[
-\lambda+\sum_{\mathcal{M}\in\mathscr{V}^{(p)}}[1+o_{\epsilon}(1)]\,\frac{\bm{g}(\mathcal{M})}{\bm{h}(\mathcal{M})}\,\mu_{\epsilon}(\mathcal{E}(\mathcal{M}))\,=-\,\lambda+\sum_{\mathcal{M}\in\mathscr{V}^{(p)}}[1+o_{\epsilon}(1)]\,\frac{(\lambda-\mathfrak{L}^{(p)})\bm{h}(\mathcal{M})}{\bm{h}(\mathcal{M})}\,\mu_{\epsilon}(\mathcal{E}(\mathcal{M}))\,.
\]
Since $\mu_{\epsilon}\to\mu$, $\sum_{\mathcal{M}\in\mathscr{V}^{(p)}}\mu(\mathcal{E}(\mathcal{M}))=1$,
and ${\bm h}$ is bounded, the previous expression is bounded below
by
\[
-\lambda+\sum_{\mathcal{M}\in\mathscr{V}^{(p)}}\frac{(\lambda-\mathfrak{L}^{(p)})\bm{h}(\mathcal{M})}{\bm{h}(\mathcal{M})}\,\mu(\mathcal{E}(\mathcal{M}))+o_{\epsilon}(1)\,=\sum_{\mathcal{M}\in\mathscr{V}^{(p)}}\frac{-\mathfrak{L}^{(p)}\bm{h}(\mathcal{M})}{\bm{h}(\mathcal{M})}\,\omega(\mathcal{M})+o_{\epsilon}(1)\,.
\]
Therefore,
\[
\liminf_{\epsilon\to0}\,\theta_{\epsilon}^{(p)}\mathcal{I}_{\epsilon}(\mu_{\epsilon})\,\ge\,\sum_{\mathcal{M}\in\mathscr{V}^{(p)}}\frac{-\mathfrak{L}^{(p)}\bm{h}(\mathcal{M})}{\bm{h}(\mathcal{M})}\,\omega(\mathcal{M}) \,.
\]
Taking the supremum over all positive $\bm{h}$ yields
\[
\liminf_{\epsilon\to0}\,\theta_{\epsilon}^{(p)}\mathcal{I}_{\epsilon}(\mu_{\epsilon})\,\ge\,\sup_{\bm{u}>0}\,\sum_{\mathcal{M}\in\mathscr{V}^{(p)}}\frac{-\mathfrak{L}^{(p)}\bm{u}(\mathcal{M})}{\bm{u}(\mathcal{M})}\,\omega(\mathcal{M})\,=\,\mathfrak{J}^{(p)}(\omega)\,.
\]
\end{proof}

\subsection{$\Gamma-\limsup$}

For the $\Gamma-\limsup$ at the time scale $\theta_{\epsilon}^{(p)}$,
$p\in\llbracket1,\,\mathfrak{q}\rrbracket$, the convexity of $\mathcal{I}_{\epsilon}$
together with Lemma \ref{l_DV_equiv-2} implies that it suffices to consider measures $\omega\in\mathcal{P}(\mathscr{V}^{(p)})$
supported on a single equivalence class of the chain $\{{\bf y}^{(p)}(t)\}_{t\ge0}$.

\subsubsection{\label{subsec521}Equivalence class}
We first introduce the notion of a simple set.
\begin{itemize}
\item If a nonempty set $\mathcal{M}\subset\mathcal{M}_{0}$ satisfies
\[
U(\bm{m})=U(\bm{m}')\ \ \text{for all}\ \bm{m},\,\bm{m}'\in\mathcal{M}\,,
\]
$\mathcal{M}$ is said to be \textit{\textcolor{blue}{simple}} and
we denote by \textcolor{blue}{$U(\mathcal{M})$ }the common value.
\end{itemize}
By Proposition \ref{p: tree}-(2) (cf. property $\mathfrak{P}_{1}$
in \cite{LLS-2nd}), every $\mathcal{M}\in\mathscr{S}^{(n)}$, $n\in\llbracket1,\,\mathfrak{q}\rrbracket$, is simple.
Furthermore, Lemma \ref{l_equi_height} shows that for any $p\in\llbracket1,\,\mathfrak{q}\rrbracket$
and any equivalence class $\mathfrak{D}\subset\mathscr{V}^{(p)}$ of
the limiting Markov chain $\{{\bf y}^{(p)}(t)\}_{t\ge0}$,
\[
U(\mathcal{M})=U(\mathcal{M}')\ \text{for all}\ \mathcal{M},\,\mathcal{M}'\in\mathfrak{D}\,.
\]
We denote this common value by ${\color{blue}H_{\mathfrak{D}}}\in\mathbb{R}$.

Let \textcolor{blue}{$\{{\bf y}_{\mathfrak{D}}^{(p)}(t)\}_{t\ge0}$} be
the Markov chain restricted to $\mathfrak{D}$, with jump rates
\begin{equation}
{\color{blue}r_{\mathfrak{D}}^{(p)}(\mathcal{M},\,\mathcal{M}')}:=r^{(p)}(\mathcal{M},\,\mathcal{M}')\ \ ;\ \ \mathcal{M},\,\mathcal{M}'\in\mathfrak{D}\,,\label{e_y_D}
\end{equation}
where ${\color{blue}r^{(p)}}:\mathscr{V}^{(p)}\times\mathscr{V}^{(p)}\to[0,\,\infty)$
are the jump rates of $\{{\bf y}^{(p)}(t)\}_{t\ge0}$. We also denote by ${\color{blue}\nu_{\mathfrak{D}}}\in\mathcal{P}(\mathfrak{D})$
the measure $\nu$ conditioned on $\mathfrak{D}$:
\[
\nu_{\mathfrak{D}}(\mathcal{M}):=\frac{\nu(\mathcal{M})}{\sum_{\mathcal{M}'\in\mathfrak{D}}\nu(\mathcal{M}')}\,.
\]

The following result shows that the restricted chain $\{{\bf y}_{\mathfrak{D}}^{(p)}(t)\}_{t\ge0}$
is reversible with respect to $\nu_{\mathfrak{D}}$.
\begin{prop}
\label{p_rev}
Fix $p\in\llbracket1,\,\mathfrak{q}\rrbracket$ and let
$\mathfrak{D}\subset\mathscr{V}^{(p)}$ be an equivalence class of the
limiting chain $\{{\bf y}^{(p)}(t)\}_{t\ge0}$ such that $|\mathfrak{D}|\ge2$.
Then, $\{{\bf y}_{\mathfrak{D}}^{(p)}(t)\}_{t\ge0}$ is reversible
with respect to the conditioned measure $\nu_{\mathfrak{D}}$.
\end{prop}

The proof is postponed to Section \ref{subsec_rev}, as it requires
several notions introduced in \cite{LLS-1st,LLS-2nd}.

\subsubsection{Construction of a sequence of measures}

Recall from \eqref{e_def_nu} the definition of $\nu_{\star}$.
\begin{prop}
\label{p_test}
Fix $p\in\llbracket1,\,\mathfrak{q}\rrbracket$ and let
$\mathfrak{D}\subset\mathscr{V}^{(p)}$ be an equivalence class of the
limiting chain $\{{\bf y}^{(p)}(t)\}_{t\ge0}$. Then there exists
a family $\{h_{\mathcal{M}}^{\epsilon}:\mathcal{M}\in\mathfrak{D}\}$
of continuous functions $h_{\mathcal{M}}^{\epsilon}:\mathbb{R}^{d}\to\mathbb{R}$
satisfying the following conditions.
\begin{enumerate}
\item For all $\mathcal{M}\in\mathfrak{D}$,
\[
0\le h_{\mathcal{M}}^{\epsilon}\le1\,,\quad
h_{\mathcal{M}}^{\epsilon}(\bm{x})=1 \ \text{for}\  \bm{x}\in\mathcal{E}(\mathcal{M}) \,,
\]
and
\[
\lim_{\epsilon\to0}e^{H_{\mathfrak{D}}/\epsilon}\int_{\mathbb{R}^{d}\setminus\mathcal{E}(\mathcal{M})}(h_{\mathcal{M}}^{\epsilon})^{2}\,d\pi_{\epsilon}=0\,.
\]
\item For all $\mathcal{M}\in\mathfrak{D}$, 
\begin{equation}
\label{p_test-Diri}
\lim_{\epsilon\to0}\,e^{H_{\mathfrak{D}}/\epsilon}\theta_{\epsilon}^{(p)}\epsilon\int_{\mathbb{R}^{d}}\left|\nabla h_{\mathcal{M}}^{\epsilon}\right|^{2}d\pi_{\epsilon}=\frac{\nu(\mathcal{M})}{\nu_{\star}}\sum_{\mathcal{M}''\in\mathscr{V}^{(p)}}r^{(p)}(\mathcal{M},\,\mathcal{M}'')\,.
\end{equation}
\item If $|\mathfrak{D}|\ge2$, for all distinct $\mathcal{M},\, \mathcal{M}'\in\mathfrak{D}$,
\[
\begin{aligned} & \lim_{\epsilon\to0}\,e^{H_{\mathfrak{D}}/\epsilon}\theta_{\epsilon}^{(p)}\epsilon\int_{\mathbb{R}^{d}}\nabla h_{\mathcal{M}}^{\epsilon}\cdot\nabla h_{\mathcal{M}'}^{\epsilon}\,d\pi_{\epsilon}\\
 & =-\frac{1}{2\nu_{\star}}\left(\nu(\mathcal{M})\,r^{(p)}(\mathcal{M},\,\mathcal{M}')+\nu(\mathcal{M}')\,r^{(p)}(\mathcal{M}',\,\mathcal{M})\right)\,.
\end{aligned}
\]
\end{enumerate}
\end{prop}

The proof is postponed to Section \ref{sec_pf_equi_pot}.

Define
\[
{\color{blue}\mathcal{E}(\mathfrak{D}}\mathclose{\color{blue})}:=\bigcup_{\mathcal{M}\in\mathfrak{D}}\mathcal{E}(\mathcal{M})\,.
\]
The following is a consequence of Proposition \ref{p_test}.
\begin{lem}
\label{l_limsup-0}
Fix $p\in\llbracket1,\,\mathfrak{q}\rrbracket$
and let $\mathfrak{D}\subset\mathscr{V}^{(p)}$ be an equivalence class of the limiting chain $\{{\bf y}^{(p)}(t)\}_{t\ge0}$.
Let $\{h_{\mathcal{M}}^{\epsilon}:\mathcal{M}\in\mathfrak{D}\}$ be the family of continuous functions defined in Proposition \ref{p_test}.
For ${\bf g}:\mathscr{V}^{(p)}\to\mathbb{R}$ such that ${\bf g}(\mathcal{M})=0$ for $\mathcal{M}\notin\mathfrak{D}$, define $G_{\epsilon}=G_{\epsilon}^{\mathfrak{D},\,{\bf g}}:\mathbb{R}^{d}\to\mathbb{R}$
by
\begin{equation}
G_{\epsilon}(\bm{x})=\sum_{\mathcal{M}\in\mathfrak{D}}e^{H_{\mathfrak{D}}/2\epsilon}{\bf g}(\mathcal{M})\,h_{\mathcal{M}}^{\epsilon}(\bm{x})\,.\label{e_limsup_G}
\end{equation}
Then, for each $\mathcal{M}_1\in\mathfrak{D}$, $\bm{m}\in\mathcal{M}_1$,
and $\delta>0$ such that $B_{\delta}(\bm{m})\subset\mathcal{E}(\bm{m})$,
\[
\begin{gathered}\lim_{\epsilon\to0}\int_{B_{\delta}(\bm{m})}(G_{\epsilon})^{2}d\pi_{\epsilon}={\bf g}(\mathcal{M}_1)^{2}\frac{\nu(\bm{m})}{\nu_{\star}}\,,\\
\lim_{\epsilon\to0}\int_{\mathbb{R}^{d}\setminus B_{\delta}(\mathfrak{D})}(G_{\epsilon})^{2}d\pi_{\epsilon}=0\,,
\end{gathered}
\]
where $B_{\delta}(\mathfrak{D}):=\bigcup_{\mathcal{M}\in\mathfrak{D}}\bigcup_{\bm{m}\in\mathcal{M}}B_{\delta}(\bm{m})$.
Moreover,
\[
\lim_{\epsilon\to0}\theta_{\epsilon}^{(p)}\epsilon\int_{\mathbb{R}^{d}}\left|\nabla G_{\epsilon}\right|^{2}d\pi_{\epsilon}=\nu_{\star}^{-1}(A_{1}-A_{2})
\]
where
\[
\begin{aligned}A_{1} & :=\sum_{\mathcal{M}\in\mathfrak{D}}\nu(\mathcal{M}){\bf g}(\mathcal{M})^{2}\sum_{\mathcal{M}'\in\mathscr{V}^{(p)}\setminus\{\mathcal{M}\}}r^{(p)}(\mathcal{M},\,\mathcal{M}')\,,\\
A_{2} & :=\sum_{\mathcal{M}\in\mathfrak{D}}\nu(\mathcal{M}){\bf g}(\mathcal{M})\sum_{\mathcal{M}'\in\mathscr{V}^{(p)}\setminus\{\mathcal{M}\}}{\bf g}(\mathcal{M}')r^{(p)}(\mathcal{M},\,\mathcal{M}')\,.
\end{aligned}
\]
\end{lem}

\begin{proof}
Fix $\mathcal{M}\in\mathfrak{D}$, $\bm{m}\in\mathcal{M}$, and $\delta>0$. By
Proposition \ref{p_test}-(1),
\[
\int_{B_{\delta}(\bm{m})}(G_{\epsilon})^{2}d\pi_{\epsilon}=\sum_{\mathcal{M}',\,\mathcal{M}''\in\mathfrak{D}}{\bf g}(\mathcal{M}'){\bf g}(\mathcal{M}'')e^{H_{\mathfrak{D}}/\epsilon}\int_{B_{\delta}(\bm{m})}h_{\mathcal{M}'}^{\epsilon}h_{\mathcal{M}''}^{\epsilon}\,d\pi_{\epsilon}\,.
\]
By the second property of Proposition \ref{p_test}-(1), the overlap with other wells is negligible, so that only the term $\mathcal{M}'=\mathcal{M}''=\mathcal{M}$ contributes in the limit. Since $h_{\mathcal{M}}^{\epsilon}=1$ on $\mathcal{E}(\mathcal{M})$, using the asymptotics of $\pi_{\epsilon}$ near $\bm{m}$,
\[
\pi_{\epsilon}(B_{\delta}(\bm{m}))=[1+o_{\epsilon}(1)]\,\frac{\nu(\bm{m})}{\nu_{\star}}e^{-H_{\mathfrak{D}}/\epsilon}\,,
\]
we obtain
\[
\int_{B_{\delta}(\bm{m})}(G_{\epsilon})^{2}d\pi_{\epsilon}=[1+o_{\epsilon}(1)]\,{\bf g}(\mathcal{M})^{2}\frac{\nu(\bm{m})}{\nu_{\star}}\,.
\]

Next, consider the contribution inside $\mathcal{E}(\mathfrak{D})$ but away from neighborhoods of the minima.
Since ${\bf g}$ and $h_{\mathcal{M}'}^{\epsilon}$ are bounded and since
\[
\pi_{\epsilon}\big(\mathcal{E}(\mathcal{M}')\setminus\bigcup_{\bm{m}\in\mathcal{M}'}B_{\delta}(\bm{m})\big)=o_{\epsilon}(1)e^{-H_{\mathfrak{D}}/\epsilon}\quad \text{for}\ \ \mathcal{M}'\in\mathfrak{D}\,,
\]
we deduce
\begin{equation}
\lim_{\epsilon\to0}\int_{\mathcal{E}(\mathfrak{D})\setminus B_{\delta}(\mathfrak{D})}(G_{\epsilon})^{2}d\pi_{\epsilon}=0\,.\label{e_57}
\end{equation}

Now consider the contribution outside $\mathcal{E}(\mathfrak{D})$.
By H\"older's inequality,
\[
\begin{aligned} & \int_{\mathbb{R}^{d}\setminus\mathcal{E}(\mathfrak{D})}(G_{\epsilon})^{2}d\pi_{\epsilon}\\
 & =\sum_{\mathcal{M}',\,\mathcal{M}''\in\mathfrak{D}}{\bf g}(\mathcal{M}'){\bf g}(\mathcal{M}'')e^{H_{\mathfrak{D}}/\epsilon}\int_{\mathbb{R}^{d}\setminus\mathcal{E}(\mathfrak{D})}h_{\mathcal{M}'}^{\epsilon}h_{\mathcal{M}''}^{\epsilon}\,d\pi_{\epsilon}\\
 & \le\sum_{\mathcal{M}',\,\mathcal{M}''\in\mathfrak{D}} \left| {\bf g}(\mathcal{M}'){\bf g}(\mathcal{M}'') \right| \sqrt{e^{H_{\mathfrak{D}}/\epsilon}\int_{\mathbb{R}^{d}\setminus\mathcal{E}(\mathfrak{D})}(h_{\mathcal{M}'}^{\epsilon})^{2}d\pi_{\epsilon}}\sqrt{e^{H_{\mathfrak{D}}/\epsilon}\int_{\mathbb{R}^{d}\setminus\mathcal{E}(\mathfrak{D})}(h_{\mathcal{M}''}^{\epsilon})^{2}d\pi_{\epsilon}}\,.
\end{aligned}
\]
By Proposition \ref{p_test}-(1), each factor inside the square roots vanishes as $\epsilon\to0$.
Together with \eqref{e_57}, this proves
\[
\lim_{\epsilon\to0}\int_{\mathbb{R}^{d}\setminus B_{\delta}(\mathfrak{D})}(G_{\epsilon})^{2}d\pi_{\epsilon}=0\,.
\]

Finally, we evaluate the Dirichlet form. Since 
\[
\left|\nabla G_{\epsilon}\right|^{2}=e^{H_{\mathfrak{D}}/\epsilon}\sum_{\mathcal{M}',\,\mathcal{M}''\in\mathfrak{D}}{\bf g}(\mathcal{M}'){\bf g}(\mathcal{M}'')\nabla h_{\mathcal{M}'}^{\epsilon}\cdot\nabla h_{\mathcal{M}''}^{\epsilon}\,,
\]
Proposition \ref{p_test}-(2, 3) completes the proof.
\end{proof}

\subsubsection{Main lemma}

\begin{lem}
\label{l_limsup}
Fix $p\in\llbracket1,\,\mathfrak{q}\rrbracket$
and let $\mathfrak{D}\subset\mathscr{V}^{(p)}$ be an equivalence class of the limiting chain $\{{\bf y}^{(p)}(t)\}_{t\ge0}$.
For any $\omega\in\mathcal{P}(\mathfrak{D})$, there exists a sequence
$(\mu_{\epsilon})_{\epsilon>0}$ in $\mathcal{P}(\mathbb{R}^{d})$
such that $\mu_{\epsilon}\to\sum_{\mathcal{M}\in\mathfrak{D}}\omega(\mathcal{M})\pi_{\mathcal{M}}$
as $\epsilon\to0$ and
\[
\limsup_{\epsilon\to0}\theta_{\epsilon}^{(p)}\mathcal{I}_{\epsilon}(\mu_{\epsilon})\le\mathfrak{J}^{(p)}(\omega)\,.
\]
\end{lem}

\begin{proof}
Suppose that $|\mathfrak{D}|\ge2$.
By Proposition \ref{p_rev}, the
Markov chain $\{{\bf y}_{\mathfrak{D}}^{(p)}(t)\}_{t\ge0}$, defined
in \eqref{e_y_D}, is reversible with respect to the probability measure
$\nu_{\mathfrak{D}}=\nu/\nu(\mathfrak{D})\in\mathcal{P}(\mathfrak{D})$.
Let $\mathfrak{L}_{\mathfrak{D}}^{(p)}$ be the infinitesimal generator
of $\{{\bf y}_{\mathfrak{D}}^{(p)}(t)\}_{t\ge0}$. By Lemmas
\ref{l_DV_equiv-1} and \ref{l_DV_equiv-3},
\begin{equation}
\begin{aligned}\mathfrak{J}^{(p)}(\omega) & =\mathfrak{J}_{\mathfrak{D}}^{(p)}(\omega)+\sum_{\mathcal{M}\in\mathfrak{D}}\sum_{\mathcal{M}'\in\mathscr{V}^{(p)}\setminus\mathfrak{D}}\omega(\mathcal{M})r^{(p)}(\mathcal{M},\,\mathcal{M}')\\
 & =-\sum_{\mathcal{M}\in\mathfrak{D}}\nu_{\mathfrak{D}}(\mathcal{M}){\bf h}(\mathcal{M})\mathfrak{L}_{\mathfrak{D}}^{(p)}{\bf h}(\mathcal{M})+\sum_{\mathcal{M}\in\mathfrak{D}}\omega(\mathcal{M})\sum_{\mathcal{M}'\in\mathscr{V}^{(p)}\setminus\mathfrak{D}}r^{(p)}(\mathcal{M},\,\mathcal{M}')\,,
\end{aligned}
\label{e_58}
\end{equation}
where $\mathfrak{J}_{\mathfrak{D}}^{(p)}:\mathcal{P}(\mathfrak{D})\to[0,\,\infty]$
denotes the large deviation rate functional of $\{{\bf y}_{\mathfrak{D}}^{(p)}(t)\}_{t\ge0}$,
and
\[
{\bf h}(\mathcal{M})=\sqrt{\frac{\omega(\mathcal{M})}{\nu_{\mathfrak{D}}(\mathcal{M})}}\ \ ;\ \ \mathcal{M}\in\mathfrak{D}\,.
\]
Extend ${\bf h}:\mathfrak{D}\to\mathbb{R}$ to ${\bf h}:\mathscr{V}^{(p)}\to\mathbb{R}$
so that ${\bf h}(\mathcal{M})=0$ for $\mathcal{M}\in\mathscr{V}^{(p)}\setminus\mathfrak{D}$.

Set
\[
\bm{g}(\mathcal{M}):=\sqrt{\frac{\nu_{\star}\omega(\mathcal{M})}{\nu(\mathcal{M})}}=\sqrt{\frac{\nu_{\star}}{\nu(\mathfrak{D})}}\bm{h}(\mathcal{M})\ \ ;\ \ \mathcal{M}\in\mathscr{V}^{(p)}\,,
\]
and define $G_{\epsilon}$ as in \eqref{e_limsup_G}. By Lemma \ref{l_limsup-0},
$\lim_{\epsilon\to0}\int_{\mathbb{R}^{d}}|G_{\epsilon}|^{2}d\pi_{\epsilon}=\sum_{\mathcal{M}\in\mathfrak{D}}\omega(\mathcal{M})=1$.
Let
\begin{equation}
F_{\epsilon}(\bm{x}):=\frac{1}{\sqrt{\int_{\mathbb{R}^{d}}|G_{\epsilon}|^{2}d\pi_{\epsilon}}}|G_{\epsilon}(\bm{x})|=\frac{1}{1+o_{\epsilon}(1)}|G_{\epsilon}(\bm{x})|\,,\label{e_59}
\end{equation}
and set $\mu_{\epsilon}=|F_{\epsilon}|^{2}d\pi_{\epsilon}\in\mathcal{P}(\mathbb{R}^d)$.
Then,
by Lemma \ref{l_limsup-0},
\[
\begin{aligned} & \lim_{\epsilon\to0}\mu_{\epsilon}\Big(\mathbb{R}^{d}\setminus\bigcup_{\mathcal{M}'\in\mathfrak{D}}B_{\delta}(\mathcal{M}')\Big)=0\,,\\
 & \lim_{\epsilon\to0}\mu_{\epsilon}\Big(B_{\delta}(\bm{m})\Big)=\frac{\nu(\bm{m})}{\nu(\mathcal{M})}\omega(\mathcal{M})=\pi_{\mathcal{M}}(\bm{m})\,\omega(\mathcal{M})\,,
\end{aligned}
\]
for all sufficiently small $\delta>0$, $\mathcal{M}\in\mathfrak{D}$, and $\bm{m}\in\mathcal{M}$, so that $\mu_{\epsilon}\to\sum_{\mathcal{M}\in\mathfrak{D}}\omega(\mathcal{M})\pi_{\mathcal{M}}$.

By the definition \eqref{e_59} of $F_{\epsilon}$ and Lemma \ref{l_limsup-0},
\begin{equation}
\begin{aligned} & \lim_{\epsilon\to0}\theta_{\epsilon}^{(p)}\epsilon\int_{\mathbb{R}^{d}}\left|\nabla F_{\epsilon}\right|^{2}d\pi_{\epsilon}\\
 & =\frac{1}{\nu_{\star}}\bigg(\sum_{\mathcal{M}\in\mathfrak{D}}\nu(\mathcal{M}){\bf g}(\mathcal{M})^{2}\sum_{\mathcal{M}'\in\mathscr{V}^{(p)}\setminus\{\mathcal{M}\}}r^{(p)}(\mathcal{M},\,\mathcal{M}')\\
 & \ \ \ \ \ -\sum_{\mathcal{M}\in\mathfrak{D}}\nu(\mathcal{M}){\bf g}(\mathcal{M})\sum_{\mathcal{M}'\in\mathscr{V}^{(p)}\setminus\{\mathcal{M}\}}{\bf g}(\mathcal{M}')\,r^{(p)}(\mathcal{M},\,\mathcal{M}')\bigg)\,.
\end{aligned}
\label{e_511}
\end{equation}
Since
\[
\begin{aligned} & \frac{\nu(\mathcal{M})}{\nu_{\star}}{\bf g}(\mathcal{M})^{2}=\frac{\nu(\mathcal{M})}{\nu(\mathfrak{D})}{\bf h}(\mathcal{M})^{2}=\nu_{\mathfrak{D}}(\mathcal{M}){\bf h}(\mathcal{M})^{2}\,,\\
 & \frac{\nu(\mathcal{M})}{\nu_{\star}}{\bf g}(\mathcal{M}){\bf g}(\mathcal{M}')=\frac{\nu(\mathcal{M})}{\nu(\mathfrak{D})}{\bf h}(\mathcal{M}){\bf h}(\mathcal{M}')=\nu_{\mathfrak{D}}(\mathcal{M}){\bf h}(\mathcal{M}){\bf h}(\mathcal{M}')\,,
\end{aligned}
\]
the right-hand side of \eqref{e_511} is equal to
\[
\begin{aligned} & \sum_{\mathcal{M}\in\mathfrak{D}}\nu_{\mathfrak{D}}(\mathcal{M}){\bf h}(\mathcal{M})^{2}\sum_{\mathcal{M}'\in\mathscr{V}^{(p)}\setminus\{\mathcal{M}\}}r^{(p)}(\mathcal{M},\,\mathcal{M}')\\
 & \ \ \ \ \ -\sum_{\mathcal{M}\in\mathfrak{D}}\nu_{\mathfrak{D}}(\mathcal{M}){\bf h}(\mathcal{M})\sum_{\mathcal{M}'\in\mathscr{V}^{(p)}\setminus\{\mathcal{M}\}}\bm{h}(\mathcal{M}')r^{(p)}(\mathcal{M},\,\mathcal{M}')\\
 & =-\sum_{\mathcal{M}\in\mathfrak{D}}\nu_{\mathfrak{D}}(\mathcal{M}){\bf h}(\mathcal{M})\sum_{\mathcal{M}'\in\mathfrak{D}\setminus\{\mathcal{M}\}}r^{(p)}(\mathcal{M},\,\mathcal{M}')\left({\bf h}(\mathcal{M}')-{\bf h}(\mathcal{M})\right)\\
 & \ \ \ \ \ +\sum_{\mathcal{M}\in\mathfrak{D}}\nu_{\mathfrak{D}}(\mathcal{M}){\bf h}(\mathcal{M})^{2}\sum_{\mathcal{M}'\in\mathscr{V}^{(p)}\setminus\mathfrak{D}}r^{(p)}(\mathcal{M},\,\mathcal{M}')\\
 & =-\sum_{\mathcal{M}\in\mathfrak{D}}\nu_{\mathfrak{D}}(\mathcal{M}){\bf h}(\mathcal{M})\mathfrak{L}_{\mathfrak{D}}^{(p)}\bm{h}(\mathcal{M})+\sum_{\mathcal{M}\in\mathfrak{D}}\omega(\mathcal{M})\sum_{\mathcal{M}'\in\mathscr{V}^{(p)}\setminus\mathfrak{D}}r^{(p)}(\mathcal{M},\,\mathcal{M}')\,,
\end{aligned}
\]
which coincides with \eqref{e_58}. Therefore, by \eqref{e_func_max},
\begin{equation}
\lim_{\epsilon\to0}\theta_{\epsilon}^{(p)}\mathcal{I}_{\epsilon}(\mu_{\epsilon})=\lim_{\epsilon\to0}\theta_{\epsilon}^{(p)}\epsilon\int_{\mathbb{R}^{d}}\left|\nabla F_{\epsilon}\right|^{2}d\pi_{\epsilon}=\mathfrak{J}^{(p)}(\omega)\,.
\end{equation}

If $|\mathfrak{D}|=1$, say $\mathfrak{D}=\{\mathcal{M}^{(0)}\}$
and $\omega=\delta_{\mathcal{M}^{(0)}}$, then by \eqref{e_J_delta},
\begin{equation}
\mathfrak{J}^{(p)}(\omega)=\sum_{\mathcal{M}\in\mathscr{V}^{(p)}\setminus\{\mathcal{M}^{(0)}\}}r^{(p)}(\mathcal{M}^{(0)},\,\mathcal{M})\,.\label{e_512}
\end{equation}
Define
\[
\bm{g}(\mathcal{M})=\begin{cases}
\sqrt{\frac{\nu_{\star}}{\nu(\mathcal{M}^{(0)})}} & \mathcal{M}=\mathcal{M}^{(0)}\,,\\
0 & \mathcal{M}\in\mathscr{V}^{(p)}\setminus\{\mathcal{M}^{(0)}\}\,,
\end{cases}
\]
and define functions $G_{\epsilon}$, $F_{\epsilon}$, and the sequence
of measures $(\mu_{\epsilon})_{\epsilon}$ as above.
Then, $\mu_{\epsilon}\to\pi_{\mathcal{M}^{(0)}}$ and
Lemma \ref{l_limsup-0} gives
\[
\begin{aligned}\lim_{\epsilon\to0}\theta_{\epsilon}^{(p)}\epsilon\int_{\mathbb{R}^{d}}\left|\nabla F_{\epsilon}\right|^{2}d\pi_{\epsilon} & =\nu_{\star}^{-1}\nu(\mathcal{M}^{(0)}){\bf g}(\mathcal{M}^{(0)})^{2}\sum_{\mathcal{M}\in\mathscr{V}^{(p)}\setminus\{\mathcal{M}^{(0)}\}}r^{(p)}(\mathcal{M}^{(0)},\,\mathcal{M})\\
 & =\sum_{\mathcal{M}\in\mathscr{V}^{(p)}\setminus\{\mathcal{M}^{(0)}\}}r^{(p)}(\mathcal{M}^{(0)},\,\mathcal{M})\,,
\end{aligned}
\]
which, together with \eqref{e_func_max} and \eqref{e_512}, completes the proof.
\end{proof}

\subsection{Proof of Proposition \ref{p_Gamma}-(4)}
\begin{proof}[Proof of Proposition \ref{p_Gamma}-(4)]

$ $

\smallskip
\noindent $\Gamma-\liminf$.
\smallskip

We proceed by induction on $p$.

\smallskip
\noindent\textbf{Step 1.} $p=1$
\smallskip

Let $\mu\in\mathcal{P}(\mathbb{R}^{d})$ be a measure that is not a convex combination of $\delta_{\bm{m}}$, $\bm{m}\in\mathcal{M}_{0}$.
For any sequence $(\mu_{\epsilon})_{\epsilon>0}$
in $\mathcal{P}(\mathbb{R}^{d})$ such that $\mu_{\epsilon}\to\mu$
as $\epsilon\to0$, Proposition \ref{p_Gamma}-(3) yields
\[
\liminf_{\epsilon\to0}\,\mathcal{I}_{\epsilon}(\mu_{\epsilon}) \ge \mathcal{J}^{(0)}(\mu)>0\,,
\]
hence
\begin{equation}
\liminf_{\epsilon\to0}\,\theta_{\epsilon}^{(1)}\,\mathcal{I}_{\epsilon}(\mu_{\epsilon})\,=\,\infty=\,\mathcal{J}^{(1)}(\mu)\,.\label{e_liminf-1}
\end{equation}
If instead $\mu=\sum_{\bm{m}\in\mathcal{M}_{0}}\omega(\bm{m})\delta_{\bm{m}}$
for some $\omega\in\mathcal{P}(\mathcal{M}_{0})$, Lemma \ref{l_liminf} gives
\begin{equation}
\liminf_{\epsilon\to0}\,\theta_{\epsilon}^{(1)}\,\mathcal{I}_{\epsilon}(\mu_{\epsilon})\,\ge\,\mathfrak{J}^{(1)}(\omega)\,=\,\mathcal{J}^{(1)}(\mu)\,.\label{e_liminf-2}
\end{equation}
By the definition \eqref{e_J^p} of $\mathcal{J}^{(1)}$, together with
\eqref{e_liminf-1} and \eqref{e_liminf-2}, the sequence $(\theta_{\epsilon}^{(1)}\mathcal{I}_{\epsilon})_{\epsilon>0}$
satisfies Definition \ref{def_Gamma}-(1) with the limit $\mathcal{J}^{(1)}$.

\smallskip
\noindent\textbf{Step 2.} $p\in\llbracket2,\,\mathfrak{q}\rrbracket$
\smallskip

Fix $p\in\llbracket2,\,\mathfrak{q}\rrbracket$ and assume that for every $n<p$,
the sequence $(\theta_{\epsilon}^{(n)}\mathcal{I}_{\epsilon})_{\epsilon>0}$ satisfies Definition \ref{def_Gamma}-(1) with the limit $\mathcal{J}^{(n)}$.

Let $\mu\in\mathcal{P}(\mathbb{R}^{d})$ be a measure that cannot be written in the form
\begin{equation}
\mu\,=\,\sum_{\mathcal{M}\in\mathscr{V}^{(p)}}\omega(\mathcal{M})\,\pi_{\mathcal{M}}\label{e_liminf-3}
\end{equation}
for any $\omega\in\mathcal{P}(\mathscr{V}^{(p)})$. By Lemma \ref{l42}
and the induction hypothesis, for any sequence $(\mu_{\epsilon})_{\epsilon>0}$
in $\mathcal{P}(\mathbb{R}^{d})$ such that $\mu_{\epsilon}\to\mu$,
\[
\liminf_{\epsilon\to0}\theta_{\epsilon}^{(p-1)}\mathcal{I}_{\epsilon}(\mu_{\epsilon})\ge\mathcal{J}^{(p-1)}(\mu)>0\,,
\]
hence
\begin{equation}
\liminf_{\epsilon\to0}\theta_{\epsilon}^{(p)}\mathcal{I}_{\epsilon}(\mu_{\epsilon})=\infty\,.\label{e_liminf-4}
\end{equation}
If $\mu\in\mathcal{P}(\mathbb{R}^{d})$ is of the form \eqref{e_liminf-3},
Lemma \ref{l_liminf} gives
\begin{equation}
\liminf_{\epsilon\to0}\theta_{\epsilon}^{(p)}\mathcal{I}_{\epsilon}(\mu_{\epsilon})\ge\mathfrak{J}^{(p)}(\omega)=\mathcal{J}^{(p)}(\mu)\,.\label{e_liminf-5}
\end{equation}
Combining \eqref{e_J^p},
\eqref{e_liminf-4}, and \eqref{e_liminf-5} shows that $(\theta_{\epsilon}^{(p)}\mathcal{I}_{\epsilon})_{\epsilon>0}$
satisfies Definition \ref{def_Gamma}-(1) with the limit $\mathcal{J}^{(p)}$.

\smallskip
\noindent $\Gamma-\limsup$.
\smallskip

Fix $p\in\llbracket1,\,\mathfrak{q}\rrbracket$. If $\mu\in\mathcal{P}(\mathbb{R}^{d})$
is not of the form \eqref{e_liminf-3}
for any $\omega\in\mathcal{P}(\mathscr{V}^{(p)})$, then $\mathcal{J}^{(p)}(\mu)=\infty$ by the definition
\eqref{e_J^p} of $\mathcal{J}^{(p)}$,
so there is nothing to prove. Suppose that $\mathcal{J}^{(p)}(\mu)<\infty$
so that $\mu\in\mathcal{P}(\mathbb{R}^{d})$ is of the form \eqref{e_liminf-3}.
Decompose $\mathscr{V}^{(p)}$ as
\[
\mathscr{V}^{(p)}=\bigcup_{i=1}^{\mathfrak{l}}\mathfrak{D}_{i}
\]
where $\mathfrak{D}_{1},\,\dots,\,\mathfrak{D}_{\mathfrak{l}}$ are
the equivalence classes of $\{{\bf y}^{(p)}(t)\}_{t\ge0}$, and write
\[
\omega=\sum_{i=1}^{\mathfrak{l}}\omega(\mathfrak{D}_{i})\omega_{\mathfrak{D}_{i}}\,,
\]
where $\omega_{\mathfrak{D}_{i}}$ is the measure $\omega$ conditioned on $\mathfrak{D}_{i}$.

For each $i$, let $\mu_{\epsilon}^{i}\in\mathcal{P}(\mathbb{R}^{d})$
be the sequence provided by Lemma \ref{l_limsup} applied to $\omega_{\mathfrak{D}_i}$.
Define
\[
\mu_{\epsilon}=\sum_{i=1}^{\mathfrak{l}}\omega(\mathfrak{D}_{i})\mu_{\epsilon}^{i}\,.
\]
By convexity of $\mathcal{I}_{\epsilon}$ and  Lemma \ref{l_limsup},
\[
\limsup_{\epsilon\to0}\theta_{\epsilon}^{(p)}\mathcal{I}_{\epsilon}(\mu_{\epsilon})\le\sum_{i=1}^{\mathfrak{l}}\omega(\mathfrak{D}_{i})\limsup_{\epsilon\to0}\theta_{\epsilon}^{(p)}\mathcal{I}_{\epsilon}(\mu_{\epsilon}^{i})\le\sum_{i=1}^{\mathfrak{l}}\omega(\mathfrak{D}_{i})\mathfrak{J}^{(p)}(\omega_{\mathfrak{D}_{i}})\,.
\]
Finally, by Lemma \ref{l_DV_equiv-2}, 
\[
\sum_{i=1}^{\mathfrak{l}}\omega(\mathfrak{D}_{i})\mathfrak{J}^{(p)}(\omega_{\mathfrak{D}_{i}})=\mathfrak{J}^{(p)}(\omega)\,,
\]
which completes the proof.
\end{proof}

\section{\label{sec_tree_rig}Tree structure}

In this section, we present the rigorous definition of the tree
structure informally introduced in Section \ref{subsec_tree}.

\subsection{\label{subsec: MC1}The first layer}

We first recall several notions related to the energy landscape
induced by $U$ introduced in \cite[Section 4.1]{LLS-2nd}. Note that we
consider here the reversible case in which the drift $\bm{b}$ is equal
to $-\nabla U$.

\begin{itemize}
\item For each pair $\boldsymbol{m}'\neq\boldsymbol{m}''\in\mathcal{M}_{0}$,
denote by $\Theta(\boldsymbol{m}',\,\boldsymbol{m}'')$ the \emph{communication
height} between $\boldsymbol{m}'$ and $\boldsymbol{m}''$: 
\[
{\color{blue}\Theta(\boldsymbol{m}',\,\boldsymbol{m}''}\mathclose{\color{blue})}:=\inf_{\substack{\boldsymbol{z}:[0\,1]\rightarrow\mathbb{R}^{d}}
}\max_{t\in[0,\,1]}U(\boldsymbol{z}(t))\,,
\]
where the infimum is taken over all continuous paths $\boldsymbol{z}(\cdot)$
such that $\boldsymbol{z}(0)=\boldsymbol{m}'$ and $\boldsymbol{z}(1)=\boldsymbol{m}''$.
Clearly, $\Theta(\boldsymbol{m}',\,\boldsymbol{m}'')=\Theta(\boldsymbol{m}'',\,\boldsymbol{m}')$.
\item For $\bm{c}_{1},\,\bm{c}_{2}\in\mathcal{C}_{0}$, we write \textcolor{blue}{$\bm{c}_{1}\curvearrowright\bm{c}_{2}$}
if there exists a heteroclinic orbit connecting $\bm{c}_{1}$ to $\bm{c}_{2}$.
\item For each saddle point $\boldsymbol{\sigma}\in\mathcal{S}_{0}$,
the matrix $(\nabla^{2}U)(\boldsymbol{\sigma})$ has one negative
eigenvalue, represented by
${\color{blue} -\lambda_{1}^{\bm{\sigma}}}<0$. For
$\boldsymbol{\sigma}\in\mathcal{S}_{0}$, let the weight
$\omega(\boldsymbol{\sigma})$, the so-called \textit{Eyring--Kramers
constant}, be defined by
\[
{\color{blue}\omega(\boldsymbol{\sigma}}\mathclose{\color{blue})}:=\frac{\lambda_{1}^{\bm{\sigma}}}{2\pi\sqrt{-\,\det(\nabla^{2}U)(\boldsymbol{\sigma})}}\,.
\]
\end{itemize}

Let
${\color{blue} \mathscr{V}^{(1)}}:=\left\{
\{\bm{m}\}:\bm{m}\in\mathcal{M}_{0}\right\} $.  For
$\bm{m}\in\mathscr{V}^{(1)}$, denote by $\Xi(\boldsymbol{m})$ the
difference between the
height which separates $\bm{m}$ from local minima of lower or equal potential value, and the height
of $\bm{m}$:
\begin{equation}
{\color{blue}\Xi(\boldsymbol{m}}\mathclose{\color{blue})}:=\inf\left\{ \Theta(\boldsymbol{m},\,\boldsymbol{m}'):\boldsymbol{m}'\in\mathcal{M}_{0}\setminus\{\boldsymbol{m}\}\text{ such that }U(\boldsymbol{m}')\le U(\boldsymbol{m})\right\} -U(\bm{m})\label{e_def_Xi}\,.
\end{equation}
Let $d^{(1)}$ be the smallest height difference: 
\[
{\color{blue} d^{(1)}}
:=\min_{\bm{m}\in\mathscr{V}^{(1)}}\Xi(\bm{m})\,.
\]
Since, by assumption, $|\mathscr{V}^{(1)}|\ge2$, there exists
$\bm{m}\in\mathscr{V}^{(1)}$ such that $\Xi(\bm{m})<\infty$, so that
$d^{(1)}<\infty$.

For $\bm{m}\in\mathscr{V}^{(1)}$, let
$\mathcal{S}^{(1)}(\boldsymbol{m})$ be the set of saddle points
connected to the local minimum $\boldsymbol{m}$:
\begin{gather*}
{\color{blue}\mathcal{S}^{(1)}(\boldsymbol{m}}\mathclose{\color{blue})}:=\big\{\,\bm{\sigma}\in\mathcal{S}_{0}:\boldsymbol{\sigma}\curvearrowright\boldsymbol{m}\ ,\ \ U(\bm{\sigma})=U(\bm{m})+\Xi(\bm{m})\,\big\}\,.
\end{gather*}
Denote by $\mathcal{S}(\boldsymbol{m},\boldsymbol{m}')$, $\boldsymbol{m}'\neq\boldsymbol{m}$,
the set of saddle points which separate $\boldsymbol{m}$ from $\boldsymbol{m}'$:
\[
{\color{blue}\mathcal{S}(\boldsymbol{m},\,\boldsymbol{m}'}\mathclose{\color{blue})}:=\big\{\,\boldsymbol{\sigma}\in\mathcal{S}^{(1)}(\boldsymbol{m}):\boldsymbol{\sigma}\curvearrowright\boldsymbol{m}\;,\;\;\boldsymbol{\sigma}\curvearrowright\boldsymbol{m}'\,\big\}\,.
\]
Note that we may have
$\mathcal{S}(\bm{m},\,\bm{m}')\ne\mathcal{S}(\bm{m}',\,\bm{m})$ or
$\mathcal{S}(\boldsymbol{m},\,\boldsymbol{m}')=\varnothing$ for some
$\bm{m},\,\bm{m}'\in\mathscr{V}^{(1)}$.  Mind that if
$\Xi(\bm m) = d^{(1)}$, and
$\mathcal{S}  (\boldsymbol{m},\,\boldsymbol{m}') \neq \varnothing$ for some
$\bm m'\in \mathcal{ M}_0$, then $U(\bm m) \ge U(\bm m')$.

Denote by $\omega(\boldsymbol{m},\boldsymbol{m}')$ the sum of the
Eyring--Kramers constants of the saddle points in
$\mathcal{S}(\boldsymbol{m},\boldsymbol{m}')$:
\[
{\color{blue}\omega(\boldsymbol{m},\,\boldsymbol{m}'}\mathclose{\color{blue})} :=\sum_{\boldsymbol{\sigma}\in\mathcal{S}(\boldsymbol{m},\boldsymbol{m}')}\,\omega(\boldsymbol{\sigma})\;.
\]
Recall the definition of the weight $\nu(\bm m)$,
$\bm m\in \mathcal{ M}_0$, given in \eqref{e_def_nu}. For
$\bm{m},\,\bm{m}'\in\mathscr{V}^{(1)}$, define
\[
{\color{blue}r^{(1)}(\{\boldsymbol{m} \},\, \{\boldsymbol{m}'\}}
\mathclose{\color{blue})}:=\begin{cases}
\frac{1}{\nu(\boldsymbol{m})}\,\omega(\boldsymbol{m},\boldsymbol{m}') & \bm{m}\ne\bm{m}' \ \text{and} \ \Xi(\boldsymbol{m})=d^{(1)}\,,\\
0 & \text{otherwise}\,,
\end{cases}
\]
and let \textcolor{blue}{$\{{\bf y}^{(1)}(t)\}_{t\ge0}$} be the
$\mathscr{V}^{(1)}$-valued Markov chain with jump rates
$r^{(1)}:\mathscr{V}^{(1)}\times\mathscr{V}^{(1)}\to[0,\,\infty)$.  If
$\{{\bf y}^{(1)}(t)\}_{t\ge0}$ has only one closed irreducible class, the
construction is over.

\subsection{\label{subsec: MC2}The upper levels}
First, recall from Section \ref{subsec521} that a nonempty set
$\mathcal M\subset\mathcal M_0$ is called simple if $U$ is constant
on $\mathcal M$. For such a set, $U(\mathcal M)$ denotes this common
value.
We now recall several notions introduced in \cite[Section 4.2]{LLS-2nd}.
\begin{itemize}
\item For two disjoint non-empty subsets $\mathcal{M}$ and $\mathcal{M}'$
of $\mathcal{M}_{0}$, let $\Theta(\mathcal{M},\,\mathcal{M}')$ be
the \textit{communication height} between the two sets: 
\[
{\color{blue}\Theta(\mathcal{M},\,\mathcal{M}'}\mathclose{\color{blue})}:=\min_{\boldsymbol{m}\in\mathcal{M},\,\boldsymbol{m}'\in\mathcal{M}'}\Theta(\boldsymbol{m},\,\boldsymbol{m}')\,,
\]
with the convention that $\Theta(\mathcal{\mathcal{M}},\,\varnothing)=+\infty$.
\item For a simple set $\mathcal{\mathcal{M}}\subset\mathcal{M}_{0}$,
denote by $\widetilde{\mathcal{M}}$ the set of local minima of $U$
which do not belong to $\mathcal{M}$ and whose energy is lower than or equal to that of $\mathcal{M}$: 
\[
{\color{blue}\widetilde{\mathcal{M}}}:=\big\{\,\boldsymbol{m}\in\mathcal{M}_{0}\setminus\mathcal{M}:U(\boldsymbol{m})\le U(\mathcal{M})\,\big\}\,.
\]
Note that $\widetilde{\mathcal{\mathcal{M}}}=\varnothing$ if and
only if $\mathcal{M}$ contains all the global minima of $U$.\smallskip{}
\item For a saddle point $\boldsymbol{\sigma}\in\mathcal{S}_{0}$ and local
minimum $\boldsymbol{m}\in\mathcal{M}_{0}$, we write\textit{\textcolor{blue}{{}
$\boldsymbol{\sigma}\rightsquigarrow\boldsymbol{m}$}} if $\boldsymbol{\sigma}\curvearrowright\boldsymbol{m}$
or if there exist $n\ge1$, $\boldsymbol{\sigma}_{1},\,\dots,\,\boldsymbol{\sigma}_{n}\in\mathcal{S}_{0}$
and $\boldsymbol{m}_{1}\,\,\dots,\,\boldsymbol{m}_{n}\in\mathcal{M}_{0}$
such that 
\[
\max\{U(\boldsymbol{\sigma}_{1}),\,\dots,\,U(\boldsymbol{\sigma}_{n})\,\}<U(\boldsymbol{\sigma})\;\;\;\text{and\;\;\;}\boldsymbol{\sigma}\curvearrowright\boldsymbol{m}_{1}\curvearrowleft\boldsymbol{\sigma}_{1}\curvearrowright\cdots\curvearrowright\boldsymbol{m}_{n}\curvearrowleft\boldsymbol{\sigma}_{n}\curvearrowright\boldsymbol{m}\,.
\]
For $\mathcal{\mathcal{M}}\subset\mathcal{M}_{0}$, write ${\color{blue}\boldsymbol{\sigma}\rightsquigarrow\mathcal{\mathcal{M}}}$
and \textcolor{blue}{$\bm{\sigma}\curvearrowright\mathcal{M}$} if
for some $\boldsymbol{m}\in\mathcal{\mathcal{M}}$, $\boldsymbol{\sigma}\rightsquigarrow\boldsymbol{m}$
and $\bm{\sigma}\curvearrowright\bm{m}$ , respectively. \smallskip{}
\item Fix a non-empty simple set $\mathcal{\mathcal{M}}\subset\mathcal{M}_{0}$
such that $\widetilde{\mathcal{M}}\neq\varnothing$. For a set $\mathcal{\mathcal{M}}'\subset\mathcal{M}_{0}$
such that $\mathcal{M}'\cap\mathcal{M}=\varnothing$, we write\textit{\textcolor{blue}{{}
$\mathcal{M}\rightarrow\mathcal{M}'$}} if there exists $\boldsymbol{\sigma}\in\mathcal{S}_{0}$
such that
\begin{equation}
U(\boldsymbol{\sigma})=\Theta(\mathcal{M},\,\widetilde{\mathcal{M}})=\Theta(\mathcal{M},\,\mathcal{M}')\;\;\text{and}\;\;\mathcal{M}'\,\curvearrowleft\,\boldsymbol{\sigma}\,\rightsquigarrow\,\mathcal{M}\,.\label{eq:con_gate}
\end{equation}
To emphasize the saddle point $\boldsymbol{\sigma}$ between $\mathcal{M}$
and $\mathcal{M}'$ we sometimes write ${\color{blue}\mathcal{M}\rightarrow_{\boldsymbol{\sigma}}\mathcal{M}'}$.
\smallskip{}
\item Denote by $\mathcal{S}(\mathcal{M},\,\mathcal{M}')$ the set of saddle
points $\boldsymbol{\sigma}\in\mathcal{S}_{0}$ satisfying \eqref{eq:con_gate},
\begin{equation}
{\color{blue}\mathcal{S}(\mathcal{M},\,\mathcal{M}'}\mathclose{\color{blue})}:=\{\,\boldsymbol{\sigma}\in\mathcal{S}_{0}:\mathcal{M}\to_{\boldsymbol{\sigma}}\mathcal{M}'\,\}\,.\label{e_SMM}
\end{equation}
The set $\mathcal{S}(\mathcal{M},\,\mathcal{M}')$ represents the
collection of lowest connection points which separate $\mathcal{M}$
from $\mathcal{M}'$. Note that we may have $\mathcal{S}(\mathcal{M},\,\mathcal{M}')\ne\mathcal{S}(\mathcal{M}',\,\mathcal{M})$
or $\mathcal{S}(\mathcal{M},\,\mathcal{M}')=\varnothing$ for some
$\mathcal{M},\,\mathcal{M}'\subset\mathcal{M}_0$.
\end{itemize}

Recall the definition of $\Lambda^{(n)}$, $n\ge1$, introduced at the
beginning of Section \ref{subsec_tree}. Fix $k\ge1$ and suppose that
the quintuples $\Lambda^{(n)}$, $n\in\llbracket1,\,k\rrbracket$, have
been defined. Denote by \textcolor{blue}{$\mathfrak{n}_{k}$} the
number of closed irreducible classes of the Markov chain $\{{\bf y}^{(k)}(t)\}_{t\ge0}$. If
$\mathfrak{n}_{k}=1$, the construction is over.  Otherwise, denote by
\textcolor{blue}{$\mathscr{R}_{1}^{(k)},\,\dots,\,\mathscr{R}_{\mathfrak{n}_{k}}^{(k)}$}
the closed irreducible classes and by
\textcolor{blue}{$\mathscr{T}^{(k)}$} the collection of
transient states of the Markov chain $\{{\bf y}^{(k)}(t)\}_{t\ge0}$, respectively.

Recall from \eqref{e_M_p+1} and \eqref{e_V_p+1} the definitions of
$\mathcal{M}_{i}^{(k+1)}$, $1\le i\le \mf n_k$, $\mathscr{V}^{(k+1)}$,
$\mathscr{N}^{(k+1)}$, and $\mathscr{S}^{(k+1)}$. By Proposition
\ref{p: tree}-(2) below, all $\mathcal{M}\in\mathscr{S}^{(k+1)}$ are
simple. For $\mathcal{M}\in\mathscr{V}^{(k+1)}$, define
\begin{equation}
{\color{blue}\Xi(\mathcal{M}}\mathclose{\color{blue})}:=\Theta(\mathcal{M},\,\widetilde{\mathcal{M}})-U(\mathcal{M})\ \ \text{and}\ \ {\color{blue}d^{(k+1)}}:=\min_{\mathcal{M}\in\mathscr{V}^{(k+1)}}\Xi(\mathcal{M})\,.\label{e_Xi}
\end{equation}
Since $\mathfrak{n}_{k}\ge2$, there exists
$\mathcal{M}\in\mathscr{V}^{(k+1)}$ such that
$\Xi(\mathcal{M})<\infty$ so that $d^{(k+1)}<\infty$.

Denote by $\widehat{r}^{(k)}:\mathscr{S}^{(k)}\times\mathscr{S}^{(k)}\to[0,\,\infty)$
the jump rates of the $\mathscr{S}^{(k)}$-valued Markov chain $\{\widehat{\mathbf{y}}^{(k)}(t)\}_{t\ge0}$.
Since $\mathscr{S}^{(k+1)}=\mathscr{V}^{(k+1)}\cup\mathscr{N}^{(k+1)}$,
we can divide the definition of the jump rate\textcolor{blue}{\emph{
}}${\color{blue}\widehat{r}^{(k+1)}}:\mathscr{S}^{(k+1)}\times\mathscr{S}^{(k+1)}\to[0,\,\infty)$
of $\{\widehat{\mathbf{y}}^{(k+1)}(t)\}_{t\ge0}$ into four cases: 
\begin{itemize}
\item {[}\textbf{Case }1: $\mathcal{M}=\mathcal{M}'\in\mathscr{S}^{(k+1)}${]}
We set $\widehat{r}^{(k+1)}(\mathcal{M},\,\mathcal{M}')=0$.
\item {[}\textbf{Case }2: $\mathcal{M}\in\mathscr{N}^{(k+1)}$ and $\mathcal{M}'\in\mathscr{N}^{(k+1)}${]}
Since $\mathcal{M},\,\mathcal{M}'\in\mathscr{S}^{(k)}$ , we set 
\begin{equation}
\widehat{r}^{(k+1)}(\mathcal{M},\,\mathcal{M}'):=\widehat{r}^{(k)}(\mathcal{M},\,\mathcal{M}')\,.\label{e_rate-1}
\end{equation}
\item {[}\textbf{Case }3: $\mathcal{M}\in\mathscr{N}^{(k+1)}$ and $\mathcal{M}'\in\mathscr{V}^{(k+1)}${]}
Since $\mathcal{M}\in\mathscr{S}^{(k)}$ and since $\mathcal{M}'$
is the union of elements (possibly just one) in $\mathscr{V}^{(k)}$,
we set 
\begin{equation}
\widehat{r}^{(k+1)}(\mathcal{M},\,\mathcal{M}'):=\sum_{\mathcal{M}''\in\mathscr{R}^{(k)}(\mathcal{M}')}\widehat{r}^{(k)}(\mathcal{M},\,\mathcal{M}'')\,,\label{e_rate-2}
\end{equation}
where $\mathscr{R}^{(k)}(\mathcal{M}')$, $\mathcal{M}'\in\mathscr{V}^{(k+1)}$,
is the closed irreducible class of $\{{\bf y}^{(k)}(t)\}_{t\ge0}$ such that
$\mathcal{M}'=\bigcup_{\mathcal{M}''\in\mathscr{R}^{(k)}(\mathcal{M}')}\mathcal{M}''\,.$

\item {[}\textbf{Case }4: $\mathcal{M}\in\mathscr{V}^{(k+1)}$ and $\mathcal{M}'\in\mathscr{S}^{(k+1)}${]}
Let 
\[
{\color{blue} \omega(\mathcal{M},\,\mathcal{M}') }
:=\sum_{\boldsymbol{\sigma}\in\mathcal{S}(\mathcal{M},\,\mathcal{M}')}\omega(\boldsymbol{\sigma})\;.
\]
It is understood here that $\omega(\mathcal{M},\,\mathcal{M}')=0$
if the set $\mathcal{S}(\mathcal{M},\,\mathcal{M}')$ is empty. 
Set
\begin{equation}
\widehat{r}^{(k+1)}(\mathcal{M},\,\mathcal{M}'):= \begin{cases}
\frac{1}{\nu(\mathcal{M})}\,\omega(\mathcal{M},\,\mathcal{M}') & \Xi(\mathcal{M})=d^{(k+1)} \,,\\
0 & \text{otherwise} \,.
\end{cases}
\label{e_rate-3}
\end{equation}
\end{itemize}
where $\nu(\mathcal{M}) $ has been introduced in \eqref{e_def_nu}.

Define \textcolor{blue}{$\{\widehat{\mathbf{y}}^{(k+1)}(t)\}_{t\ge0}$}
as the $\mathscr{S}^{(k+1)}$-valued, continuous-time Markov chain with
jump rates $\widehat{r}^{(k+1)}:\mathscr{S}^{(k+1)}\times\mathscr{S}^{(k+1)}\to[0,\,\infty)$. By \cite[Lemma
5.8]{LLS-2nd}, all recurrent classes of
$\{\widehat{\mathbf{y}}^{(k+1)}(t)\}_{t\ge0}$ contain an element of
$\mathscr{V}^{(k+1)}$.  Therefore, by \cite[Lemma B.1]{LLS-2nd} and
\cite[display (B.1)]{LLS-2nd}, the trace process of
$\{\widehat{\mathbf{y}}^{(k+1)}(t)\}_{t\ge0}$ on $\mathscr{V}^{(k+1)}$
is well defined (cf. \cite[Appendix B]{LLS-2nd}).  Denote by
\textcolor{blue}{$\{{\bf y}^{(k+1)}(t)\}_{t\ge0}$} the trace
process. This completes the construction of the quintuples
$\Lambda^{(1)},\,\dots,\,\Lambda^{(k+1)}$.

If $\mathfrak{n}_{k+1}$, the number of closed irreducible classes of
$\{{\bf y}^{(k+1)}(t)\}_{t\ge0}$, is $1$, the construction is over,
and $\mathfrak{q}=k+1$. If $\mathfrak{n}_{k+1}>1$, we add a new layer
as in this subsection.

We conclude this section with important properties of the tree
structure derived in \cite{LLS-2nd}.
For any set or element $A$, ${\color{blue} H_A}$ denotes the first
hitting time of $A$ for a given process.

\begin{prop}
\label{p: tree}
We have the following.
\begin{enumerate}
\item If $\mathfrak{n}_{n}>1$, $\mathfrak{n}_{n}>\mathfrak{n}_{n+1}$.
In particular, there exists $\mathfrak{q}\in\mathbb{N}$ such that
$\mathfrak{n}_{1}>\cdots>\mathfrak{n}_{\mathfrak{q}}=1$.
\item For all $n\in\llbracket1,\,\mathfrak{q}\rrbracket$ and $\mathcal{M}\in\mathscr{S}^{(n)}$,
$\mathcal{M}$ is simple.
\item $0<d^{(1)}<\cdots<d^{(\mathfrak{q})}<\infty$.
\item For all $n\in\llbracket1,\,\mathfrak{q}\rrbracket$ and $\mathcal{M},\,\mathcal{M}'\in\mathscr{S}^{(n)}$,
$\widehat{r}^{(n)}(\mathcal{M},\,\mathcal{M}')>0$ if and
only if $\Xi(\mathcal{M})\le d^{(n)}$ and
$\mathcal{M}\to\mathcal{M}'$.

\item Denote by
\textcolor{blue}{$\widehat{\mathcal{Q}}_{\mathcal{M}}^{(p)}$},
$1\le p\le \mf q$, $\mathcal{M}\in\mathscr{S}^{(p)}$, the law of the
Markov chain $\{\widehat{{\bf y}}^{(p)}(t)\}_{t\ge0}$ starting from
$\mathcal{M}$. For all $n\in\llbracket1,\,\mathfrak{q}\rrbracket$,
$\mathcal{M}\in\mathscr{N}^{(n)}$, and
$\mathcal{M}'\in\mathscr{V}^{(n)}$,
\[
\lim_{\epsilon\to0}\sup_{\bm{x}\in\mathcal{E}(\mathcal{M})}\bigg|\mathbb{P}_{\bm{x}}^{\epsilon}\Big[H_{\mathcal{E}^{(n)}}=H_{\mathcal{E}(\mathcal{M}')}\Big]-\widehat{\mathcal{Q}}_{\mathcal{M}}^{(n)}\Big[H_{\mathscr{V}^{(n)}}=H_{\mathcal{M}'}\Big]\bigg|=0\,,
\]
where $\mathcal{E}^{(n)}$, $\mathcal{E}(\mathcal{M})$,
$\mathcal{E}(\mathcal{M}')$ are the metastable sets defined in
\eqref{e_E_M}.
\end{enumerate}
\end{prop}

\begin{proof}
The first property is \cite[Theorem 4.7-(3)]{LLS-2nd}. The next three
properties are postulates $\mathfrak{P}_{1}$, $\mathfrak{P}_{2}$, and
$\mathfrak{P}_{3}$ defined in \cite[Definition 4.4]{LLS-2nd}. It is
proved in \cite[Corollary 4.8]{LLS-2nd} that conditions
$\mathfrak{P}_{1}$, $\mathfrak{P}_{2}$, and $\mathfrak{P}_{3}$
hold. The last property is Condition $\mathfrak{H}^{(n)}$ introduced
in \cite[Definition 3.10]{LLS-2nd} which was proven to be true in
\cite[Section 3]{LLS-2nd} (cf. \cite[Figure 3.1]{LLS-2nd}).
\end{proof}
The following result is \cite[Proposition 4.9]{LLS-2nd}.
\begin{prop}
\label{p_char}Let $n\in\llbracket1,\,\mathfrak{q}\rrbracket$ and
$\mathcal{M}\in\mathscr{S}^{(n)}$.
\[
\begin{cases}
\Xi(\mathcal{M})<d^{(n)} & \text{iff}\ \mathcal{M}\in\mathscr{N}^{(n)}\,,\\
\Xi(\mathcal{M})=d^{(n)} & \text{iff}\ \mathcal{M}\in\mathscr{V}^{(n)}\ \text{and}\ \mathcal{M}\ \text{is not an absorbing state of}\ {\bf y}^{(n)}\,,\\
\Xi(\mathcal{M})>d^{(n)} & \text{iff}\ \mathcal{M}\in\mathscr{V}^{(n)}\ \text{and}\ \mathcal{M}\ \text{is an absorbing state of}\ {\bf y}^{(n)}\,.
\end{cases}
\]
\end{prop}

\section{\label{sec_pf_equi_pot}Proof of Proposition \ref{p_test}}

In this section, we prove Proposition \ref{p_test}.  For each
$p\in\llbracket1,\,\mathfrak{q}\rrbracket$, and equivalence class
$\mathfrak{D}\subset\mathscr{V}^{(p)}$ of the Markov chain
$\{{\bf y}^{(p)}(t)\}_{t\ge0}$, we construct the sequences
$(h_{\mathcal{M}}^{\epsilon})_{\epsilon>0}$,
$\mathcal{M}\in\mathfrak{D}$, of functions
$h_{\mathcal{M}}^{\epsilon}:\mathbb{R}^{d}\to[0,\,1]$ satisfying the
conditions of Proposition \ref{p_test}.

Fix $p\in\llbracket1,\,\mathfrak{q}\rrbracket$ and an equivalence
class $\mathfrak{D}\subset\mathscr{V}^{(p)}$ of the limiting chain
$\{{\bf y}^{(p)}(t)\}_{t\ge0}$. Let
${\color{blue}\widehat{\mathfrak{D}}}\subset\mathscr{S}^{(p)}$ be the
equivalence class of $\{\widehat{{\bf y}}^{(p)}(t)\}_{t\ge0}$
containing $\mathfrak{D}$, so that
$\mathfrak{D}=\widehat{\mathfrak{D}}\cap\mathscr{V}^{(p)}$.  We divide
the proof into two cases, depending on whether $\mathfrak{D}$ contains
an absorbing state or not.

\subsection{Equivalence classes formed by an absorbing state}

In this subsection, suppose that $\mathcal{M}_{1}\in\mathfrak{D}$ for
some absorbing state $\mathcal{M}_{1}\in\mathscr{V}^{(p)}$ of
$\{{\bf y}^{(p)}(t)\}_{t\ge0}$, i.e.,
$\mathfrak{D}=\{\mathcal{M}_{1}\}$.  We start by recalling several
notions introduced in \cite{LLS-2nd}.

\begin{itemize}
\item For $\mathcal{A}\subset\mathbb{R}^{d}$, define
\[
{\color{blue}\mathcal{M}^{*}(\mathcal{A}}\mathclose{\color{blue})}:=\{\bm{m}\in\mathcal{M}_{0}\cap\mathcal{A}:U(\bm{m})=\min_{\bm{x}\in\mathcal{A}}U(\bm{x})\}\,.
\]
\item For $p\in\llbracket1,\,\mathfrak{q}\rrbracket$ and $\mathcal{A}\subset\mathbb{R}^{d}$,
define
\[
\begin{aligned}
{\color{blue}\mathscr{V}^{(p)}(\mathcal{A}}\mathclose{\color{blue})} & :=\{\mathcal{M}\in\mathscr{V}^{(p)}:\mathcal{M}\subset\mathcal{A}\}\,,\\
{\color{blue}\mathscr{N}^{(p)}(\mathcal{A}}\mathclose{\color{blue})} & :=\{\mathcal{M}\in\mathscr{N}^{(p)}:\mathcal{M}\subset\mathcal{A}\}\,,\\
{\color{blue}\mathscr{S}^{(p)}(\mathcal{A}}\mathclose{\color{blue})} & :=\{\mathcal{M}\in\mathscr{S}^{(p)}:\mathcal{M}\subset\mathcal{A}\}\,.
\end{aligned}
\]
\end{itemize}
The next lemma shows the existence of the test function satisfying the
conditions in Proposition \ref{p_test} when $\mathfrak{D}$ contains
(and therefore consists of) an absorbing state.

\begin{lem}
\label{l_test_absobing}Suppose that $\mathcal{M}_{1}\in\mathscr{V}^{(p)}$
is an absorbing state of $\{{\bf y}^{(p)}(t)\}_{t\ge0}$. Then, there
exists a smooth function $h_{\mathcal{M}_{1}}:\mathbb{R}^{d}\to\mathbb{R}$
satisfying the following conditions.
\begin{enumerate}
\item $0\le h_{\mathcal{M}_{1}}\le1$ and $h_{\mathcal{M}_{1}}(\bm{x})=1$
for $\bm{x}\in\mathcal{E}(\mathcal{M}_{1})$.
\item $\lim_{\epsilon\to0}e^{U(\mathcal{M}_{1})/\epsilon}\int_{\mathbb{R}^{d}\setminus\mathcal{E}(\mathcal{M}_{1})}(h_{\mathcal{M}_{1}})^{2}\,d\pi_{\epsilon}=0$.
\item $\lim_{\epsilon\to0}\,e^{U(\mathcal{M}_{1})/\epsilon}\theta_{\epsilon}^{(p)}\epsilon\int_{\mathbb{R}^{d}}\left|\nabla h_{\mathcal{M}_{1}}\right|^{2}d\pi_{\epsilon}=0$.
\end{enumerate}
\end{lem}

\begin{proof}
For $b\ge0$, denote by $\mathcal{A}_{b}$ the connected component
of $\{U<U(\mathcal{M}_{1})+d^{(p)}+b\}$ containing $\mathcal{M}_{1}$.
By the proof of \cite[Lemma 10.2]{LLS-2nd}, there exists $a>0$ such
that $4a<\Xi(\mathcal{M}_{1})-d^{(p)}$, $\mathcal{A}_{b}$ is well
defined for $b\in[0,\,4a]$, and $\mathcal{M}_{1}=\mathcal{M}^{*}(\mathcal{A}_{4a})$.
Take $a>0$ small enough so that there is no critical point $\bm{c}\in\mathcal{C}_{0}$
such that $U(\bm{c})\in(U(\mathcal{M}_{1})+d^{(p)},U(\mathcal{M}_{1})+d^{(p)}+4a)$.
By \cite[Lemma A.14]{LLS-2nd}, 
\begin{equation}
\mathcal{M}_{0}\cap\mathcal{A}_{4a}=\mathcal{M}_{0}
\cap\mathcal{A}_{a}\,.\label{e_l71-1}
\end{equation}

We first claim that 
\begin{equation}
U(\bm{x})\ge U(\mathcal{M}_{1})+d^{(p)}+a\ \ \text{for all}\ \bm{x}\in\mathcal{A}_{4a}\setminus\mathcal{A}_{a}\,.\label{e_l71-2}
\end{equation}
Suppose that there exists
$\bm{x}_{0}\in\mathcal{A}_{4a}\setminus\mathcal{A}_{a}$ satisfying
$U(\bm{x}_{0})<U(\mathcal{M}_{1})+d^{(p)}+a$. Let $\mathcal{H}$ be the
connected component of $\{U<U(\mathcal{M}_{1})+d^{(p)}+a\}$ containing
$\bm{x}_{0}$. Since $\bm{x}_{0}\in\mathcal{A}_{4a}$,
$\mathcal{H}\subset\mathcal{A}_{4a}$, and since
$\bm{x}_{0}\not\in\mathcal{A}_{a}$,
$\mathcal{H}\cap\mathcal{A}_{a}=\varnothing$.  As $\mathcal{H}$ is
a level set, there exists a local minimum
$\bm{m}_{0}\in\mathcal{H}\cap\mathcal{M}_{0}$ so that
$(\mathcal{A}_{4a}\setminus\mathcal{A}_{a})\cap\mathcal{M}_{0}\ne\varnothing$,
which contradicts \eqref{e_l71-1}. Therefore, \eqref{e_l71-2}
holds.

Since $\mathcal{M}_{1}=\mathcal{M}^{*}(\mathcal{A}_{4a})$, there
exists $c_{0}>0$ such that
\begin{equation}
U(\bm{x})\ge U(\mathcal{M}_{1})+c_{0}\ \ \text{for all}\ \bm{x}\in\mathcal{A}_{4a}\setminus\mathcal{E}(\mathcal{M}_{1})\,.\label{e_l71-3}
\end{equation}
 Moreover, there exists a smooth function $h_{\mathcal{M}_{1}}:\mathbb{R}^{d}\to\mathbb{R}$
independent of $\epsilon>0$ such that
\begin{itemize}
\item $0\le h_{\mathcal{M}_{1}}(\bm{x})\le1$ for $\bm{x}\in\mathbb{R}^{d}$,
\item $h_{\mathcal{M}_{1}}(\bm{x})=1$ for $\bm{x}\in\mathcal{A}_{2a}$,
and
\item $h_{\mathcal{M}_{1}}(\bm{x})=0$ for $\bm{x}\in(\mathcal{A}_{4a})^{c}$.
\end{itemize}

We claim that the function
$h_{\mathcal{M}_{1}}:\mathbb{R}^{d}\to\mathbb{R}$ satisfies the
conditions of the lemma. The first condition is obvious from the
construction. By \eqref{e: growth}, $\mathcal{A}_{4a}$ is
bounded. Therefore, by \eqref{e_l71-3},
$\pi_{\epsilon}\left(\mathcal{A}_{4a}\setminus\mathcal{E}(\mathcal{M}_{1})\right)\le
C_{1}\epsilon^{-d/2}e^{-[U(\mathcal{M}_{1})+c_{0}]/\epsilon}$ for some finite constant
$C_{1}>0$. Since
\[
e^{U(\mathcal{M}_{1})/\epsilon}\int_{\mathbb{R}^{d}\setminus\mathcal{E}(\mathcal{M}_{1})}(h_{\mathcal{M}_{1}})^{2}\,d\pi_{\epsilon}\le e^{U(\mathcal{M}_{1})/\epsilon}\pi_{\epsilon}\left(\mathcal{A}_{4a}\setminus\mathcal{E}(\mathcal{M}_{1})\right)\,,
\]
$h_{\mathcal{M}_{1}}$ satisfies the second condition. As $\mathcal{A}_{4a}$
is bounded, by \eqref{e_l71-2}, $\pi_{\epsilon}(\mathcal{A}_{4a}\setminus\mathcal{A}_{a})\le C_{2}\epsilon^{-d/2} e^{-[U(\mathcal{M}_{1})+d^{(p)}+a]/\epsilon}$
for some finite constant $C_{2}>0$. Since $\nabla h_{\mathcal{M}_{1}}$ is uniformly
bounded, $\nabla h_{\mathcal{M}_{1}}(\bm{x})=0$ for $\bm{x}\in(\mathcal{A}_{4a}\setminus\mathcal{A}_{a})^{c}$,
and $\mathcal{A}_{4a}$ is bounded,
\[
\begin{aligned}e^{U(\mathcal{M}_{1})/\epsilon}\theta_{\epsilon}^{(p)}\epsilon\int_{\mathbb{R}^{d}}\left|\nabla h_{\mathcal{M}_{1}}\right|^{2}d\pi_{\epsilon} & \le\left(\|\nabla h_{\mathcal{M}_{1}}\|_{L^{\infty}(\mathbb{R}^{d})}\right)^{2}e^{U(\mathcal{M}_{1})/\epsilon}\theta_{\epsilon}^{(p)}\epsilon\,\pi_{\epsilon}(\mathcal{A}_{4a}\setminus\mathcal{A}_{a})\\
 & \le C_{2}\left(\|\nabla
 h_{\mathcal{M}_{1}}\|_{L^{\infty}(\mathbb{R}^{d})}\right)^{2}
 \epsilon^{1-d/2} \, e^{-a/\epsilon}\,.
\end{aligned}
\]
This shows  that $h_{\mathcal{M}_{1}}$ satisfies the last condition,
completing the proof of the lemma.
\end{proof}

\subsection{Equivalence classes without absorbing states}
\label{72}

Throughout this subsection, without recalling it at each statement, we
suppose that $\mathfrak{D}$ does not contain ${\bf y}^{(p)}$-absorbing
states. Thus, either $|\mathfrak{D}|\ge2$ or
$\mathfrak{D}=\{\mathcal{M}\}$ for some transient state
$\mathcal{M}\in\mathscr{V}^{(p)}$ of the chain
$\{{\bf y}^{(p)}(t)\}_{t\ge0}$. Recall from the beginning of this
section the definition of the set $\widehat{\mathfrak{D}}$. Keep in
mind that $\widehat{\mathfrak{D}}$ is the family of sets in
$\mathscr{S}^{(p)}$, and that
$\mathfrak{D}=\widehat{\mathfrak{D}}\cap\mathscr{V}^{(p)}$.

We claim that
\begin{equation}
\label{01}
\text{ $\widehat{\mathfrak{D}}$ does not contain
$\widehat{{\bf y}}^{(p)}$-absorbing states.}
\end{equation}
Indeed, by Proposition \ref{p: tree}-(4), if
$\mathcal{M}\in\mathscr{S}^{(p)}$ is an absorbing state of the chain
$\{\widehat{{\bf y}}^{(p)}(t)\}_{t\ge0}$,
$\Xi(\mathcal{M})>d^{(p)}$. Hence, by Proposition \ref{p_char},
$\mathcal{M}$ is an absorbing state of the chain
$\{{\bf y}^{(p)}(t)\}_{t\ge0}$, in contradiction with the hypothesis
of this subsection that $\mathfrak{D}$ does not contain
${\bf y}^{(p)}$-absorbing states.  This proves \eqref{01}.

\subsubsection{\label{sec721}Level sets containing equivalence
classes}

In this subsection, we construct level sets containing equivalence
classes. Fix a ${\bf y}^{(p)}$-equivalence class $\mathfrak{D}$
satisfying the assumption of Section \ref{72}. The next lemma shows
the existence of a level set containing the equivalence class
$\widehat{\mathfrak{D}}$.

\begin{lem}
\label{l_equi_height}
We have that
\begin{enumerate}
\item $U(\mathcal{M})=U(\mathcal{M}')$ for all
$\mathcal{M},\,\mathcal{M}'\in\mathfrak{D}$.

\item $\widehat{\mathfrak{D}}$ is contained in a connected component of
$\{U\le H+d^{(p)}\}$, where $H:=U(\mathcal{M})$ for $\mathcal{M}\in\mathfrak{D}$.
In particular, this component contains $\mathfrak{D}$.
\end{enumerate}
\end{lem}

\begin{proof}
Consider the first assertion. If $|\mathfrak{D}|=1$, there is nothing
to prove. If $|\mathfrak{D}|\ge2$, the assertion is \cite[Lemma
5.2-(2)]{LLS-2nd}.  Mind that \cite[Lemma 5.2-(2)]{LLS-2nd} is derived
for the recurrent classes $\mathscr{R}$ of the Markov chain
$\{{\bf y}^{(p)}(t)\}_{t\ge0}$ such that $|\mathscr{R}|\ge2$. But the
proof is the same for equivalence classes $\mathfrak{D}$ of
$\{{\bf y}^{(p)}(t)\}_{t\ge0}$ such that $|\mathfrak{D}|\ge2$.

We turn to the second assertion. Suppose that $\widehat{\mathfrak{D}}=\{\mathcal{M}\}$
for some $\mathcal{M}\in\mathscr{V}^{(p)}$. If $\mathcal{M}$ is
a singleton, there is nothing to prove. If $\mathcal{M}$ is not a
singleton, then $p\ge2$. Since $p\ge2$ and $d^{(p)}>d^{(p-1)}$,
by \cite[Lemmas 5.3 and A.9]{LLS-2nd}, there exists a connected component
of $\{U<H+d^{(p)}\}$ containing $\mathcal{M}$ so that the second
assertion holds.

Suppose that $|\widehat{\mathfrak{D}}|\ge2$. Then,
the second assertion is \cite[Lemma 13.2-(3)]{LLS-2nd}. Note that
\cite[Lemma 13.2-(3)]{LLS-2nd} is derived for the recurrent classes
$\mathscr{R}$ of the Markov chain $\{{\bf y}^{(p)}(t)\}_{t\ge0}$
such that $|\mathscr{R}|\ge2$. But the proof is the same for equivalence
classes $\widehat{\mathfrak{D}}$ of $\{\widehat{{\bf y}}^{(p)}(t)\}_{t\ge0}$
such that $|\widehat{\mathfrak{D}}|\ge2$.
\end{proof}

By Lemma \ref{l_equi_height}, there exists
${\color{blue}H=H_{\mathfrak{D}}}\in\mathbb{R}$ such that
$H=U(\mathcal{M})$ for $\mathcal{M}\in\mathfrak{D}$, and a connected
component\textcolor{blue}{{} $\mathcal{K}=\mathcal{K}_{\mathfrak{D}}$
}of $\{U\le H+d^{(p)}\}$ containing $\widehat{\mathfrak{D}}$. Since
$\mathcal{K}$ contains $\mathfrak{D}$ and $U(\mathcal{M})=H$ for
$\mathcal{M}\in\mathfrak{D}$, $\mathcal{K}$ is not a singleton.  Then,
by \cite[Lemma A.11]{LLS-2nd},
\begin{equation}
\mathcal{K}=\bigcup_{i=1}^{\ell}\overline{\mathcal{W}_{i}}\ ,\ \ \mathcal{M}_{0}\cap\mathcal{K}=\mathcal{M}_{0}\cap\bigcup_{i=1}^{\ell}\mathcal{W}_{i}\,,\label{e_lv_decomp}
\end{equation}
where \textcolor{blue}{$\mathcal{W}_{1},\,\dots,\,\mathcal{W}_{\ell}$}
denote all connected components of $\{U<H+d^{(p)}\}$ intersecting
with $\mathcal{K}$.

\begin{itemize}
\item For $p\in\llbracket1,\,\mathfrak{q}\rrbracket$ and $\mathcal{A}\subset\mathbb{R}^{d}$,
we say that $\mathcal{A}$ \emph{does not separate $(p)$-states}
if for all $\mathcal{M}\in\mathscr{S}^{(p)}$, $\mathcal{M}\subset\mathcal{A}$
or $\mathcal{M}\subset\mathcal{A}^{c}$.
\end{itemize}

The following is the main property of the level set $\mathcal{K}$ containing
the equivalence class $\mathfrak{D}$. Since the proof is technical,
it is postponed to Section \ref{subsec_lv_equiv}.

\begin{lem}
\label{l_equi_lv} The integer $\ell\in\mathbb{N}$ and the sets $\mathcal{W}_{1},\,\dots,\,\mathcal{W}_{\ell}$
introduced in \eqref{e_lv_decomp} are such that
\begin{enumerate}
\item $\ell\ge2$.
\item For each $i\in\llbracket1,\,\ell\rrbracket$, $\mathcal{W}_{i}$ does
not separate $(p)$-states. In particular, for all $\mathcal{M}\in\widehat{\mathfrak{D}}$,
there exists $a\in\llbracket1,\,\ell\rrbracket$ such that
$\mathcal{M}\in\mathscr{S}^{(p)}(\mathcal{W}_{a})$.

\item For each $i\in\llbracket1,\,\ell\rrbracket$, if $\mathscr{S}^{(p)}(\mathcal{W}_{i})\cap\mathfrak{D}\ne\varnothing$,
then $\mathscr{S}^{(p)}(\mathcal{W}_{i})\cap\mathfrak{D}=\mathscr{V}^{(p)}(\mathcal{W}_{i})=\{\mathcal{M}^{*}(\mathcal{W}_{i})\}$
and $U\left(\mathcal{M}^{*}(\mathcal{W}_{i})\right)=H$.
\item For each $i\in\llbracket1,\,\ell\rrbracket$, if $\mathscr{S}^{(p)}(\mathcal{W}_{i})\cap\mathfrak{D}=\varnothing$
and $\mathscr{S}^{(p)}(\mathcal{W}_{i})\cap\widehat{\mathfrak{D}}\ne\varnothing$,
then $\mathscr{V}^{(p)}(\mathcal{W}_{i})=\varnothing$, $\mathcal{M}^{*}(\mathcal{W}_{i})\in\widehat{\mathfrak{D}}$,
and $U\left(\mathcal{M}^{*}(\mathcal{W}_{i})\right)>H$.
\end{enumerate}
\end{lem}

Let ${\color{blue}\mathcal{M}_{i}}:=\mathcal{M}^{*}(\mathcal{W}_{i})$
for $i\in\llbracket1,\,\ell\rrbracket$. Without loss of generality,
assume that for some $1\le{\color{blue}n\le m}\le\ell$, 
\begin{itemize}
\item $\mathscr{S}^{(p)}(\mathcal{W}_{i})\cap\mathfrak{D}\ne\varnothing$
for $i\in\llbracket1,\,n\rrbracket$,
\item $\mathscr{S}^{(p)}(\mathcal{W}_{i})\cap\mathfrak{D}=\varnothing$
and $\mathscr{S}^{(p)}(\mathcal{W}_{i})\cap\widehat{\mathfrak{D}}\ne\varnothing$
for $i\in\llbracket n+1,\,m\rrbracket$,
\item and $\mathscr{S}^{(p)}\left(\mathcal{W}_{i}\right)\cap\widehat{\mathfrak{D}}=\varnothing$
for $i\in\llbracket m+1,\,\ell\rrbracket$.
\end{itemize}

By Lemma \ref{l_equi_lv},
$\mathfrak{D}={\color{blue}\{\mathcal{M}_{1},\,\dots,\,\mathcal{M}_{n}\}}$
and
${\color{blue}\mathcal{M}_{n+1},\,\dots,\,
\mathcal{M}_{m}}\in\widehat{\mathfrak{D}}\setminus\mathfrak{D}$. Note
that $\widehat{\mathfrak{D}} \setminus\mathfrak{D} $ may contain other
sets.

By definition of $H$,
$U(\mathcal{M}_{1})=\cdots=U(\mathcal{M}_{n})=H$. We claim that
\begin{equation}
\label{02}
U(\mathcal{M})>H\ \text{for all}\
\mathcal{M}\in\widehat{\mathfrak{D}}\setminus\mathfrak{D}\,. 
\end{equation}
In particular, $U(\mathcal{M}_{n+1}),\,\dots,\,U(\mathcal{M}_{m})>H$.
To prove \eqref{02}, fix
$\mathcal{M}\in\widehat{\mathfrak{D}}\setminus\mathfrak{D}$.  Then,
there exists $i\in\llbracket1,\,m\rrbracket$ such that
$\mathcal{M}\in\mathscr{S}^{(p)}(\mathcal{W}_{i})$.  If
$i\in\llbracket1,\,n\rrbracket$, since
$\mathcal{M}\ne\mathcal{M}^{*}(\mathcal{W}_{i})$,
$U(\mathcal{M})>H$. If $i\in\llbracket n+1,\,m\rrbracket$,
\begin{equation}
\label{03}
U(\mathcal{M})\ge U(\mathcal{M}_{i})>H
\end{equation}
where the last inequality comes from Lemma \ref{l_equi_lv}-(4).

\subsubsection{\label{sec62}Test functions}

In this subsection, we construct the test functions introduced in
Proposition \ref{p_test}. They are approximations of the equilibrium
potentials, as these functions satisfy the properties required in the
proposition.  This is explained in detail below equation
\eqref{e_def_h}.  We follow \cite[Section 8]{LS-22a}, with some
modifications of the test functions on shallow wells.

Recall the definition of the level set $\mc K$ introduced in \eqref{e_lv_decomp}.
Let
\[
{\color{blue}\delta}=\delta(\epsilon):=\sqrt{\epsilon\log\frac{1}{\epsilon}}\,,
\]
and let $J>0$ be a large number satisfying $J^{2}>d+10$
(cf. \cite[Lemma 10.4]{LS-22a}).  Denote by
\textcolor{blue}{$\mathcal{K}_{\epsilon}$} the connected component of
$\{U<H+d^{(p)}+J^{2}\delta^{2}\}$ containing $\mathcal{K}$. For distinct
$i,\,j\in\llbracket1,\,\ell\rrbracket$, define
${\color{blue}\Sigma_{i,\,j}=\Sigma_{j,\,i}}:=\overline{\mathcal{W}_{i}}\cap\overline{\mathcal{W}_{j}}$.
By \cite[Lemma A.1]{LLS-2nd},
$\Sigma_{i,\,j}=\partial\mathcal{W}_{i}\cap\partial\mathcal{W}_{j}$,
elements of $\Sigma_{i,\,j}$ are saddle points
($\Sigma_{i,\,j}\subset\mathcal{S}_0$), and
\begin{equation}
U(\bm{\sigma})=H+d^{(p)}\ \text{for all}\ \bm{\sigma}\in\Sigma_{i,\,j},\,i,\,j\in\llbracket1,\,\ell\rrbracket\,.\label{e_sigma-H+d}
\end{equation}
For $i<j\in\llbracket1,\,\ell\rrbracket$ and $\bm{\sigma}\in\Sigma_{i,\,j}$,
denote by \textcolor{blue}{$-\lambda_{1}^{\bm{\sigma}}<0<\lambda_{2}^{\bm{\sigma}}\le \cdots \le \lambda_{d}^{\bm{\sigma}}$}
the eigenvalues of $\nabla^{2}U(\bm{\sigma})$ and by \textcolor{blue}{$\bm{e}_{1}^{\bm{\sigma}}$},
\textcolor{blue}{$\bm{e}_{k}^{\bm{\sigma}}$}, $k\in\llbracket2,\,d\rrbracket$,
the eigenvectors of $\nabla^{2}U(\bm{\sigma})$ corresponding to $-\lambda_{1}^{\bm{\sigma}}$
and $\lambda_{k}^{\bm{\sigma}}$, respectively. Choose $\bm{e}_{1}^{\bm{\sigma}}$
pointing towards $\mathcal{W}_{i}$: for all sufficiently small $a>0$,
$\bm{\sigma}+a\bm{e}_{1}^{\bm{\sigma}}\in\mathcal{W}_{i}$.

\begin{figure}
\begin{tikzpicture}
	\path [fill=purple!30!white] (-3,2.4) to [out=0,in=170] (-2,2.3) to (-2,-2.3) to [out=190,in=0] (-3,-2.4) to (-3,2.4);
	\path [fill=purple!30!white] (2,2.3) to [out=10,in=180] (3,2.4) to (3,-2.4) to [out=180,in=-10] (2,-2.3) to (2,2.3);
	\fill[CornflowerBlue!30!white] (-2,-2.7) rectangle (2,2.7);
	\path [fill=orange!30!white] (-2,2.3) to [out=-10,in=120] (-0.1,0.8) to [out=-60,in=180] (0,0.7) to [out=0,in=240] (0.1,0.8) to [out=60,in=190] (2,2.3) to (2,-2.3) to [out=170,in=-60] (0.1,-0.8) to [out=120,in=0] (0,-0.7) to [out=180,in=60] (-0.1,-0.8) to [out=240,in=10] (-2,-2.3) to (-2,2.3);

	\draw[dotted] (-3,2) to [out=0,in=120] (0,0);
	\draw[dotted] (0,-0) to [out=-60,in=180] (3,-2);

	\draw[dotted] (-3,-2) to [out=0,in=240] (0,0);
	\draw[dotted] (0,0) to [out=60,in=180] (3,2);

	\draw[dotted] (-2,-2.7) to (-2,2.7);
	\draw[dotted] (2,-2.7) to (2,2.7);
	\draw[dotted] (-2,2.7) to (2,2.7);
	\draw[dotted] (-2,-2.7) to (2,-2.7);
	\draw[line width=1pt] (2,2.3) to (2,-2.3);
	\draw[line width=1pt] (-2,-2.3) to (-2,2.3);

	\draw[line width=0.5pt] (-3,2.4) to [out=0,in=170] (-2,2.3);
	\draw[line width=1pt] (-2,2.3) to [out=-10,in=120] (-0.1,0.8);
	\draw[line width=1pt] (-0.1,0.8) to [out=-60,in=180] (0,0.7);
	\draw[line width=1pt] (0,0.7) to [out=0,in=240] (0.1,0.8);
	\draw[line width=1pt] (0.1,0.8) to [out=60,in=190] (2,2.3);
	\draw[line width=0.5pt] (2,2.3) to [out=10,in=180] (3,2.4);

	\draw[line width=0.5pt] (3,-2.4) to [out=180,in=-10] (2,-2.3);
	\draw[line width=1pt] (2,-2.3) to [out=170,in=-60] (0.1,-0.8);
	\draw[line width=1pt] (0.1,-0.8) to [out=120,in=0] (0,-0.7);
	\draw[line width=1pt] (0,-0.7) to [out=180,in=60] (-0.1,-0.8);
	\draw[line width=1pt] (-0.1,-0.8) to [out=240,in=10] (-2,-2.3);
	\draw[line width=0.5pt] (-2,-2.3) to [out=190,in=0] (-3,-2.4);

	\filldraw (0,0) circle (2pt);
	\draw (-0.3,0) node{$\bm{\sigma}$};
	\draw (2.6,0) node{$\mathcal{W}_{i}^{\epsilon}$};
	\draw (-2.5,0) node{$\mathcal{W}_{j}^{\epsilon}$};

	\draw (-3.7,-1) node{$\mathcal{B}_{\epsilon}^{\bm{\sigma}}$};
	\draw [-latex, dash pattern=on 1.5pt off 1pt] (-3.3,-1) to (-1,-0.5);
	\draw (3.8,1.5) node{$\partial_{i}\mathcal{B}_{\epsilon}^{\bm{\sigma}}$};
	\draw [-latex, dash pattern=on 1.5pt off 1pt] (3.3,1.5) to (2,1);
	\draw (-3.7,1.5) node{$\partial_{j}\mathcal{B}_{\epsilon}^{\bm{\sigma}}$};
	\draw [-latex, dash pattern=on 1.5pt off 1pt] (-3.3,1.5) to (-2,1);

	\draw (0,2.1) node{$\mathcal{C}_{\epsilon}^{\bm{\sigma}}$};
	\draw (0,-2.1) node{$\mathcal{C}_{\epsilon}^{\bm{\sigma}}$};
	\draw [decorate,decoration={brace,raise=2pt,amplitude=5pt}] (-2,2.7) -- (2,2.7);
	\draw (0,3.2) node{$\partial_{0}\mathcal{C}_{\epsilon}^{\bm{\sigma}}$};
	\draw [decorate,decoration={brace,mirror,raise=2pt,amplitude=5pt}] (-2,-2.7) -- (2,-2.7);
	\draw (0,-3.2) node{$\partial_{0}\mathcal{C}_{\epsilon}^{\bm{\sigma}}$};

	\draw (-3,3.5) node{$U=H+d^{(p)}$};
	\draw [-latex, dash pattern=on 1.5pt off 1pt] (-3,3.25) to (-2.5,2);
	\draw (3,3.5) node{$U=H+d^{(p)}+J^{2}\delta^{2}$};
	\draw [-latex, dash pattern=on 1.5pt off 1pt] (3,3.2) to (2.5,2.4);
\end{tikzpicture}

\caption{The sets around a saddle point $\bm{\sigma}$}
\label{fig1}
\end{figure}

Define the box \textcolor{blue}{$\mathcal{C}_{\epsilon}^{\bm{\sigma}}$}
centered at $\bm{\sigma}$ by
\begin{align*}
\mathcal{C}_{\epsilon}^{\bm{\sigma}}:=\bigg\{\,\bm{\sigma}+\sum_{k=1}^{d}\alpha_{k}\bm{e}_{k}^{\boldsymbol{\sigma}}\in\mathbb{R}^{d}: & -\frac{J\delta}{\sqrt{\lambda_{1}^{\bm{\sigma}}}}\leq\alpha_{1}\leq\frac{J\delta}{\sqrt{\lambda_{1}^{\bm{\sigma}}}}\\
 & \;\text{and}\,-\frac{2J\delta}{\sqrt{\lambda_{k}^{\bm{\sigma}}}}\leq\alpha_{k}\leq\frac{2J\delta}{\sqrt{\lambda_{k}^{\bm{\sigma}}}}\,\text{ for }\,2\leq k\leq d\,\bigg\}\,,
\end{align*}
 and define
\[
{\color{blue}\partial_{0}\mathcal{C}_{\epsilon}^{\bm{\sigma}}}:=\left\{ \bm{\sigma}+\sum_{k=1}^{d}\alpha_{k}\bm{e}_{k}^{\boldsymbol{\sigma}}\in\mathcal{C}_{\epsilon}^{\bm{\sigma}}:\alpha_{k}=\pm\frac{2J\delta}{\sqrt{\lambda_{k}^{\bm{\sigma}}}}\,\text{for some}\,2\le k\le d\right\} \,.
\]
By the proof of \cite[Lemma 8.3]{LS-22a},
\begin{equation}
U(\bm{x})\ge U(\bm{\sigma})+\frac{3}{2}J^{2}\delta^{2}[1+o_{\epsilon}(1)]\quad \text{for all}\ \bm{x}\in\partial_{0}\mathcal{C}_{\epsilon}^{\bm{\sigma}}\,,\label{e_67}
\end{equation}
so that $\partial_{0}\mathcal{C}_{\epsilon}^{\bm{\sigma}}\subset(\mathcal{K}_{\epsilon})^{c}$
for sufficiently small $\epsilon>0$. By \eqref{e_lv_decomp} and
\eqref{e_67}, for sufficiently small $\epsilon>0$, the set $\mathcal{K}_{\epsilon}\setminus(\bigcup_{1\le i<j\le\ell}\bigcup_{\bm{\sigma}\in\Sigma_{i,\,j}}\mathcal{C}_{\epsilon}^{\bm{\sigma}})$
has $\ell$ connected components and each component intersects with
exactly one of $\mathcal{W}_{i}$, $i\in\llbracket1,\,\ell\rrbracket$.
Furthermore, each $\mathcal{W}_{i}$, $i\in\llbracket1,\,\ell\rrbracket$,
intersects with exactly one of these connected components. Denote by
\textcolor{blue}{$\mathcal{W}_{i}^{\epsilon}$}, $i\in\llbracket1,\,\ell\rrbracket$,
the connected component of $\mathcal{K}_{\epsilon}\setminus(\bigcup_{1\le i<j\le\ell}\bigcup_{\bm{\sigma}\in\Sigma_{i,\,j}}\mathcal{C}_{\epsilon}^{\bm{\sigma}})$
intersecting with $\mathcal{W}_{i}$. Let ${\color{blue}\mathcal{B}_{\epsilon}^{\bm{\sigma}}}:=\mathcal{C}_{\epsilon}^{\bm{\sigma}}\cap\mathcal{K}_{\epsilon}$.
Since $\bm{e}_{1}^{\bm{\sigma}}$ points towards $\mathcal{W}_{i}$,
define for $\bm{\sigma}\in\Sigma_{i,\,j}$, $i<j\in\llbracket1,\,\ell\rrbracket$,
\[
\begin{aligned}{\color{blue}\partial_{i}\mathcal{B}_{\epsilon}^{\bm{\sigma}}} & :=\left\{ \bm{\sigma}+\sum_{k=1}^{d}\alpha_{k}\bm{e}_{k}^{\boldsymbol{\sigma}}\in\mathcal{B}_{\epsilon}^{\bm{\sigma}}:\alpha_{1}=\frac{J\delta}{\sqrt{\lambda_{1}^{\bm{\sigma}}}}\right\} \,,\\
{\color{blue}\partial_{j}\mathcal{B}_{\epsilon}^{\bm{\sigma}}} & :=\left\{ \bm{\sigma}+\sum_{k=1}^{d}\alpha_{k}\bm{e}_{k}^{\boldsymbol{\sigma}}\in\mathcal{B}_{\epsilon}^{\bm{\sigma}}:\alpha_{1}=-\frac{J\delta}{\sqrt{\lambda_{1}^{\bm{\sigma}}}}\right\} \,.
\end{aligned}
\]
Then, $\mathcal{K}_{\epsilon}$ can be decomposed as
\[
\mathcal{K}_{\epsilon}=\left(\bigcup_{1\le i<j\le\ell}\bigcup_{\bm{\sigma}\in\Sigma_{i,\,j}}\mathcal{B}_{\epsilon}^{\bm{\sigma}}\right)\cup\left(\bigcup_{1\le i\le\ell}\mathcal{W}_{i}^{\epsilon}\right)\,.
\]
We refer to Figure \ref{fig1} for a visualization of the sets
defined above.

For $\bm{\sigma}\in\Sigma_{i,\,j}$, $i<j\in\llbracket1,\,\ell\rrbracket$,
define $p_{\epsilon}^{\bm{\sigma}}:\mathcal{B}_{\epsilon}^{\bm{\sigma}}\to\mathbb{R}$
by
\[
p_{\epsilon}^{\bm{\sigma}}(\bm{x}):=\frac{1}{c_{\epsilon}^{\bm{\sigma}}}\int_{-J\delta/\sqrt{\lambda_{1}^{\bm{\sigma}}}}^{(\bm{x}-\bm{\sigma})\cdot\bm{e}_{1}^{\bm{\sigma}}}e^{-\frac{\lambda_{1}^{\bm{\sigma}}}{2\epsilon}t^{2}}dt\,,
\]
where the normalizing constant is given by
\[
c_{\epsilon}^{\bm{\sigma}}:=\int_{-J\delta/\sqrt{\lambda_{1}^{\bm{\sigma}}}}^{J\delta/\sqrt{\lambda_{1}^{\bm{\sigma}}}}e^{-\frac{\lambda_{1}^{\bm{\sigma}}}{2\epsilon}t^{2}}dt=\sqrt{\frac{2\pi\epsilon}{\lambda_{1}^{\bm{\sigma}}}}\,[1+o_{\epsilon}(1)]\,.
\]
By definition,
\[
p_{\epsilon}^{\bm{\sigma}}(\bm{x})=\begin{cases}
0 & \bm{x}\in\partial_{j}\mathcal{B}_{\epsilon}^{\bm{\sigma}}\,,\\
1 & \bm{x}\in\partial_{i}\mathcal{B}_{\epsilon}^{\bm{\sigma}}\,,
\end{cases}
\]
and $0\le p_{\epsilon}^{\bm{\sigma}}(\bm{x})\le1$ for
$\bm{x}\in\mathcal{B}_{\epsilon}^{\bm{\sigma}}$.

Recall from Proposition \ref{p: tree} that
\textcolor{blue}{$\widehat{\mathcal{Q}}_{\mathcal{M}}^{(p)}$},
$\mathcal{M}\in\mathscr{S}^{(p)}$, represents the law of the Markov
chain $\{\widehat{{\bf y}}^{(p)}(t)\}_{t\ge0}$ starting from
$\mathcal{M}$. For $i\in\llbracket1,\,n\rrbracket$, denote by
${\color{blue}{\bf
h}_{i}^{(p)}}:\llbracket1,\,\ell\rrbracket\to[0,\,1]$ the
$\widehat{\mathbf{y}}^{(p)}$-equilibrium potential between $\mc M_i$
and $\ms V^{(p)}\setminus \{\mc M_i\}$, set to be $0$ on
$\llbracket1,\,m\rrbracket^{c}$:
\begin{equation}
{\color{blue}{\bf h}_{i}^{(p)}(k}\mathclose{\color{blue})}:=\begin{cases}
\widehat{\mathcal{Q}}_{\mathcal{M}_{k}}^{(p)}\left[H_{\mathscr{V}^{(p)}}=H_{\mathcal{M}_{i}}\right] & k\in\llbracket1,\,m\rrbracket\,,\\
0 & k\in\llbracket m+1,\,\ell\rrbracket\,.
\end{cases}\label{e_def_h}
\end{equation}
Mind that ${\bf h}_{i}^{(p)}(k)=\delta_i(k)$ for $k \in \llbracket 1, \, n \rrbracket$.

With the help of the equilibrium potentials ${\bf h}_{i}^{(p)}$, we
define a family of test functions
$(h_{\mathcal{M}}^{\epsilon})_{\epsilon>0}$,
$\mathcal{M}\in\mathfrak{D}$, which fulfill the requirements of
Proposition \ref{p_test}.

By the definition of the generator $\mathfrak{L}^{(p)}$, the Dirichlet
form of $\delta_{\mathcal{M}}:\mathscr{V}^{(p)}\to\mathbb{R}$ with
respect to the generator $\mathfrak{L}^{(p)}$ is given by
\[
-\sum_{\mathcal{M}'\in\mathscr{V}^{(p)}} \nu(\mathcal{M}')\delta_{\mathcal{M}}(\mathcal{M}')\left(\mathfrak{L}^{(p)}\delta_{\mathcal{M}}\right)(\mathcal{M}')=\nu(\mathcal{M})\sum_{\mathcal{M}''\in\mathscr{V}^{(p)}}r^{(p)}(\mathcal{M},\,\mathcal{M}'')\,.
\]
Thus, we need to find test functions $h_{\mathcal{M}}^{\epsilon}$
whose Dirichlet forms multiplied by
$\nu_\star e^{H/\epsilon}\theta_\epsilon^{(p)}$  converge to
the Dirichlet form of $\delta_{\mathcal{M}}$ with respect to
$\mathfrak{L}^{(p)}$.

Recall that $H_A$ denotes the first hitting time of $A$ for a given process. Since
$\delta_{\mathcal{M}}(\cdot)=\mathcal{Q}_{\,\cdot}^{(p)}\left[H_{\mathcal{M}}=H_{\mathscr{V}^{(p)}}\right]$
is the $\mathbf{y}^{(p)}$-equilibrium potential between $\mc M$ and
$\ms V^{(p)}\setminus \{\mc M\}$, and since $\mathbf{y}^{(p)}$
describes the reduced evolution of the diffusion process, potential
theory suggests that the desired test function $h_{\mathcal{M}}^{\epsilon}$
should approximate the equilibrium potential
\[
h_{\mathcal{E}(\mathcal{M}),\,\mathcal{E}^{(p)}\setminus\mathcal{E}(\mathcal{M})}(\bm{x}):=\mathbb{P}_{\bm{x}}^{\epsilon}\left[H_{\mathcal{E}(\mathcal{M})}=H_{\mathcal{E}^{(p)}}\right] \,, \quad \bm{x}\in\mathbb{R}^d\,.
\]

We now construct a test function on $\mathcal{K}_\epsilon$ close to $h_{\mathcal{E}(\mathcal{M}),\,\mathcal{E}^{(p)}\setminus\mathcal{E}(\mathcal{M})}$.
Fix $\mathcal{M}\in\mathfrak{D}$ and let $i\in\llbracket 1, \, n \rrbracket$ be such that $\mathcal{M}_{i}=\mathcal{M}$.

\smallskip
\noindent$\bullet$ Behavior inside wells $\mathcal{W}_k$, $k\in\llbracket 1, \, \ell \rrbracket$
\smallskip

If $k\in\llbracket1,\,n\rrbracket$, since $\mathcal{M}_k \in \mathscr{V}^{(p)}$, we expect that as $\epsilon\to0$, for $\bm{x}\in\mathcal{E}(\mathcal{M}_k)$,
\[
h_{\mathcal{E}(\mathcal{M}),\,\mathcal{E}^{(p)}\setminus\mathcal{E}(\mathcal{M})}(\bm{x}) \approx \delta_{\mathcal{M}_i}(\mathcal{M}_k)={\bf h}_{i}^{(p)}(k) \,.
\]
For $k \in \llbracket n+1, m \rrbracket$, since $\mathcal{M}_k\in\mathscr{N}^{(p)}$, Proposition \ref{p: tree}-(5) yields
\[
\lim_{\epsilon\to0}h_{\mathcal{E}(\mathcal{M}),\,\mathcal{E}^{(p)}\setminus\mathcal{E}(\mathcal{M})}(\bm{x})=\widehat{\mathcal{Q}}_{\mathcal{M}_k}^{(p)}\left[H_{\mathscr{V}^{(p)}}=H_{\mathcal{M}_i}\right]={\bf h}_{i}^{(p)}(k) \,, \quad \bm{x}\in\mathcal{E}(\mathcal{M}_k)  \,.
\]
If $k\in\llbracket m+1,\, \ell\rrbracket$, then $\mathcal{W}_k$ contains no element of $\widehat{\mathfrak{D}}$. For $\mathcal{M}'\in\mathscr{V}^{(p)}(\mathcal{W}_k)$, the Markov chain $\{{\bf y}^{(p)}(t)\}_{t\ge0}$ starting from $\mathcal{M}'$ cannot reach $\mathfrak{D}$ with positive probability. Therefore, it is expected that as $\epsilon\to0$ for $\bm{x}\in\mathcal{E}(\mathcal{M}')$, $\mathcal{M}'\in\mathscr{V}^{(p)}(\mathcal{W}_k)$,
\[
h_{\mathcal{E}(\mathcal{M}),\,\mathcal{E}^{(p)}\setminus\mathcal{E}(\mathcal{M})}(\bm{x}) \approx 0={\bf h}_{i}^{(p)}(k) \,.
\]
In summary, the value of the test function inside each well
$\mathcal{W}_k$, $k\in\llbracket 1,\, \ell\rrbracket$, is given by
${\bf h}_{i}^{(p)}(k)$.

\smallskip
\noindent$\bullet$ Behavior near saddle points $\mathcal{B}_{\epsilon}^{\bm{\sigma}}$, $\bm{\sigma}\in\Sigma_{a,\,b}$, $a<b\in\llbracket1,\,\ell\rrbracket$
\smallskip

We next consider the neighborhoods of saddle points $\mathcal{B}_{\epsilon}^{\bm{\sigma}}$, $\bm{\sigma}\in\Sigma_{a,\,b}$, $a<b\in\llbracket1,\,\ell\rrbracket$.
The equilibrium potential $h_{\mathcal{E}(\mathcal{M}),\,\mathcal{E}^{(p)}\setminus\mathcal{E}(\mathcal{M})}$ satisfies
\[
\mathscr{L}_{\epsilon}h_{\mathcal{E}(\mathcal{M}),\,\mathcal{E}^{(p)}\setminus\mathcal{E}(\mathcal{M})}(\bm{x})=0 \,, \quad \bm{x}\in\mathbb{R}^{d}\setminus\mathcal{E}^{(p)}\,.
\]
Fix $\bm{\sigma}\in\Sigma_{a,\,b}$, $a<b\in\llbracket1,\,\ell\rrbracket$.
As proved in \cite[Proposition 8.5]{LS-22a}, $\mathscr{L}_\epsilon p_{\epsilon}^{\bm{\sigma}}(\bm{x})$ is negligible for $\bm{x}\in\mathcal{B}_{\epsilon}^{\bm{\sigma}}$.
Therefore, our test function $h_{\mathcal{M}}^{\epsilon}$ is approximated by the continuous function ${\color{blue}h_{i}^{\epsilon}}:\mathcal{K}_{\epsilon}\to\mathbb{R}$ defined by
\begin{equation}
h_{i}^{\epsilon}(\bm{x}):=\begin{cases}
{\bf h}_{i}^{(p)}(k) & \bm{x}\in\mathcal{W}_{k}^{\epsilon},\,k\in\llbracket1,\,\ell\rrbracket\,,\\{}
{\bf h}_{i}^{(p)}(b)+[{\bf h}_{i}^{(p)}(a)-{\bf h}_{i}^{(p)}(b)]\,p_{\epsilon}^{\bm{\sigma}}(\bm{x}) & \bm{x}\in\mathcal{B}_{\epsilon}^{\bm{\sigma}},\,\bm{\sigma}\in\Sigma_{a,\,b},\,a<b\in\llbracket1,\,\ell\rrbracket\,.
\end{cases}\label{e_def_h_i}
\end{equation}

Define the vector field
${\color{blue}\Phi_{i}^{\epsilon}}:\mathbb{R}^{d}\to\mathbb{R}^{d}$ as
\[
\Phi_{i}^{\epsilon}(\bm{x}):=\begin{cases}
\left[{\bf h}_{i}^{(p)}(a)-{\bf h}_{i}^{(p)}(b)\right]\nabla p_{\epsilon}^{\bm{\sigma}}(\bm{x}) & \bm{x}\in\mathcal{B}_{\epsilon}^{\bm{\sigma}},\,\bm{\sigma}\in\Sigma_{a,\,b},\,a<b\in\llbracket1,\,\ell\rrbracket\,,\\
0 & \text{otherwise}\,.
\end{cases}
\]

The following proposition is the main result of this section. The
proof is postponed to Section \ref{subsec_pf_p_cap}. Recall the
definition of the weights $\nu(\mathcal{M})$, $\mc M\subset \mc M_0$,
and $\nu_{\star}$, given in \eqref{e_def_nu}.

\begin{prop}
\label{p_cap}
Recall that we assumed that $\mf D$ has no absorbing states.
For all $i\in\llbracket1,\,n\rrbracket$,
\[
\lim_{\epsilon\to0}e^{H/\epsilon}\theta_{\epsilon}^{(p)}\epsilon\int_{\mathbb{R}^{d}}|\Phi_{i}^{\epsilon}|^{2}d\pi_{\epsilon}=\frac{\nu(\mathcal{M}_{i})}{\nu_{\star}}\sum_{\mathcal{M}'\in\mathscr{V}^{(p)}\setminus\{\mathcal{M}_{i}\}}r^{(p)}(\mathcal{M}_{i},\,\mathcal{M}')\,. 
\]
If $n\ge2$, for distinct $i,\,j\in\llbracket1,\,n\rrbracket$,
\[
\lim_{\epsilon\to0}e^{H/\epsilon}\theta_{\epsilon}^{(p)}\epsilon\int_{\mathbb{R}^{d}}\Phi_{i}^{\epsilon}\cdot\Phi_{j}^{\epsilon}\,d\pi_{\epsilon}=-\frac{1}{2\nu_{\star}}\left(\nu(\mathcal{M}_{i})\,r^{(p)}(\mathcal{M}_{i},\,\mathcal{M}_{j})+\nu(\mathcal{M}_{j})\,r^{(p)}(\mathcal{M}_{j},\,\mathcal{M}_{i})\right)\,.
\]
\end{prop}

Let $\color{blue} \xi:\mathbb{R}^{d}\to\mathbb{R}$ be a smooth,
positive, rotationally invariant function supported on the unit ball
$B_{1}$ and satisfying
\[
\int_{\mathbb{R}^{d}}\xi(x)\,dx=1.
\]
For $\eta>0$ , write
\[
{\color{blue} \xi_{\eta}(\bm{x}) } :=\eta^{-d}\xi(\eta^{-1}\bm{x})\,.
\]
The following result is \cite[Proposition 10.2]{LS-22a}.
\begin{lem}
\label{l_75} For all $i\in\llbracket1,\,n\rrbracket$,
\[
\lim_{\epsilon\to0} e^{H/\epsilon}\theta_{\epsilon}^{(p)}\epsilon \int_{\mathbb{R}^{d}}\left|\nabla\left(h_{i}^{\epsilon}*\xi_{\epsilon^{2}}\right)-\Phi_{i}^{\epsilon}\right|^{2}d\pi_{\epsilon}=0\,,
\]
where $*$ represents the usual convolution.
\end{lem}

Fix $\eta>0$ small enough so that there is no critical point
$\bm{c}\in\mathcal{C}_{0}$ such that
$U(\bm{c})\in(H+d^{(p)},\,H+d^{(p)}+\eta)$. Let $\color{blue} \Omega$
be the connected component of $\{U<H+d^{(p)}+\eta\}$ containing
$\mathcal{K}$.  For $\mathcal{A},\,\mathcal{B}\subset\mathbb{R}^{d}$,
define
${\color{blue}d(\mathcal{A},\,\mathcal{B}}\mathclose{\color{blue})}:=\inf\{|\bm{x}-\bm{y}|:\bm{x}\in\mathcal{A},\,\bm{y}\in\mathcal{B}\}$.
If $\mathcal{A}=\{\bm{x}\}$ for some $\bm{x}\in\mathbb{R}^{d}$, let us
write
${\color{blue}d(\bm{x},\,\mathcal{B}}\mathclose{\color{blue})}:=d(\{\bm{x}\},\,\mathcal{B})$.
Since $h_{i}^{\epsilon}(\bm{x})=0$,
$i\in\llbracket1,\,n\rrbracket$, for
$\bm{x}\notin\mathcal{K}_{\epsilon}$,
$\bigcap_{\epsilon>0}\mathcal{K}_{\epsilon}=\mathcal{K}$, and
$d(\mathcal{K},\,\Omega^{c})>0$, there exists $\epsilon_{1}>0$ such
that for $\epsilon\in(0,\,\epsilon_{1})$,
\begin{equation}
(h_{i}^{\epsilon}*\xi_{\epsilon^{2}})
(\bm{x})=0\ \text{for}\ \bm{x}\in\Omega^{c}\,.\label{e_eps1}
\end{equation}

For $i\in\llbracket1,\,\ell\rrbracket$, define
\[
{\color{blue}\mathcal{V}_{i}^{\epsilon}}:=\{\bm{x}\in\mathcal{W}_{i}^{\epsilon}:d(\bm{x},\,\partial\mathcal{W}_{i}^{\epsilon})>\epsilon^{2}\}\,.
\]
Decompose $\Omega$ as
\[
\Omega=\mathcal{A}_{\epsilon}\cup\left(\bigcup_{i=1}^{\ell}\mathcal{V}_{i}^{\epsilon}\right)\,,
\]
where ${\color{blue}\mathcal{A}_{\epsilon}}:=\Omega\setminus\left(\bigcup_{i=1}^{\ell}\mathcal{V}_{i}^{\epsilon}\right)$.
Mind that $\mathcal{A}_{\epsilon}\subset\left(\Omega\setminus\mathcal{K}_{\epsilon}\right)\cup\left(\bigcup_{1\le j< k\le\ell}\bigcup_{\bm{\sigma}\in\Sigma_{j,\,k}}\mathcal{B}_{\epsilon}^{\bm{\sigma}}\right)\cup\left(\bigcup_{1\le i\le\ell}\mathcal{W}_{i}^{\epsilon}\setminus\mathcal{V}_{i}^{\epsilon}\right)$.
We claim that there exists $\epsilon_{2}>0$ such that for
$\epsilon\in(0,\,\epsilon_{2})$, 
\begin{equation}
U(\bm{x})>H+d^{(p)}/2\quad \text{for}\ \bm{x}\in\mathcal{A}_{\epsilon}\,.\label{e_eps2}
\end{equation}
By the definition of $\mathcal{K}_{\epsilon}$, $U(\bm{x})>H+d^{(p)}$
for $\bm{x}\in\Omega\setminus\mathcal{K}_{\epsilon}$. Since $\bigcap_{\epsilon>0}\mathcal{B}_{\epsilon}^{\bm{\sigma}}=\{\bm{\sigma}\}$
and $U(\bm{\sigma})=H+d^{(p)}$, $\bm{\sigma}\in\bigcup_{1\le i< j\le\ell}\Sigma_{i,\,j}$,
there exists $\epsilon_{2}^{(1)}>0$ such that $U(\bm{x})>H+d^{(p)}/2$
for $\bm{x}\in\bigcup_{1\le i< j\le\ell}\bigcup_{\bm{\sigma}\in\Sigma_{i,\,j}}\mathcal{B}_{\epsilon}^{\bm{\sigma}}$
and $\epsilon\in(0,\,\epsilon_{2}^{(1)})$. Since $U(\bm{x})=H+d^{(p)}$
for $\bm{x}\in\partial\mathcal{W}_{i}$, $i\in\llbracket1,\,\ell\rrbracket$,
$\lim_{\epsilon\to0}d(\partial\mathcal{W}_{i},\,\partial\mathcal{W}_{i}^{\epsilon})=0$,
and $d(\mathcal{W}_{i}^{\epsilon}\setminus\mathcal{V}_{i}^{\epsilon},\partial\mathcal{W}_{i}^{\epsilon})\le\epsilon^{2}$,
there exists $\epsilon_{2}^{(2)}>0$ such that $U(\bm{x})>H+d^{(p)}/2$
for $\bm{x}\in\bigcup_{1\le i\le\ell}\mathcal{W}_{i}^{\epsilon}\setminus\mathcal{V}_{i}^{\epsilon}$
and $\epsilon\in(0,\,\epsilon_{2}^{(2)})$. Then, $\epsilon_{2}:=\min\{\epsilon_{2}^{(1)},\,\epsilon_{2}^{(2)}\}$
satisfies \eqref{e_eps2}.

We are in a position to  prove Proposition \ref{p_test}.

\begin{proof}[Proof of Proposition \ref{p_test}]
\begin{flushleft}
 Suppose that $\mathfrak{D}$ contains an absorbing state $\mathcal{M}_{1}\in\mathscr{V}^{(p)}$
of $\{{\bf y}^{(p)}(t)\}_{t\ge0}$. Then, $\mathfrak{D}=\{\mathcal{M}_{1}\}$
and $r^{(p)}(\mathcal{M}_{1},\,\mathcal{M}')=0$ for all $\mathcal{M}'\in\mathscr{V}^{(p)}$
so that the proof is a direct consequence of Lemma \ref{l_test_absobing}.
\par\end{flushleft}

Suppose that $\mathfrak{D}$ does not contain absorbing states. Fix
$\mathcal{M}\in\mathfrak{D}$. Then, there exists $i\in\llbracket1,\,n\rrbracket$
such that $\mathcal{M}=\mathcal{M}_{i}$. Let $h_{\mathcal{M}}^{\epsilon}:=h_{i}^{\epsilon}*\xi_{\epsilon^{2}}$.
By \eqref{e_eps1},
\begin{equation}
e^{H_{\mathfrak{D}}/\epsilon}\int_{\mathbb{R}^{d}\setminus\Omega}(h_{\mathcal{M}}^{\epsilon})^{2}\,d\pi_{\epsilon}=0\,.\label{e_pf_p_test-1}
\end{equation}
By \eqref{e_eps2}, since $h_{\mathcal{M}}^{\epsilon}$ is uniformly
bounded and $\mc A_\epsilon$ is a bounded set, there exists $C_{1}>0$ such that
\begin{equation}
\limsup_{\epsilon\to0}e^{H_{\mathfrak{D}}/\epsilon}\int_{\mathcal{A}_{\epsilon}}(h_{\mathcal{M}}^{\epsilon})^{2}\,d\pi_{\epsilon}\le C_{1}\limsup_{\epsilon\to0}\epsilon^{-d/2}e^{-d^{(p)}/(2\epsilon)}=0\,.\label{e_pf_p_test-2}
\end{equation}
Since $\mathcal{M}=\mathcal{M}^{*}(\mathcal{W}_{i})$ and $U(\mathcal{M})=H_{\mathfrak{D}}$,
\[
\inf_{\bm{x}\in\overline{\mathcal{W}_{i}}
\setminus\operatorname{int}\mathcal{E}(\mathcal{M})}
U(\bm{x})>H_{\mathfrak{D}}.
\]
By construction, the sets $\mathcal{W}_{i}^{\epsilon}$ are contained in a common compact set $\mathcal{K}_1$ for all sufficiently small $\epsilon>0$, and
\[
\lim_{\epsilon\to0}
\sup_{\bm{x}\in\mathcal{W}_{i}^{\epsilon}}
\operatorname{dist}(\bm{x},\overline{\mathcal{W}_{i}})=0.
\]
Hence, by continuity of $U$, there exists $c_{1}>0$ such that, for all sufficiently small $\epsilon>0$,
\[
U(\bm{x})\ge H_{\mathfrak{D}}+c_{1}
\quad\text{for all }\bm{x}\in
\mathcal{W}_{i}^{\epsilon}\setminus\mathcal{E}(\mathcal{M}).
\]
Therefore, since $\mathcal{V}_{i}^{\epsilon}\subset\mathcal{W}_{i}^{\epsilon}$ and $h_{\mathcal{M}}^{\epsilon}$ is uniformly bounded, there exists $C_{2}>0$ such that
\begin{equation}
\limsup_{\epsilon\to0}e^{H_{\mathfrak{D}}/\epsilon}\int_{\mathcal{V}_{i}^{\epsilon}\setminus\mathcal{E}(\mathcal{M})}(h_{\mathcal{M}}^{\epsilon})^{2}\,d\pi_{\epsilon}\le C_{2}\limsup_{\epsilon\to0}\epsilon^{-d/2}e^{-c_1/\epsilon}=0\,.\label{e_pf_p_test-3}
\end{equation}
Fix $k\in\llbracket1,\,\ell\rrbracket\setminus\{i\}$. If $k\in\llbracket1,\,n\rrbracket\cup\llbracket m+1,\,\ell\rrbracket$,
since $h_{i}^{\epsilon}(\bm{x})=0$ for $\bm{x}\in\mathcal{W}_{k}^{\epsilon}$,
$h_{\mathcal{M}}^{\epsilon}(\bm{x})=0$ for $\bm{x}\in\mathcal{V}_{k}^{\epsilon}$
so that
\begin{equation}
e^{H_{\mathfrak{D}}/\epsilon}\int_{\mathcal{V}_{k}^{\epsilon}}(h_{\mathcal{M}}^{\epsilon})^{2}\,d\pi_{\epsilon}=0\,.\label{e_pf_p_test-4}
\end{equation}
If $k\in\llbracket n+1,\,m\rrbracket$, by \eqref{03},
$U(\mc M_k)> H_{\mf D}$. Hence, there exists $c_2>0$ such that
$U(\bm{x})>H_{\mathfrak{D}}+c_2$ for
$\bm x \in \mathcal{V}_{k}^{\epsilon}$. As
$h_{\mathcal{M}}^{\epsilon}$ is uniformly bounded, there exists
$C_{3}>0$ such that
\begin{equation}
\limsup_{\epsilon\to0}e^{H_{\mathfrak{D}}/\epsilon}\int_{\mathcal{V}_{k}^{\epsilon}}(h_{\mathcal{M}}^{\epsilon})^{2}\,d\pi_{\epsilon}\le C_{3}\limsup_{\epsilon\to0}\epsilon^{-d/2}e^{-c_2/\epsilon}=0\,.\label{e_pf_p_test-5}
\end{equation}
Hence, the first assertion follows from \eqref{e_pf_p_test-1}-\eqref{e_pf_p_test-5}.

For the second assertion, by Proposition \ref{p_cap} and Lemma \ref{l_75},
\[
\lim_{\epsilon\to0}e^{H/\epsilon}\theta_{\epsilon}^{(p)}\epsilon\int_{\mathbb{R}^{d}}\left|\nabla h_{\mathcal{M}}^{\epsilon}\right|^{2}d\pi_{\epsilon}=\frac{\nu(\mathcal{M}_{i})}{\nu_{\star}}\sum_{\mathcal{M}''\in\mathscr{V}^{(p)}}r^{(p)}(\mathcal{M}_{i},\,\mathcal{M}'')\,.
\]

We turn to the last assertion. Suppose that $|\mathfrak{D}|\ge2$. For
$\mathcal{M}'\in\mathfrak{D}\setminus\{\mathcal{M}\}$, let
$j\in\llbracket1,\,n\rrbracket\setminus\{i\}$ be such that
$\mathcal{M}'=\mathcal{M}_{j}$.  Then, by Proposition \ref{p_cap} and
Lemma \ref{l_75},
\[
\lim_{\epsilon\to0}e^{H/\epsilon}\theta_{\epsilon}^{(p)}\epsilon\int_{\mathbb{R}^{d}}\nabla h_{\mathcal{M}}^{\epsilon}\cdot\nabla h_{\mathcal{M}'}^{\epsilon}\,d\pi_{\epsilon}=-\frac{1}{2\nu_{\star}}\left(\nu(\mathcal{M}_{i})r^{(p)}(\mathcal{M}_{i},\,\mathcal{M}_{j})+\nu(\mathcal{M}_{j})r^{(p)}(\mathcal{M}_{j},\,\mathcal{M}_{i})\right)\,.
\]
This completes the proof of Proposition \ref{p_test}.
\end{proof}

\subsubsection{\label{subsec_lv_equiv}Proof of Lemma \ref{l_equi_lv}}

In this part, we prove Lemma \ref{l_equi_lv}. Recall that for $\mathcal{A}\subset\mathbb{R}^{d}$,
${\color{blue}\mathcal{M}^{*}(\mathcal{A}}\mathclose{\color{blue})}:=\{\bm{m}\in\mathcal{M}_{0}\cap\mathcal{A}:U(\bm{m})=\min_{\bm{x}\in\mathcal{A}}U(\bm{x})\}$.

\begin{lem}
\label{l_76}
The integer $\ell\in\mathbb{N}$ and the
sets $\mathcal{W}_{1},\,\dots,\,\mathcal{W}_{\ell}$ introduced in
\eqref{e_lv_decomp} satisfy the following.

\begin{enumerate}
\item $\ell\ge2$ and there exists a saddle point $\bm{\sigma}\in\mathcal{S}_{0}\cap\mathcal{K}$
such that $U(\bm{\sigma})=H+d^{(p)}$. In particular, $\mathcal{K}$
is the connected component of $\{U\le U(\bm{\sigma})\}$ containing
$\bm{\sigma}$.

\item For each $i\in\llbracket1,\,\ell\rrbracket$, $\mathcal{W}_{i}$ does
not separate $(p)$-states.
\item For each $i\in\llbracket1,\,\ell\rrbracket$, if $\min_{\bm{x}\in\mathcal{W}_{i}}U(\bm{x})=H$,
then
$\mathscr{V}^{(p)}(\mathcal{W}_{i})=\{\mathcal{M}^{*}(\mathcal{W}_{i})\}$.

\item For each $i\in\llbracket1,\,\ell\rrbracket$, if $\min_{\bm{x}\in\mathcal{W}_{i}}U(\bm{x})>H$,
then $\mathscr{V}^{(p)}(\mathcal{W}_{i})=\varnothing$ and $\mathcal{M}^{*}(\mathcal{W}_{i})\in\mathscr{N}^{(p)}$.
\end{enumerate}
\end{lem}

\begin{proof}
Recall that we assumed at the beginning of this section that $\mf D$
has no absorbing states.  Let $\mathcal{M}\in\mathfrak{D}$. Since
$\mathcal{M}$ is not an absorbing state, by Proposition \ref{p_char},
$\Xi(\mathcal{M})=d^{(p)}$.  Then,
$\Theta(\mathcal{M},\,\widetilde{\mathcal{M}})=U(\mathcal{M})+\Xi(\mathcal{M})=H+d^{(p)}<\infty$
so that $\widetilde{\mathcal{M}}\ne\varnothing$. Therefore, since
$\mathcal{K}$ is the connected component of
$\{U\le\Theta(\mathcal{M},\,\widetilde{\mathcal{M}})\}$ containing
$\mathcal{M}$, the first assertion is proven by \cite[Lemma
A.13]{LLS-2nd}.  Since $\mathcal{K}$ is a connected component of
$\{U\le U(\bm{\sigma})\}$ and $\ell\ge2$, the other assertions follow
from \cite[Lemma 5.9]{LLS-2nd}.
\end{proof}
Recall that
${\color{blue}\mathcal{M}_{i}}:=\mathcal{M}^{*}(\mathcal{W}_{i})$ for
$i\in\llbracket1,\,\ell\rrbracket$. By Lemma \ref{l_equi_height}-(2),
$\widehat{\mathfrak{D}}$ is contained in $\mathcal{K}$. Hence, by
\eqref{e_lv_decomp}, any element $\mathcal{M}$ in
$\widehat{\mathfrak{D}}$ is such that
$\mathcal{M}\subset\bigcup_{i\in\llbracket1,\,\ell\rrbracket}\mathcal{W}_{i}$.
By Lemma \ref{l_76}-(2), the sets $\mathcal{W}_{i}$,
$i\in\llbracket1,\,\ell\rrbracket$, do not separate
$(p)$-states. Thus, for $\mathcal{M}\in\widehat{\mathfrak{D}}$, there
exists $j\in\llbracket1,\,\ell\rrbracket$ such that
$\mathcal{M}\subset\mathcal{W}_{j}$.

\begin{lem}
\label{l_77}
For all $i\in\llbracket1,\,\ell\rrbracket$ such that
$\widehat {\mathfrak{D} }\cap \mathscr{S}^{(p)}(\mathcal{W}_{i} ) \neq
\varnothing$, $\mathcal{M}_{i}\in\widehat{\mathfrak{D}}$,
$U(\mathcal{M}_{i})\ge H$, and $\Xi(\mathcal{M}_{i})\le d^{(p)}$.
\end{lem}

\begin{proof}
Recall that we assumed at the beginning of this section that $\mf D$
has no absorbing states.  Fix
$\mathcal{M}\in\widehat{\mathfrak{D}} \cap
\mathscr{S}^{(p)}(\mathcal{W}_{i})$ for some
$i\in\llbracket1,\,\ell\rrbracket$. First, suppose that
$\widehat{\mathfrak{D}}$ is contained in $\mathcal{W}_{i}$, i.e.,
$\widehat{\mathfrak{D}}\cap\mathscr{S}^{(p)}(\mathcal{W}_{k})=\varnothing$
for all $k\in\llbracket1,\,\ell\rrbracket\setminus\{i\}$.

Let
$\mathcal{M}'\in\mathfrak{D}\cap\mathscr{V}^{(p)}(\mathcal{W}_{i})$.
Mind that $\mc M'$ may be equal to $\mc M$.
By Lemma
\ref{l_equi_height}, $U(\mathcal{M}')=H$.  By \eqref{e_lv_decomp} and
Lemma \ref{l_76}-(2), $\mathcal{W}_{i}$ is the connected component of
$\{U<H+d^{(p)}\}$ containing $\mathcal{M}'$. Since $\mathcal{M}'$ is
not an absorbing state, by Proposition \ref{p_char},
$d^{(p)}=\Xi(\mathcal{M}')$ so that
\[
H+d^{(p)}=U(\mathcal{M}')+\Xi(\mathcal{M}')=\Theta(\mathcal{M}',\,\widetilde{\mathcal{M}'})\,.
\]
Therefore, by Lemma \ref{l_pot1}-(1),
$\widetilde{\mathcal{M}'}\subset(\mathcal{W}_{i})^{c}$.  Hence,
$U(\bm{m}) > U(\mathcal{M}')$ for all
$\bm{m}\in (\mathcal{M}_{0} \setminus \mathcal{M}')
\cap\mathcal{W}_{i}$. Thus
$\mathcal{M}'=\mathcal{M}^{*}(\mathcal{W}_{i})=\mathcal{M}_{i}$. It
follows from the estimates obtained for $\mc M'$ that
$\mathcal{M}_{i}\in\widehat{\mathfrak{D}}$, $U(\mathcal{M}_{i})= H$,
and $\Xi(\mathcal{M}_{i}) =  d^{(p)}$.

Suppose that there exist $j\in\llbracket1,\,\ell\rrbracket\setminus\{i\}$
and $\mathcal{M}^{(1)}\in\widehat{\mathfrak{D}}\cap\mathscr{S}^{(p)}(\mathcal{W}_{j})$.
Since the Markov chain $\{\widehat{{\bf y}}^{(p)}(t)\}_{t\ge0}$ can
reach $\mathcal{M}^{(1)}$ starting from $\mathcal{M}$, there exist
$k_{1}\in\mathbb{N}$ and $\mathcal{N}_{1}^{(1)},\,\dots,\,\mathcal{N}_{k_{1}}^{(1)}\in\widehat{\mathfrak{D}}$
such that 
\[
\widehat{r}^{(p)}(\mathcal{M},\,\mathcal{N}_{1}^{(1)}),\,\widehat{r}^{(p)}(\mathcal{N}_{1}^{(1)},\,\mathcal{N}_{2}^{(1)}),\,\dots,\,\widehat{r}^{(p)}(\mathcal{N}_{k_{1}-1}^{(1)},\,\mathcal{N}_{k_{1}}^{(1)}),\,\widehat{r}^{(p)}(\mathcal{N}_{k_{1}}^{(1)},\,\mathcal{M}^{(1)})>0\,.
\]
Let $\mathcal{N}_{0}^{(1)}=\mathcal{M}$ and $\mathcal{N}_{k_{1}+1}^{(1)}=\mathcal{M}^{(1)}$.
By Proposition \ref{p: tree}-(4),
\[
\mathcal{N}_{0}^{(1)}\to\cdots\to\mathcal{N}_{k_{1}+1}^{(1)}\,.
\]
Let $\mathcal{N}_{a_{1}}^{(1)}$ be the last element in $\mathscr{S}^{(p)}(\mathcal{W}_{i})$.
Since $\mathcal{M}^{(1)}\in\mathscr{S}^{(p)}(\mathcal{W}_{j})$, $a_{1}\le k_{1}$
and $\mathcal{N}_{a_{1}+1}^{(1)}\notin\mathscr{S}^{(p)}(\mathcal{W}_{i})$
so that by Lemma \ref{l_exit}, $\mathcal{M}_{i}=\mathcal{N}_{a_{1}}^{(1)}\in\widehat{\mathfrak{D}}$.
Since $\widehat{r}^{(p)}(\mathcal{M}_{i},\,\mathcal{N}_{a_{1}+1}^{(1)})>0$,
by Proposition \ref{p: tree}-(4) and \eqref{eq:con_gate}, $d^{(p)}\ge\Xi(\mathcal{M}_{i})$,
$\mathcal{M}_{i}\to\mathcal{N}_{a_{1}+1}^{(1)}$, and $\Theta(\mathcal{M}_{i},\,\widetilde{\mathcal{M}_{i}})=\Theta(\mathcal{M}_{i},\,\mathcal{N}_{a_{1}+1}^{(1)})$.
Since $\mathcal{N}_{a_{1}+1}^{(1)}\subset\mathcal{K}\setminus\mathcal{W}_{i}$,
by Lemma \ref{l_pot1}-(2), $\Theta(\mathcal{M}_{i},\,\mathcal{N}_{a_{1}+1}^{(1)})\ge H+d^{(p)}$.
Hence,
\[
d^{(p)}\ge\Xi(\mathcal{M}_{i})=\Theta(\mathcal{M}_{i},\,\widetilde{\mathcal{M}_{i}})-U(\mathcal{M}_{i})=\Theta(\mathcal{M}_{i},\,\mathcal{N}_{a_{1}+1}^{(1)})-U(\mathcal{M}_{i})\ge H+d^{(p)}-U(\mathcal{M}_{i})\,,
\]
which implies that $U(\mathcal{M}_{i})\ge H$. This completes the
proof.
\end{proof}

We are now in a position to prove Lemma \ref{l_equi_lv}.

\begin{proof}[Proof of Lemma \ref{l_equi_lv}]
The first two assertions have been proved in Lemma \ref{l_76}.

We turn to the third assertion.  Fix
$i\in\llbracket1,\,\ell\rrbracket$. Suppose that
$\mathscr{S}^{(p)}(\mathcal{W}_{i})\cap\mathfrak{D}\ne\varnothing$.
Fix
$\mathcal{M}'\in\mathscr{S}^{(p)}(\mathcal{W}_{i})\cap\mathfrak{D}$. By
Lemma \ref{l_equi_height}, $U(\mathcal{M}')=H$. By Lemma \ref{l_77},
$\mathcal{M}_{i}=\mathcal{M}^{*}(\mathcal{W}_{i})\in\widehat{\mathfrak{D}}$
and $U(\mathcal{M}_{i})\ge H$. Hence,
$U(\mathcal{M}')=H\le U(\mathcal{M}_{i})\le
U(\mathcal{M}')$, so that  $U(\mathcal{M}_{i})=H$ and
$\mathcal{M}'\subset\mathcal{M}_{i}$.  As
$\mathcal{M}_{i}\in\widehat{\mathfrak{D}}$,
$\mathcal{M}'=\mathcal{M}_{i}$.  Since $U(\mathcal{M}_i)=H$, by Lemma \ref{l_76}-(3),
$\mathscr{V}^{(p)}(\mathcal{W}_{i})=\{\mathcal{M}_i\}$. Therefore, since $\mathcal{M}_i$ is the unique element of $\mathscr{V}^{(p)}(\mathcal{W}_{i})$ and $\mathcal{M}_i =\mathcal{M}' \in \mathfrak{D}$,
$\mathscr{S}^{(p)}(\mathcal{W}_{i})\cap\mathfrak{D}=\{\mathcal{M}_{i}\}$.

It remains to consider the fourth assertion. Fix
$i\in\llbracket1,\,\ell\rrbracket$.  Suppose that
$\mathscr{S}^{(p)}(\mathcal{W}_{i})\cap\mathfrak{D}=\varnothing$ and
$\mathscr{S}^{(p)}(\mathcal{W}_{i})\cap\widehat{\mathfrak{D}}\ne\varnothing$.
By Lemma \ref{l_77}, $\mathcal{M}_{i}\in\widehat{\mathfrak{D}}$ and
$U(\mathcal{M}_{i})\ge H$. If $U(\mathcal{M}_{i})=H$, by Lemma
\ref{l_76}-(3), $\mathcal{M}_{i}\in\mathscr{V}^{(p)}$ so that
$\mathcal{M}_{i}\in\mathscr{V}^{(p)}\cap\widehat{\mathfrak{D}}=\mathfrak{D}$,
which is a contradiction. Therefore, $U(\mathcal{M}_i)>H$ and by Lemma \ref{l_76}-(4), $\mathscr{V}^{(p)}(\mathcal{W}_i)=\varnothing$.
\end{proof}

\section{\label{sec_pf_last}Proof of Propositions \ref{p_rev} and \ref{p_cap}}

In this section, we prove Propositions \ref{p_rev} and \ref{p_cap}.
It follows from the hypotheses of these results that $\mathfrak{D}$
has no ${\bf y}^{(p)}$-absorbing states. This condition is thus adopted throughout this
section without further comment.

Recall that for $\mathcal{A}\subset\mathbb{R}^{d}$,
${\color{blue}\mathcal{M}^{*}(\mathcal{A}}\mathclose{\color{blue})}:=\{\bm{m}\in\mathcal{M}_{0}\cap\mathcal{A}:U(\bm{m})=\min_{\bm{x}\in\mathcal{A}}U(\bm{x})\}$,
and recall from Section \ref{sec721} that:
\begin{itemize}
\item There exists ${\color{blue}H=H_{\mathfrak{D}}}\in\mathbb{R}$ such
that $H=U(\mathcal{M})$ for $\mathcal{M}\in\mathfrak{D}$.
\item \textcolor{blue}{$\mathcal{K}=\mathcal{K}_{\mathfrak{D}}$ }is the
connected component of $\{U\le H+d^{(p)}\}$ containing $\widehat{\mathfrak{D}}$.
\item By \eqref{e_lv_decomp},
\begin{equation}
\mathcal{K}=\bigcup_{i=1}^{\ell}\overline{\mathcal{W}_{i}}\ ,\ \ \mathcal{M}_{0}\cap\mathcal{K}=\mathcal{M}_{0}\cap\bigcup_{i=1}^{\ell}\mathcal{W}_{i}\,,\label{e_lv_decomp-2}
\end{equation}
where \textcolor{blue}{$\mathcal{W}_{1},\,\dots,\,\mathcal{W}_{\ell}$}
denote the connected components of $\{U<H+d^{(p)}\}$ intersecting
with $\mathcal{K}$.
\item $\mathfrak{D}={\color{blue}\{\mathcal{M}_{1},\,\dots,\,\mathcal{M}_{n}\}}$,
${\color{blue}\mathcal{M}_{n+1},\,\dots,\,\mathcal{M}_{m}}\in\widehat{\mathfrak{D}}\setminus\mathfrak{D}$
for some $1\le{\color{blue}n\le m}\le\ell$, where ${\color{blue}\mathcal{M}_{i}}:=\mathcal{M}^{*}(\mathcal{W}_{i})$,
$i\in\llbracket1,\,\ell\rrbracket$, and $\mathscr{S}^{(p)}\left(\bigcup_{m+1\le i\le\ell}\mathcal{W}_{i}\right)\cap\widehat{\mathfrak{D}}=\varnothing$.
\item $U(\mathcal{M}_{1})=\cdots=U(\mathcal{M}_{n})=H$ and $U(\mathcal{M}_{n+1}),\,\dots,\,U(\mathcal{M}_{m})>H$.
\end{itemize}
Moreover, recall from \eqref{e_sigma-H+d} that:
\begin{itemize}
\item For distinct $i,\,j\in\llbracket1,\,\ell\rrbracket$, ${\color{blue}\Sigma_{i,\,j}=\Sigma_{j,\,i}}:=\overline{\mathcal{W}_{i}}\cap\overline{\mathcal{W}_{j}}=\partial\mathcal{W}_{i}\cap\partial\mathcal{W}_{j}\subset\mathcal{S}_{0}$
and
\[
U(\bm{\sigma})=H+d^{(p)}\ \text{for all}\ \bm{\sigma}\in\Sigma_{i,\,j},\,i\ne j\in\llbracket1,\,\ell\rrbracket\,.
\]
\end{itemize}

The next result generalizes Claim B stated in the proof of
\cite[Lemma 5.11]{LLS-2nd}.

\begin{lem}
\label{l_81}
Let $i\in\llbracket1,\,\ell\rrbracket$ be such that
$U(\mathcal{M}_{i})\ge H$.  For every
$\mathcal{M}\in\mathscr{S}^{(p)}(\mathcal{W}_{i})$, the Markov chain
$\{\widehat{{\bf y}}^{(p)}(t)\}_{t\ge0}$, starting from $\mathcal{M}$,
reaches $\mathcal{M}_{i}$ with positive probability.
\end{lem}

\begin{proof}
Fix $i\in\llbracket1,\,\ell\rrbracket$ such that
$U(\mathcal{M}_{i})\ge H$, and
$\mathcal{M}\in\mathscr{S}^{(p)}(\mathcal{W}_{i})$.  If
$\mathcal{M}=\mathcal{M}_{i}$, the claim is trivial. Assume henceforth that
$\mathcal{M}\ne\mathcal{M}_{i}$. Since $U(\mathcal{M}_{i})\ge H$ and
$\mathcal{M}\ne\mathcal{M}_{i}$, by Lemma \ref{l_76}-(3, 4),
$\mathcal{M}\in\mathscr{N}^{(p)}(\mathcal{W}_{i})$. By \cite[Lemma
5.8]{LLS-2nd}, there is no $\widehat{{\bf y}}^{(p)}$-recurrent class
consisting only of elements of $\mathscr{N}^{(p)}$. Therefore,
starting from $\mathcal{M}$, the Markov chain
$\{\widehat{{\bf y}}^{(p)}(t)\}_{t\ge0}$ reaches some
$\mathcal{M}'\in\mathscr{V}^{(p)}$ with positive probability.

If $\mathcal{M}'\in\mathscr{V}^{(p)}(\mathcal{W}_{i})$, then by
Lemma \ref{l_76}-(3, 4), $U(\mathcal{M}_{i})=H$ and
$\mathcal{M}' = \mathcal{M}_{i}$ so that starting from $\mathcal{M}$,
the Markov chain $\{\widehat{{\bf y}}^{(p)}(t)\}_{t\ge0}$ reaches
$\mathcal{M}_{i}$ with positive probability.

Suppose instead that $\mathcal{M}'\in\mathscr{V}^{(p)}\left((\mathcal{W}_{i})^{c}\right)$.
Then, either $\widehat{r}^{(p)}(\mathcal{M},\,\mathcal{M}')>0$ or
there exist $a\ge1$ and $\mathcal{N}_{1},\,\dots,\,\mathcal{N}_{a}\in\mathscr{S}^{(p)}$
such that 
\begin{equation}
\label{e_NNaa}
\widehat{r}^{(p)}(\mathcal{M},\,\mathcal{N}_{1})>0,\;
\widehat{r}^{(p)}(\mathcal{N}_1,\,\mathcal{N}_{2})>0,\;
\dots,
\widehat{r}^{(p)}(\mathcal{N}_{a-1},\,\mathcal{N}_{a})>0,\;
\widehat{r}^{(p)}(\mathcal{N}_a,\,\mathcal{M}')>0\;.
\end{equation}

If $\widehat{r}^{(p)}(\mathcal{M},\,\mathcal{M}')>0$, Lemma
\ref{l_exit} yields $\mathcal{M}=\mathcal{M}_{i}$, contradicting the
initial assumption $\mathcal{M}\ne\mathcal{M}_{i}$. Therefore, there
exist $a\ge1$ and
$\mathcal{N}_{1},\,\dots,\,\mathcal{N}_{a}\in\mathscr{S}^{(p)}$
satisfying \eqref{e_NNaa}. Set $\mathcal{N}_{0}=\mathcal{M}$ and let
\[
b:=\max\left\{ j\in\llbracket0,\,a\rrbracket:\mathcal{N}_{j}\in\mathscr{S}^{(p)}(\mathcal{W}_{i})\right\} \,.
\]
By Lemma \ref{l_exit},
$\mathcal{N}_{b}=\mathcal{M}^{*}(\mathcal{W}_{i})=\mathcal{M}_{i}$.
Therefore, starting from $\mathcal{M}$, the Markov chain
$\{\widehat{{\bf y}}^{(p)}(t)\}_{t\ge0}$ reaches
$\mathcal{M}_{i}$ with positive probability.
\end{proof}

The following two auxiliary lemmas relate the jump rates of
$\{\widehat{{\bf y}}^{(p)}(t)\}_{t\ge0}$ to the geometry of the level set
$\mathcal{K}$.

\begin{lem}
\label{l_82}
Let $i\in\llbracket1,\,m\rrbracket$. Then:
\begin{enumerate}
\item $\Theta(\mathcal{M}_{i},\,\widetilde{\mathcal{M}_{i}})=H+d^{(p)}$.
\item There exist $j\in\llbracket1,\,\ell\rrbracket\setminus\{i\}$ and
$\mathcal{M}\in\mathscr{S}^{(p)}(\mathcal{W}_{j})$ such that $\Sigma_{i,\,j}\ne\varnothing$
and $\widehat{r}^{(p)}(\mathcal{M}_{i},\,\mathcal{M})>0$.

\item If $\widehat{r}^{(p)}(\mathcal{M}_{i},\,\mathcal{M})>0$ for some
$\mathcal{M}\in\mathscr{S}^{(p)}\setminus\widehat{\mathfrak{D}}$, then
there exists $j\in\llbracket m+1,\,\ell\rrbracket$ such
that $\Sigma_{i,\,j}\ne\varnothing$ and
$\mathcal{M}\in\mathscr{S}^{(p)}(\mathcal{W}_{j})$.
\end{enumerate}
\end{lem}

\begin{proof}
For the first assertion, fix $i\in\llbracket1,\,m\rrbracket$.  By the
remarks in \eqref{e_lv_decomp-2},
$\mathcal{M}_{i}\in\widehat{\mathfrak{D}}$.  By the assumption
formulated at the beginning of the section, $\mathcal{M}_i$ is not a ${\bf y}^{(p)}$-absorbing state.

If $i\in\llbracket1,\,n\rrbracket$, $\mathcal{M}_{i}\in\mathfrak{D}$
and $U(\mathcal{M}_{i})=H$. Since $\mathcal{M}_{i}$ is not a
${\bf y}^{(p)}$-absorbing state, Proposition \ref{p_char} yields that
$\Xi(\mathcal{M}_{i})=d^{(p)}$. Thus,
\[
\Theta(\mathcal{M}_{i},\,\widetilde{\mathcal{M}_{i}})=U(\mathcal{M}_{i})
+\Xi(\mathcal{M}_{i})=H+d^{(p)}\,.
\]

Let $i\in\llbracket n+1,\,m\rrbracket$. By \eqref{e_lv_decomp-2},
$\mathcal{M}_{i}\in\widehat{\mathfrak{D}}\setminus\mathfrak{D}$ and
$U(\mathcal{M}_{i})>H$.  Since $U(\bm{m})>U(\mathcal{M}_{i})$ for all
$\bm{m}\in(\mathcal{M}_0\setminus\mathcal{M}_i) \cap \mathcal{W}_i$,
$\widetilde{\mathcal{M}_{i}}\subset(\mathcal{W}_{i})^{c}$.  Hence, by
Lemma \ref{l_pot1}-(2),
\begin{equation}
\Theta(\mathcal{M}_{i},\,\widetilde{\mathcal{M}_{i}})\ge H+d^{(p)}\,.
\label{e_l82-1}
\end{equation}

Fix $\mathcal{M}'\in\mathfrak{D}$. As $U(\mathcal{M}')=H < U(\mc M_i)$,
\begin{equation}
\mathcal{M}'\subset(\mathcal{W}_{i})^{c} \quad \text{and} \quad
\Theta(\mathcal{M}_{i},\,\widetilde{\mathcal{M}_{i}})
\le\Theta(\mathcal{M}_{i},\,\mathcal{M}')\,.
\label{e_l82-2}
\end{equation}
As $\mc M'$ belongs to
$\mf D$, $\mathcal{M}'\in\mathscr{V}^{(p)}(\mathcal{W}_{k}) $ for some
$k\in\llbracket1,\,n\rrbracket\setminus\{i\}$. By \cite[Lemma A.12-(1)]{LLS-2nd},
\[
\Theta(\bm{m},\,\bm{m}')=H+d^{(p)}\ \text{for all}\ \bm{m}\in\mathcal{M}_{i}\subset\mathcal{W}_{i}\ \text{and}\ \bm{m}'\in\mathcal{M}'\subset\mathcal{W}_{k}\,.
\]
Thus, $\Theta(\mathcal{M}_{i},\,\mathcal{M}')=H+d^{(p)}$, which, together with \eqref{e_l82-1} and
\eqref{e_l82-2}, completes
the proof of the first assertion.

We turn to the second assertion. By \cite[Lemma A.16-(2)]{LLS-1st},
there exists $j\in\llbracket1,\,\ell\rrbracket\setminus\{i\}$ such
that $\Sigma_{i,\,j}\ne\varnothing$. Recall from the beginning of the
section that $\Sigma_{i,\,j}=\partial\mathcal{W}_{i}\cap\partial\mathcal{W}_{j}\subset\mathcal{S}_{0}$.
Let $\bm{\sigma}\in\Sigma_{i,\,j}$. Since $\bm{\sigma}\in\partial\mathcal{W}_{i}\cap\mathcal{S}_{0}$,
by \cite[Lemma A.16-(3)]{LLS-2nd}, $\bm{\sigma}\rightsquigarrow\bm{m}$
for all $\bm{m}\in\mathcal{M}_{0}\cap\mathcal{W}_{i}$. Therefore, since $\mathcal{M}_{i}\subset\mathcal{M}_{0}\cap\mathcal{W}_{i}$, $\bm{\sigma}\rightsquigarrow\mathcal{M}_{i}$.
On the other hand, since $\bm{\sigma}\in\partial\mathcal{W}_{j}\cap\mathcal{S}_{0}$,
by \cite[Lemma A.16-(1)]{LLS-2nd}, there exists $\bm{m}'\in\mathcal{M}_{0}\cap\mathcal{W}_{j}$
such that $\bm{\sigma}\curvearrowright\bm{m}'$. 

Let $\mathcal{M}\in\mathscr{S}^{(p)}$ contain $\bm{m}'$ so that
$\bm{\sigma}\curvearrowright\mathcal{M}$. By Lemma \ref{l_equi_lv}-(2),
$\mathcal{W}_{j}$ does not separate $(p)$-states. Since
$\bm{m}'\in\mathcal{W}_{j}$,
$\mathcal{M}\in\mathscr{S}^{(p)}(\mathcal{W}_{j})$. Recall from the
beginning of the proof that
$\mathcal{M}_{i}\in\widehat{\mathfrak{D}}$.  Since
$\widehat{\mathfrak{D}}$ does not contain ${\bf y}^{(p)}$-absorbing
states [because $\mathfrak{D}$ does not contain such states], by
Proposition \ref{p_char}, $\Xi(\mathcal{M}_{i})\le d^{(p)}$.  By
Proposition \ref{p: tree}-(4), it remains to prove that
$\mathcal{M}_{i}\to\mathcal{M}$, i.e., from \eqref{eq:con_gate}, that
\begin{equation}
U(\boldsymbol{\sigma})=\Theta(\mathcal{M}_{i},\,\widetilde{\mathcal{M}_{i}})
=\Theta(\mathcal{M}_{i},\,\mathcal{M})
\;\;\text{and}\;\;\mathcal{M}\,\curvearrowleft\,\boldsymbol{\sigma}\,\rightsquigarrow\,\mathcal{M}_{i}\,.\label{e_l82-4}
\end{equation}
The second property holds by definition of $\bm{\sigma}$ and
$\mathcal{M}$. Since
$\mathcal{M}\in\mathscr{S}^{(p)}(\mathcal{W}_{j})$,
$\mathcal{M}\subset(\mathcal{W}_{i})^{c}$. Hence, by Lemma
\ref{l_pot1}-(2),
$\Theta(\mathcal{M}_{i},\,\mathcal{M})\ge H+d^{(p)}$. On the other
hand, by \cite[Lemma A.6-(3)]{LLS-2nd} and \eqref{e_sigma-H+d},
$\Theta(\mathcal{M}_{i},\,\mathcal{M})\le U(\bm{\sigma})=H+d^{(p)}$.
This, together with the first assertion of the lemma, proves
\eqref{e_l82-4}.

We turn to the third assertion. Let
$\mathcal{M}\in\mathscr{S}^{(p)}\setminus\widehat{\mathfrak{D}}$ be
such that $\widehat{r}^{(p)}(\mathcal{M}_{i},\,\mathcal{M})>0$.  By
\eqref{eq:con_gate} and Proposition \ref{p: tree}-(4), there exists
$\bm{\sigma}\in\mathcal{S}_{0}$ such that
$\mathcal{M}_{i}\to_{\bm{\sigma}}\mathcal{M}$, i.e.,
\[
U(\boldsymbol{\sigma})=\Theta(\mathcal{M}_{i},\,\widetilde{\mathcal{M}_{i}})=\Theta(\mathcal{M}_{i},\,\mathcal{M})\;\;\text{and}\;\;\mathcal{M}\,\curvearrowleft\,\boldsymbol{\sigma}\,\rightsquigarrow\,\mathcal{M}_{i}\,.
\]
Pick $\bm{m}_{1}\in\mathcal{M}_{i}$ and $\bm{m}_{2}\in\mathcal{M}$
satisfying
$\Theta(\mathcal{M}_{i},\,\mathcal{M})=\Theta(\bm{m}_{1},\,\bm{m}_{2})$.
By the first assertion of the lemma,
$\Theta(\bm{m}_{1},\,\bm{m}_{2})=\Theta(\mathcal{M}_{i},\,\mathcal{M})=\Theta(\mathcal{M}_{i},\,\widetilde{\mathcal{M}_{i}})=H+d^{(p)}$.
Since $\mathcal{K}$ is the connected component of
$\{U\le H+d^{(p)}\}=\{U\le\Theta(\bm{m}_{1},\,\bm{m}_{2})\}$
containing $\bm{m}_{1}$, by \cite[Lemma A.5]{LLS-2nd}
$\bm{m}_{2}\in\mathcal{K}$ as well. By \eqref{e_lv_decomp-2}, there
exists $j\in\llbracket1,\,\ell\rrbracket$ such that
$\bm{m}_{2}\in\mathcal{W}_{j}$.  Hence, by Lemma \ref{l_equi_lv}-(2),
$\mathcal{M}\in\mathscr{S}^{(p)}(\mathcal{W}_{j})$.  Moreover, since
$\Theta(\mathcal{M}_{i},\,\mathcal{M})=H+d^{(p)}$, by Lemma
\ref{l_pot1}-(1)
$\mathcal{M}\notin\mathscr{S}^{(p)}(\mathcal{W}_{i})$, and hence
$j\ne i$.

We now claim that $\Sigma_{i,\,j}\ne\varnothing$.  Since
$\mathcal{M}\,\curvearrowleft\,\boldsymbol{\sigma}\,
\rightsquigarrow\,\mathcal{M}_{i}$, there exist
$\bm{m}_{3}\in\mathcal{M}$ and $\bm{m}_{4}\in\mathcal{M}_{i}$ such
that
$\bm{m}_{3}\curvearrowleft\boldsymbol{\sigma}\rightsquigarrow\bm{m}_{4}$.
Since $\mathcal{W}_{j}$ and $\mathcal{W}_{i}$ are the connected
components of $\{U<U(\bm{\sigma})\}=\{U<H+d^{(p)}\}$ containing
$\bm{m}_{3}$ and $\bm{m}_{4}$, respectively, and
$\bm{m}_{3}\curvearrowleft\boldsymbol{\sigma}\rightsquigarrow\bm{m}_{4}$,
\cite[Lemma A.17]{LLS-2nd} implies
$\bm{\sigma}\in\partial\mathcal{W}_{i}\cap\partial\mathcal{W}_{j}$ so
that
$\Sigma_{i,\,j}=\partial\mathcal{W}_{i}\cap\partial\mathcal{W}_{j}\ne\varnothing$.

It remains to prove that $j\in\llbracket m+1,\,\ell\rrbracket$.
Suppose by contradiction that $j\in\llbracket1,\,m\rrbracket$. Recall
from the  beginning of the section that
$\mathcal{M}_{j}\in\widehat{\mathfrak{D}}$.  By Lemma \ref{l_81},
since $\mathcal{M}\in\mathscr{S}^{(p)}(\mathcal{W}_{j})$, starting
from $\mathcal{M}$, the Markov chain
$\{\widehat{{\bf y}}^{(p)}(t)\}_{t\ge0}$ reaches $\mathcal{M}_{j}$
with positive probability. Since
$\mathcal{M}_{i},\,\mathcal{M}_{j}\in\widehat{\mathfrak{D}}$, it
follows from $\widehat{r}^{(p)}(\mathcal{M}_{i},\,\mathcal{M})>0$ that
$\mathcal{M}\in\widehat{\mathfrak{D}}$ as well, which is a
contradiction. Hence $j\in\llbracket m+1,\,\ell\rrbracket$, proving
the third assertion.
\end{proof}

Let
\[
{\color{blue}\omega_{i,\,j}}:=\sum_{\bm{\sigma}\in\Sigma_{i,\,j}}\omega(\bm{\sigma})\,.
\]
Note that $\omega_{i,\,j}=0$ if $\Sigma_{i,\,j}=\varnothing$. The
next result corresponds to \cite[display (13.6)]{LLS-2nd}.
\begin{lem}
\label{l_83}Fix $i\in\llbracket1,\,m\rrbracket$. Then, for every
$j\in\llbracket1,\,\ell\rrbracket\setminus\{i\}$,
\[
\sum_{\mathcal{M}\in\mathscr{S}^{(p)}(\mathcal{W}_{j})}\widehat{r}^{(p)}(\mathcal{M}_{i},\,\mathcal{M})=\frac{\omega_{i,\,j}}{\nu(\mathcal{M}_{i})}\,.
\]
\end{lem}

\subsection{\label{subsec_rev}Local reversibility}

Decompose 
\[
\widehat{\mathfrak{D}}=\bigcup_{i\in\llbracket1,\,m\rrbracket}\Big(\widehat{\mathfrak{D}}\cap\mathscr{S}^{(p)}(\mathcal{W}_{i})\Big)\,.
\]
Recall from the beginning of the section that $\mathcal{M}_{i}\in\widehat{\mathfrak{D}}\cap\mathscr{S}^{(p)}(\mathcal{W}_{i})$
for each $i\in\llbracket1,\,m\rrbracket$, that $\mathfrak{D}=\{\mathcal{M}_{1},\,\dots,\,\mathcal{M}_{n}\}$,
and that $\{\mathcal{M}_{n+1},\,\dots,\,\mathcal{M}_{m}\}\subset\widehat{\mathfrak{D}}\setminus\mathfrak{D}$.
The next lemma shows that this decomposition satisfies the assumptions
of Lemma \ref{l_A3}.
\begin{lem}
\label{l_rev}Fix $i\in\llbracket1,\,m\rrbracket$. 
\begin{enumerate}
\item For every $\mathcal{M}\in\widehat{\mathfrak{D}}\cap\mathscr{S}^{(p)}(\mathcal{W}_{i})$,
\[
\begin{cases}
\sum_{\mathcal{M}'\in\mathscr{S}^{(p)}\setminus\mathscr{S}^{(p)}(\mathcal{W}_{i})}\widehat{r}^{(p)}(\mathcal{M},\,\mathcal{M}')=0 & \text{if}\ \mathcal{M}\ne\mathcal{M}_{i}\,,\\
\sum_{\mathcal{M}'\in\mathscr{S}^{(p)}\setminus\mathscr{S}^{(p)}(\mathcal{W}_{i})}\widehat{r}^{(p)}(\mathcal{M},\,\mathcal{M}')>0 & \text{if}\ \mathcal{M}=\mathcal{M}_{i}\,.
\end{cases}
\]
\item Suppose $m\ge2$. Then, for every
$j\in\llbracket1,\,m\rrbracket\setminus\{i\}$,
\[
\nu(\mathcal{M}_{i})\sum_{\mathcal{M}\in\mathscr{S}^{(p)}(\mathcal{W}_{j})}\widehat{r}^{(p)}(\mathcal{M}_{i},\,\mathcal{M})=\nu(\mathcal{M}_{j})\sum_{\mathcal{M}\in\mathscr{S}^{(p)}(\mathcal{W}_{i})}\widehat{r}^{(p)}(\mathcal{M}_{j},\,\mathcal{M})=\omega_{i,\,j}\,.
\]
\end{enumerate}
\end{lem}

\begin{proof}
We prove the first assertion. Let
$\mathcal{M}\in\widehat{\mathfrak{D}}\cap\mathscr{S}^{(p)}(\mathcal{W}_{i})$
and
$\mathcal{M}'\in\mathscr{S}^{(p)}\setminus\mathscr{S}^{(p)}(\mathcal{W}_{i})$
be such that $\widehat{r}^{(p)}(\mathcal{M},\,\mathcal{M}')>0$. By
Proposition \ref{p: tree}-(4), $\mathcal{M}\to\mathcal{M}'$. Thus, by
Lemma \ref{l_exit} $\mathcal{M}=\mathcal{M}_{i}$. Hence,
\begin{equation}
\label{e_l_rev-1}
\sum_{\mathcal{M}'\in\widehat{\mathfrak{D}}\setminus\mathscr{S}^{(p)}(\mathcal{W}_{i})}\widehat{r}^{(p)}(\mathcal{M},\,\mathcal{M}')=0\quad\text{if}\quad\mathcal{M}\ne\mathcal{M}_{i}\,.
\end{equation}

On the other hand, by Lemma \ref{l_82}-(2), there exist
$k\in\llbracket1,\,\ell\rrbracket\setminus\{i\}$ and
$\mathcal{M}'\in\widehat{\mathfrak{D}}\cap\mathscr{S}^{(p)}(\mathcal{W}_{k})$
such that $\Sigma_{i,\,k}\ne\varnothing$ and
$\widehat{r}^{(p)}(\mathcal{M}_{i},\,\mathcal{M}')>0$, which, together
with \eqref{e_l_rev-1}, proves the first assertion.

We turn to the second assertion. Suppose $m\ge2$.
By Lemma \ref{l_83}, for $j\in\llbracket1,\,m\rrbracket\setminus\{i\}$,
\[
\nu(\mathcal{M}_{i})\sum_{\mathcal{M}\in\mathscr{S}^{(p)}(\mathcal{W}_{j})}\widehat{r}^{(p)}(\mathcal{M}_{i},\,\mathcal{M})=\omega_{i,\,j}=\nu(\mathcal{M}_{j})\sum_{\mathcal{M}\in\mathscr{S}^{(p)}(\mathcal{W}_{i})}\widehat{r}^{(p)}(\mathcal{M}_{j},\,\mathcal{M})\,,
\]
and this completes the proof.
\end{proof}
Now, we are ready to prove Proposition \ref{p_rev}.
\begin{proof}[Proof of Proposition \ref{p_rev}]
Since $|\mathfrak{D}|\ge2$, Lemma \ref{l_equi_lv}-(3) yields $m\ge n\ge2$.
By Lemma \ref{l_rev}, the equivalence class $\mathfrak{D}$ of the
Markov chain $\{{\bf y}^{(p)}(t)\}_{t\ge0}$ satisfies the assumptions
of Lemma \ref{l_A3}. Hence, by part (1) of that lemma, the reflected
chain $\{{\bf y}_{\mathfrak{D}}^{(p)}(t)\}_{t\ge0}$ is reversible
with respect to the restriction of the measure $\nu$ to $\mathfrak{D}$.
\end{proof}

\subsection{\label{subsec_pf_p_cap}Proof of Proposition \ref{p_cap}}

The next lemma relates the functions
$\{{\bf h}_{i}^{(p)}\}_{i\in\llbracket1,\,n\rrbracket}$, defined in \eqref{e_def_h}, to the limiting Markov chain
$\{{\bf y}^{(p)}(t)\}_{t\ge0}$.
\begin{lem}
\label{l_85}For $i\in\llbracket1,\,n\rrbracket$,
\[
\sum_{1\le a<b\le\ell}\left|{\bf h}_{i}^{(p)}(a)-{\bf h}_{i}^{(p)}(b)\right|^{2}\omega_{a,\,b}=\nu(\mathcal{M}_{i})\sum_{\mathcal{M}\in\mathscr{V}^{(p)}\setminus\{\mathcal{M}_{i}\}}r^{(p)}(\mathcal{M}_{i},\,\mathcal{M})\,.
\]
If $n\ge2$, then for $i\ne j\in\llbracket1,\,n\rrbracket$,
\[
\begin{aligned} & \sum_{1\le a<b\le\ell}\left[{\bf h}_{i}^{(p)}(a)-{\bf h}_{i}^{(p)}(b)\right]\left[{\bf h}_{j}^{(p)}(a)-{\bf h}_{j}^{(p)}(b)\right]\omega_{a,\,b}\\
 & =-\frac{1}{2}\left(\nu(\mathcal{M}_{i})\,r^{(p)}(\mathcal{M}_{i},\,\mathcal{M}_{j})+\nu(\mathcal{M}_{j})\,r^{(p)}(\mathcal{M}_{j},\,\mathcal{M}_{i})\right)\,.
\end{aligned}
\]
\end{lem}

\begin{proof}
Fix $i\in\llbracket1,\,n\rrbracket$ and set ${\bf g}_{i}:=\delta_{\mathcal{M}_{i}}:\mathscr{V}^{(p)}\to\mathbb{R}$.
Let $\widehat{{\bf g}_{i}}:\mathscr{S}^{(p)}\to\mathbb{R}$ be the
harmonic extension of ${\bf g}_{i}$. By \eqref{e_harmonic}, for
each $k\in\llbracket1,\,m\rrbracket$,
\begin{equation}
\widehat{{\bf g}_{i}}(\mathcal{M}_{k})=\widehat{\mathcal{Q}}^{(p)}_{\mathcal{M}_{k}}\left[H_{\mathscr{V}^{(p)}}=H_{\mathcal{M}_{i}}\right]={\bf h}_{i}^{(p)}(k)\,.\label{e_l85-1}
\end{equation}
 Since ${\bf h}_{i}^{(p)}(a)=0$ for $a\in\llbracket m+1,\,\ell\rrbracket$,
\eqref{e_l85-1} yields
\[
\begin{aligned} & \sum_{1\le a<b\le\ell}\left|{\bf h}_{i}^{(p)}(a)-{\bf h}_{i}^{(p)}(b)\right|^{2}\omega_{a,\,b}\\
 & =\sum_{a\in\llbracket1,\,m\rrbracket}\sum_{b\in\llbracket a+1,\,m\rrbracket}\left|{\bf h}_{i}^{(p)}(a)-{\bf h}_{i}^{(p)}(b)\right|^{2}\omega_{a,\,b}+\sum_{a\in\llbracket1,\,m\rrbracket}\sum_{b\in\llbracket m+1,\,\ell\rrbracket}{\bf h}_{i}^{(p)}(a)^{2}\omega_{a,\,b}\\
 & =\frac{1}{2}\sum_{a,\,b\in\llbracket1,\,m\rrbracket}\left|{\bf h}_{i}^{(p)}(a)-{\bf h}_{i}^{(p)}(b)\right|^{2}\omega_{a,\,b}+\sum_{a\in\llbracket1,\,m\rrbracket}\sum_{b\in\llbracket m+1,\,\ell\rrbracket}{\bf h}_{i}^{(p)}(a)^{2}\omega_{a,\,b}\\
 & =\frac{1}{2}\sum_{a,\,b\in\llbracket1,\,m\rrbracket}\left|\widehat{{\bf g}_{i}}(\mathcal{M}_{a})-\widehat{{\bf g}_{i}}(\mathcal{M}_{b})\right|^{2}\omega_{a,\,b}+\sum_{a\in\llbracket1,\,m\rrbracket}\widehat{{\bf g}_{i}}(\mathcal{M}_{a})^{2}\sum_{b\in\llbracket m+1,\,\ell\rrbracket}\omega_{a,\,b}\,.
\end{aligned}
\]
By Lemma \ref{l_83}, the last term can be written as
\[
\sum_{a\in\llbracket1,\,m\rrbracket}\widehat{{\bf g}_{i}}(\mathcal{M}_{a})^{2}\nu(\mathcal{M}_{a})\sum_{b\in\llbracket m+1,\,\ell\rrbracket}\sum_{\mathcal{M}\in\mathscr{S}^{(p)}(\mathcal{W}_{b})}\widehat{r}^{(p)}(\mathcal{M}_{a},\,\mathcal{M})\,.
\]

Fix $a\in\llbracket1,\,m\rrbracket$. Let
$\mathcal{M}\in\mathscr{S}^{(p)}\setminus\widehat{\mathfrak{D}}$ be
such that $\widehat{r}^{(p)}(\mathcal{M}_{a},\,\mathcal{M})>0$.  By
Lemma \ref{l_82}-(3), there exists
$b\in\llbracket m+1,\,\ell\rrbracket$ such that
$\mathcal{M}\in\mathscr{S}^{(p)}(\mathcal{W}_{b})$. On the other hand,
let $b\in\llbracket m+1,\,\ell\rrbracket$ and
$\mathcal{M}\in\mathscr{S}^{(p)}(\mathcal{W}_{b})$ be such that
$\widehat{r}^{(p)}(\mathcal{M}_{a},\,\mathcal{M})>0$.  Then, since
$\widehat{\mathfrak{D}}\cap\mathscr{S}^{(p)}(\mathcal{W}_{b})=\varnothing$,
$\mathcal{M}\in\mathscr{S}^{(p)}\setminus\widehat{\mathfrak{D}}$.
Therefore, for $a\in\llbracket1,\,m\rrbracket$,
\[
\sum_{b\in\llbracket m+1,\,\ell\rrbracket}\sum_{\mathcal{M}\in\mathscr{S}^{(p)}(\mathcal{W}_{b})}\widehat{r}^{(p)}(\mathcal{M}_{a},\,\mathcal{M})=\sum_{\mathcal{M}\in\mathscr{S}^{(p)}\setminus\widehat{\mathfrak{D}}}\widehat{r}^{(p)}(\mathcal{M}_{a},\,\mathcal{M})\,.
\]
Hence, by the above equalities, 
\begin{equation}
\begin{aligned} & \sum_{1\le a<b\le\ell}\left|{\bf h}_{i}^{(p)}(a)-{\bf h}_{i}^{(p)}(b)\right|^{2}\omega_{a,\,b}\\
 & =\frac{1}{2}\sum_{a,\,b\in\llbracket1,\,m\rrbracket}\left|\widehat{{\bf g}_{i}}(\mathcal{M}_{a})-\widehat{{\bf g}_{i}}(\mathcal{M}_{b})\right|^{2}\omega_{a,\,b}+\sum_{a\in\llbracket1,\,m\rrbracket}\widehat{{\bf g}_{i}}(\mathcal{M}_{a})^{2}\nu(\mathcal{M}_{a})\sum_{\mathcal{M}\in\mathscr{S}^{(p)}\setminus\widehat{\mathfrak{D}}}\widehat{r}^{(p)}(\mathcal{M}_{a},\,\mathcal{M})\,.
\end{aligned}
\label{e_l85-2}
\end{equation}

If $m\ge2$, then by Lemma \ref{l_rev}, the Markov
chain $\{\widehat{{\bf y}}^{(p)}(t)\}_{t\ge0}$ and the equivalence class
$\widehat{\mathfrak{D}}$ satisfy the assumptions of Lemma \ref{l_A3}
under the decomposition $\widehat{\mathfrak{D}}=\bigcup_{i\in\llbracket1,\,m\rrbracket}\Big(\widehat{\mathfrak{D}}\cap\mathscr{S}^{(p)}(\mathcal{W}_{i})\Big)$,
$\{\mathcal{M}_{1},\,\dots,\,\mathcal{M}_{m}\}\subset\widehat{\mathfrak{D}}$,
with measure $\nu$ conditioned on $\{\mathcal{M}_{1},\,\dots,\,\mathcal{M}_{m}\}$.
Then, \eqref{e_l85-2} and Lemma \ref{l_A3}-(2) imply
\[
\begin{aligned}\sum_{1\le a<b\le\ell}\left|{\bf h}_{i}^{(p)}(a)-{\bf h}_{i}^{(p)}(b)\right|^{2}\omega_{a,\,b} & =-\sum_{k\in\llbracket1,\,n\rrbracket}\nu(\mathcal{M}_{k}){\bf g}_{i}(\mathcal{M}_{k})\mathfrak{L}^{(p)}{\bf g}_{i}(\mathcal{M}_{k})\\
 & =\nu(\mathcal{M}_{i})\sum_{\mathcal{M}'\in\mathscr{V}^{(p)}\setminus\{\mathcal{M}_{i}\}}r^{(p)}(\mathcal{M}_{i},\,\mathcal{M}')\,,
\end{aligned}
\]
where the last equality follows from the definition of ${\bf g}_{i}$.

Next, suppose $m=1$, and hence $i=1$. By Lemma \ref{l_rev}-(1),
the Markov chain $\{\widehat{{\bf y}}^{(p)}(t)\}_{t\ge0}$ and the equivalence
class $\widehat{\mathfrak{D}}$ satisfy the assumptions of Lemma \ref{l_A4}
with $\mathfrak{D}=\{\mathcal{M}_{1}\}$. Then, \eqref{e_l85-2} and
Lemma \ref{l_A4} yield
\[
\begin{aligned}\sum_{1\le a<b\le\ell}\left|{\bf h}_{1}^{(p)}(a)-{\bf h}_{1}^{(p)}(b)\right|^{2}\omega_{a,\,b} & =\widehat{{\bf g}_{1}}(\mathcal{M}_{1})^{2}\nu(\mathcal{M}_{1})\sum_{\mathcal{M}\in\mathscr{S}^{(p)}\setminus\widehat{\mathfrak{D}}}\widehat{r}^{(p)}(\mathcal{M}_{1},\,\mathcal{M})\\
 & =-\nu(\mathcal{M}_{1}){\bf g}_{1}(\mathcal{M}_{1})\mathfrak{L}^{(p)}{\bf g}_{1}(\mathcal{M}_{1})\\
 & =\nu(\mathcal{M}_{1})\sum_{\mathcal{M}'\in\mathscr{V}^{(p)}\setminus\{\mathcal{M}_{1}\}}r^{(p)}(\mathcal{M}_{1},\,\mathcal{M}')\,,
\end{aligned}
\]
where the last equality follows from the definition of ${\bf g}_{1}$.
This completes the proof of the first assertion.

We turn to the second assertion. Assume $n\ge2$, fix $j\in\llbracket1,\,n\rrbracket\setminus\{i\}$,
and define ${\bf h}_{i,\,j}^{(p)}:={\bf h}_{i}^{(p)}+{\bf h}_{j}^{(p)}$
and ${\bf g}_{i,\,j}:=\delta_{\mathcal{M}_{i}}+\delta_{\mathcal{M}_{j}}$.
By the same argument as above,
\[
\begin{aligned} & \sum_{1\le a<b\le\ell}\left|{\bf h}_{i,\,j}^{(p)}(a)-{\bf h}_{i,\,j}^{(p)}(b)\right|^{2}\omega_{a,\,b}\\
 & =-\sum_{k\in\llbracket1,\,n\rrbracket}\nu(\mathcal{M}_{k}){\bf g}_{i,\,j}(\mathcal{M}_{k})\mathfrak{L}^{(p)}{\bf g}_{i,\,j}(\mathcal{M}_{k})\\
 & =\nu(\mathcal{M}_{i})\sum_{\mathcal{M}'\in\mathscr{V}^{(p)}\setminus\{\mathcal{M}_{i},\,\mathcal{M}_{j}\}}r^{(p)}(\mathcal{M}_{i},\,\mathcal{M}')+\nu(\mathcal{M}_{j})\sum_{\mathcal{M}'\in\mathscr{V}^{(p)}\setminus\{\mathcal{M}_{i},\,\mathcal{M}_{j}\}}r^{(p)}(\mathcal{M}_{j},\,\mathcal{M}')\,.
\end{aligned}
\]
Therefore, since
\[
\begin{aligned}
& 2\left[{\bf h}_{i}^{(p)}(a)-{\bf h}_{i}^{(p)}(b)\right]\left[{\bf h}_{j}^{(p)}(a)-{\bf h}_{j}^{(p)}(b)\right]\\
& =\left[{\bf h}_{i,\,j}^{(p)}(a)-{\bf h}_{i,\,j}^{(p)}(b)\right]^{2}-\left[{\bf h}_{i}^{(p)}(a)-{\bf h}_{i}^{(p)}(b)\right]^{2}-\left[{\bf h}_{j}^{(p)}(a)-{\bf h}_{j}^{(p)}(b)\right]^{2}\,,
\end{aligned}
\]
we conclude that
\[
\begin{aligned} & \sum_{1\le a<b\le\ell}\left[{\bf h}_{i}^{(p)}(a)-{\bf h}_{i}^{(p)}(b)\right]\left[{\bf h}_{j}^{(p)}(a)-{\bf h}_{j}^{(p)}(b)\right]\omega_{a,\,b}\\
 & =-\frac{1}{2}\left(\nu(\mathcal{M}_{i})\,r^{(p)}(\mathcal{M}_{i},\,\mathcal{M}_{j})+\nu(\mathcal{M}_{j})\,r^{(p)}(\mathcal{M}_{j},\,\mathcal{M}_{i})\right)\,.
\end{aligned}
\]
\end{proof}
We now prove Proposition \ref{p_cap}.
\begin{proof}[Proof of Proposition \ref{p_cap}]
 Recall that $\nu_{\star}$ is defined in \eqref{e_def_nu}. By \eqref{e_sigma-H+d}
and \cite[Lemma 8.4]{LMS-arma},
\[
\lim_{\epsilon\to0}e^{H/\epsilon}\theta_{\epsilon}^{(p)}\epsilon\int_{\mathcal{B}_{\epsilon}^{\bm{\sigma}}}|\nabla p_{\epsilon}^{\bm{\sigma}}|^{2}d\pi_{\epsilon}=\frac{\omega(\bm{\sigma})}{\nu_{\star}}\,.
\]
Hence, for $i,\,j\in\llbracket1,\,n\rrbracket$,
\[
\begin{aligned} & \lim_{\epsilon\to0}e^{H/\epsilon}\theta_{\epsilon}^{(p)}\epsilon\int_{\mathbb{R}^{d}}\Phi_{i}^{\epsilon}\cdot\Phi_{j}^{\epsilon}d\pi_{\epsilon}\\
 & =\sum_{a<b\in\llbracket1,\,\ell\rrbracket}\left[{\bf h}_{i}^{(p)}(a)-{\bf h}_{i}^{(p)}(b)\right]\left[{\bf h}_{j}^{(p)}(a)-{\bf h}_{j}^{(p)}(b)\right]\sum_{\bm{\sigma}\in\Sigma_{a,\,b}}\lim_{\epsilon\to0}e^{H/\epsilon}\theta_{\epsilon}^{(p)}\epsilon\int_{\mathcal{B}_{\epsilon}^{\bm{\sigma}}}|\nabla p_{\epsilon}^{\bm{\sigma}}|^{2}d\pi_{\epsilon}\\
 & =\frac{1}{\nu_{\star}}\sum_{a<b\in\llbracket1,\,\ell\rrbracket}\left[{\bf h}_{i}^{(p)}(a)-{\bf h}_{i}^{(p)}(b)\right]\left[{\bf h}_{j}^{(p)}(a)-{\bf h}_{j}^{(p)}(b)\right]\omega_{a,\,b}\,.
\end{aligned}
\]
Lemma \ref{l_85} then completes the proof.
\end{proof}

\appendix

\section{Markov chains}

In this appendix, we present general results on Markov chains on
finite state spaces. Let
\textcolor{blue}{$\mathscr{V}\subset\mathscr{S}$} be nonempty finite
sets, and \textcolor{blue}{$\{\widehat{{\bf y}}(t)\}_{t\ge0}$} denote
a continuous-time Markov chain on $\mathscr{S}$ with jump rates
${\color{blue}\widehat{r}}:\mathscr{S}\times\mathscr{S}\to[0,\,\infty)$.
Mind that we do not assume $\widehat{{\bf y}} (\cdot)$ to be irreducible.

Assume that $\mathscr{V}$ contains at least one state from each closed
irreducible class of $\{\widehat{{\bf y}}(t)\}_{t\ge0}$.  Under this
assumption, \cite[display (B.1)]{LLS-2nd} holds by \cite[Lemma
B.1]{LLS-2nd}, and hence the trace process of
$\{\widehat{{\bf y}}(t)\}_{t\ge0}$ on $\mathscr{V}$ is well defined.

Denote by \textcolor{blue}{$\{{\bf y}(t)\}_{t\ge0}$}
this trace process, and by ${\color{blue}r}:\mathscr{V}\times\mathscr{V}\to[0,\,\infty)$
its jump rates.
Let \textcolor{blue}{$\widehat{\mathfrak{L}}$}
and \textcolor{blue}{$\mathfrak{L}$} be the infinitesimal generators
of $\{\widehat{{\bf y}}(t)\}_{t\ge0}$ and $\{{\bf y}(t)\}_{t\ge0}$,
respectively. Finally, denote by \textcolor{blue}{$\widehat{\mathcal{Q}}_{x}$}
the law of $\{\widehat{{\bf y}}(t)\}_{t\ge0}$
starting at $x\in\mathscr{S}$.

\subsection{\label{sec_harmonic} Harmonic extension}

For any function ${\bf f}:\mathscr{V}\to\mathbb{R}$, denote by ${\color{blue}\widehat{{\bf f}}}:\mathscr{S}\to\mathbb{R}$
its harmonic extension, defined by
\[
\begin{cases}
\widehat{{\bf f}}(x)={\bf f}(x) & x\in\mathscr{V}\,,\\
\widehat{\mathfrak{L}}\widehat{{\bf f}}(x)=0 & x\in\mathscr{S}\setminus\mathscr{V}\,.
\end{cases}
\]
It is well known that the harmonic extension admits the stochastic representation
\begin{equation}
\widehat{{\bf f}}(x)=\sum_{y\in\mathscr{V}}\widehat{\mathcal{Q}}_{x}\left[H_{\mathscr{V}}=H_{y}\right]{\bf f}(y)\quad \text{for}\ x\in\mathscr{S}\,.\label{e_harmonic}
\end{equation}

The following result is \cite[Lemma A.1]{BGL-22AAP}.

\begin{lem}
\label{l_A1}For all ${\bf f}:\mathscr{V}\to\mathbb{R}$ and
$x\in\mathscr{V}$, $\mathfrak{L}{\bf f}(x)=\widehat{\mathfrak{L}}\widehat{{\bf f}}(x)$.
\end{lem}

\subsection{Equivalence classes}

For an equivalence class $\mathfrak{D}\subset\mathscr{V}$ of $\{{\bf y}(t)\}_{t\ge0}$,
denote by ${\color{blue}\widehat{\mathfrak{D}}}\subset\mathscr{S}$
the equivalence class of $\{\widehat{{\bf y}}(t)\}_{t\ge0}$ containing
$\mathfrak{D}$.

\begin{lem}
\label{l_A2}Fix an equivalence class $\mathfrak{D}$ of $\{{\bf y}(t)\}_{t\ge0}$.
Let ${\bf f}:\mathscr{V}\to\mathbb{R}$ be such that ${\bf f}(x)=0$ for
all $x\notin\mathfrak{D}$. Then, $\widehat{{\bf f}}(x)=0$ for every
$x\notin\widehat{\mathfrak{D}}$ such that $\widehat{r}(y,\,x)>0$
for some $y\in\widehat{\mathfrak{D}}$.
\end{lem}

\begin{proof}
Fix $x\notin\widehat{\mathfrak{D}}$ such that $\widehat{r}(y,\,x)>0$
for some $y\in\widehat{\mathfrak{D}}$. We claim that the Markov chain
$\{\widehat{{\bf y}}(t)\}_{t\ge0}$ cannot reach
$\widehat{\mathfrak{D}}$ starting from $x$:
\begin{equation}
\widehat{\mathcal{Q}}_{x}\left[H_{\widehat{\mathfrak{D}}}
<\infty\right]=0\,.
\label{e_lA2-1}
\end{equation} 
Indeed, if the Markov chain $\{\widehat{{\bf y}}(t)\}_{t\ge0}$ could
reach $\widehat{\mathfrak{D}}$ starting from $x$, since
$\widehat{r}(y,\,x)>0$, $x$ would belong to $\widehat{\mathfrak{D}}$,
which is a contradiction.

On the other hand, since $\mathscr{V}$ contains at least one element
of each closed irreducible class of $\{\widehat{{\bf y}}(t)\}_{t\ge0}$,
\begin{equation}
\widehat{\mathcal{Q}}_{x}\left[H_{\mathscr{V}}
=\infty\right]=0\,.\label{e_lA2-2}
\end{equation}

By \eqref{e_lA2-1} and \eqref{e_lA2-2},
$\widehat{\mathcal{Q}}_{x}\left[H_{\mathscr{V}}=H_{z}\right]=0$ for
all $z\in\mathfrak{D}$. Since ${\bf f}(z)=0$ for all
$z\notin\mathfrak{D}$, the harmonic representation \eqref{e_harmonic}
yields
\[
\widehat{{\bf f}}(x)=\sum_{z\in\mathscr{V}}\widehat{\mathcal{Q}}_{x}\left[H_{\mathscr{V}}=H_{z}\right]{\bf f}(z)=0\,.
\]
\end{proof}

For any equivalence class $\mathfrak{D}$ of $\{{\bf y}(t)\}_{t\ge0}$,
let \textcolor{blue}{$\{{\bf y}_{\mathfrak{D}}(t)\}_{t\ge0}$}
denote the Markov chain $\{{\bf y}(t)\}_{t\ge0}$ reflected at $\mathfrak{D}$. That is,
$\{{\bf y}_{\mathfrak{D}}(t)\}_{t\ge0}$ is the $\mathfrak{D}$-valued Markov chain
with jump rates
\[
r_{\mathfrak{D}}(\mathcal{M},\,\mathcal{M}')=r(\mathcal{M},\,\mathcal{M}')\,,\quad \mathcal{M},\,\mathcal{M}'\in\mathfrak{D}\,.
\]

\begin{lem}
\label{l_A3}
Fix an equivalence class $\mathfrak{D}\subset\mathscr{V}$
of $\{{\bf y}(t)\}_{t\ge0}$. Suppose that there exist $n,\,m\in\mathbb{N}$
such that $m\ge2$, $1\le n\le m$, and $\widehat{\mathfrak{D}}$ admits
a decomposition
\begin{equation}
\widehat{\mathfrak{D}}=\bigcup_{i\in\llbracket1,\,m\rrbracket}\widehat{\mathfrak{D}}_{i}\,,\label{e_lA3-0}
\end{equation}
satisfying the following.
\begin{itemize}
\item[(a)] For each $i\in\llbracket1,\,m\rrbracket$,
there exists $x_{i}\in\widehat{\mathfrak{D}}_{i}$
such that
\begin{equation}
\begin{cases}
\sum_{y\in\mathscr{S}\setminus\widehat{\mathfrak{D}}_{i}}\widehat{r}(x,\,y)=0\
\text{for all}\ x\in\widehat{\mathfrak{D}}_{i}\setminus\{x_{i}\}\,,\\
\sum_{y\in\mathscr{S}\setminus\widehat{\mathfrak{D}}_{i}}\widehat{r}(x_{i},\,y)>0\,.
\end{cases}\label{e_lA3-1}
\end{equation}
\item[(b)]  $\mathfrak{D}=\{x_{1},\,\dots,\,x_{n}\}$. In particular, $\widehat{\mathfrak{D}}_{i}\cap\mathscr{V}=\varnothing$
for $i\in\llbracket n+1,\,m\rrbracket$.
\item[(c)]  There exists a measure $\rho$ on $\{x_{1},\,\dots,\,x_{m}\}$ such
that for all $i\ne j\in\llbracket1,\,m\rrbracket$,
\begin{equation}
\rho(x_{i})\sum_{y\in\widehat{\mathfrak{D}}_{j}}\widehat{r}(x_{i},\,y)=\rho(x_{j})\sum_{y\in\widehat{\mathfrak{D}}_{i}}\widehat{r}(x_{j},\,y)
\,.
\label{e_lA3-2}
\end{equation}
Denote these sums by $\color{blue} \omega_{i,\,j}$ (which is symmetric
in its arguments).
\end{itemize}
Then,
\begin{enumerate}
\item The Markov chain $\{{\bf y}_{\mathfrak{D}}(t)\}_{t\ge0}$ is reversible
with respect to the measure $\rho$.
\item For any ${\bf g}:\mathscr{V}\to\mathbb{R}$ such that ${\bf g}(x)=0$
for all $x\notin\mathfrak{D}$, 
\[
-\sum_{x\in\mathfrak{D}}\rho(x){\bf g}(x)\mathfrak{L}{\bf g}(x)
=\frac{1}{2}\sum_{ i,\,j\in\llbracket1,\, m \rrbracket}
\omega_{i,\,j}\left[\widehat{{\bf g}}(x_{j})-\widehat{{\bf g}}(x_{i})\right]^{2}+\sum_{i\in\llbracket1,\,m\rrbracket}\rho(x_{i})\widehat{{\bf g}}(x_{i})^{2}\sum_{y\in\mathscr{S}\setminus\widehat{\mathfrak{D}}}\widehat{r}(x_{i},\,y)\,,
\]
where $\widehat{{\bf g}}:\mathscr{S}\to\mathbb{R}$ is the harmonic
extension of ${\bf g}$ defined in \eqref{e_harmonic}.
\end{enumerate}
\end{lem}

\begin{proof}
Consider the first assertion. Let $\rho$ be the measure introduced in
(c).  Let
\textcolor{blue}{$\{\widehat{{\bf
y}}_{\widehat{\mathfrak{D}}}(t)\}_{t\ge0}$} denote the process
obtained by reflecting $\{\widehat{{\bf y}}(t)\}_{t\ge0}$ at
$\widehat{\mathfrak{D}}$, i.e., jumps from $\widehat{\mathfrak{D}}$ to
$\mathscr{S}\setminus\widehat{\mathfrak{D}}$ are forbidden.

Since $\widehat{\mathfrak{D}}$ is an equivalence class, the chain
$\{\widehat{{\bf y}}_{\widehat{\mathfrak{D}}}(t)\}_{t\ge0}$ is
irreducible. Fix $i\in\llbracket1,\,m\rrbracket$. As
$\widehat{\mathfrak{D}}\setminus\widehat{\mathfrak{D}}_{i}\ne\varnothing$,
choose
$y\in\widehat{\mathfrak{D}}\setminus\widehat{\mathfrak{D}}_{i}$.
Since $\widehat{\mathfrak{D}}$ is an equivalence class, the chain
$\{\widehat{{\bf y}}_{\widehat{\mathfrak{D}}}(t)\}_{t\ge0}$ can reach
$y$ starting from $x_{i}$. Moreover, since $\widehat{r}(x,\,z)=0$ for
all $x\in\widehat{\mathfrak{D}}_{i}\setminus\{x_{i}\}$ and
$z\in\widehat{\mathfrak{D}}\setminus\widehat{\mathfrak{D}}_{i}$, there
exists
$y_{0}\in\widehat{\mathfrak{D}}\setminus\widehat{\mathfrak{D}}_{i}$
such that $\widehat{r}(x_{i},\,y_{0})>0$.  Thus,
\[
\sum_{y\in\widehat{\mathfrak{D}}\setminus\widehat{\mathfrak{D}}_{i}}\widehat{r}(x_{i},\,y)>0 \,.
\]
It then follows from \cite[Proposition B.2]{LLS-2nd} that the trace
process of $\{\widehat{{\bf y}}_{\widehat{\mathfrak{D}}}(t)\}_{t\ge0}$
on $\{x_{1},\,\dots,\,x_{m}\}$ is reversible with respect to the
measure $\rho$. Since $\{{\bf y}_{\mathfrak{D}}(t)\}_{t\ge0}$ is
the trace process of this process on $\mathfrak{D}$, \cite[Proposition 6.3]{BL}
implies that $\{{\bf y}_{\mathfrak{D}}(t)\}_{t\ge0}$ is reversible
with respect to the restriction of the measure $\rho$ to
$\mathfrak{D}$.

For the second assertion, we first establish the claim that for any
${\bf f}:\mathscr{V}\to\mathbb{R}$ and $i\in\llbracket1,\,m\rrbracket$,
\begin{equation}
\widehat{{\bf f}}(x)=\widehat{{\bf f}}(x_{i})\quad \text{for all}\ x\in\widehat{\mathfrak{D}}_{i}\,.\label{e_lA3-3}
\end{equation}

Fix a function ${\bf f}:\mathscr{V}\to\mathbb{R}$, $i\in\llbracket1,\,m\rrbracket$ and
$x\in\widehat{\mathfrak{D}}_{i}$.  It suffices to prove the claim for
$x\ne x_{i}$. By (a) and (b),
$\widehat{\mathcal{Q}}_{x}\left[H_{x_{i}}\le H_{\mathscr{V}}\right]=1$.
Therefore, by \eqref{e_harmonic} and the strong Markov property,
\[
\widehat{{\bf f}}(x)=\sum_{z\in\mathscr{V}}\widehat{\mathcal{Q}}_{x}\left[H_{\mathscr{V}}=H_{z}\right]{\bf f}(z)=\sum_{z\in\mathscr{V}}\widehat{\mathcal{Q}}_{x_{i}}\left[H_{\mathscr{V}}=H_{z}\right]{\bf f}(z)=\widehat{{\bf f}}(x_{i})\,,
\]
which proves \eqref{e_lA3-3}.

We turn to the second assertion. Let
${\bf g}:\mathscr{V}\to\mathbb{R}$ be such that ${\bf g}(x)=0$ for
$x\notin\mathfrak{D}$. For convenience, extend $\rho$ by setting
$\rho(x)=0$ for
$x\in\widehat{\mathfrak{D}}\setminus\{x_{1},\,\dots,\, x_m\}$.  By Lemma \ref{l_A1},
and since $\widehat{{\bf g}}$ is harmonic on
$\widehat{\mathfrak{D}} \setminus \mf D \subset \ms S \setminus \ms
V$,
\begin{align*}
-\sum_{x\in\mathfrak{D}}\rho(x){\bf g}(x)\mathfrak{L}{\bf g}(x)
=-\sum_{x\in\mathfrak{D}}\rho(x)\widehat{{\bf g}}(x)
\widehat{\mathfrak{L}}\widehat{{\bf g}}(x)
=-\sum_{x\in\widehat{\mathfrak{D}}}\rho(x)\widehat{{\bf
g}}(x)\widehat{\mathfrak{L}}\widehat{{\bf g}}(x)\,.
\end{align*}
By the decomposition of $\widehat{\mf D}$, this sum is equal to
\begin{align*}
-\sum_{i\in\llbracket1,\,m\rrbracket}
\sum_{x\in\widehat{\mathfrak{D}}_{i}}\rho(x)\widehat{{\bf
g}}(x)\widehat{\mathfrak{L}}\widehat{{\bf g}}(x)  
=\sum_{i\in\llbracket1,\,m\rrbracket}\sum_{x\in\widehat{\mathfrak{D}}_{i}}
\rho(x)\widehat{{\bf g}}(x)\sum_{y\in\mathscr{S}}
\widehat{r}(x,\,y)\left[\widehat{{\bf g}}(x)-\widehat{{\bf g}}(y)\right]\,.
\end{align*}

For fixed $i\in\llbracket1,\,m\rrbracket$,
by the decomposition \eqref{e_lA3-0}, the sum over $y$ decomposes as
\[
\sum_{y\in\widehat{\mathfrak{D}}_{i}}+\sum_{j\in\llbracket1,\,m\rrbracket\setminus\{i\}}\sum_{y\in\widehat{\mathfrak{D}}_{j}}+\sum_{y\in\mathscr{S}\setminus\widehat{\mathfrak{D}}}\,.
\]
By \eqref{e_lA3-3}, the first part is equal to
\begin{equation}
\sum_{i\in\llbracket1,\,m\rrbracket}\sum_{x\in\widehat{\mathfrak{D}}_{i}}\rho(x)\widehat{{\bf g}}(x)\sum_{y\in\widehat{\mathfrak{D}}_{i}}\widehat{r}(x,\,y)\left[\widehat{{\bf g}}(x)-\widehat{{\bf g}}(y)\right]=0\,.\label{e_lA3-4}
\end{equation}
By \eqref{e_lA3-1} and Lemma \ref{l_A2}, the third part equals
\begin{equation}
\begin{aligned} & \sum_{i\in\llbracket1,\,m\rrbracket}\sum_{x\in\widehat{\mathfrak{D}}_{i}}\rho(x)\widehat{{\bf g}}(x)\sum_{y\in\mathscr{S}\setminus\widehat{\mathfrak{D}}}\widehat{r}(x,\,y)\left[\widehat{{\bf g}}(x)-\widehat{{\bf g}}(y)\right]\\
 & =\sum_{i\in\llbracket1,\,m\rrbracket}\rho(x_{i})\widehat{{\bf g}}(x_{i})\sum_{y\in\mathscr{S}\setminus\widehat{\mathfrak{D}}}\widehat{r}(x_{i},\,y)\left[\widehat{{\bf g}}(x_{i})-\widehat{{\bf g}}(y)\right]\\
 & =\sum_{i\in\llbracket1,\,m\rrbracket}\rho(x_{i})\widehat{{\bf g}}(x_{i})^{2}\sum_{y\in\mathscr{S}\setminus\widehat{\mathfrak{D}}}\widehat{r}(x_{i},\,y)\,.
 \end{aligned}
\label{e_lA3-5}
\end{equation}

It remains to consider the second part.  For a fixed
$i\in\llbracket1,\,m\rrbracket$, by \eqref{e_lA3-1} and
\eqref{e_lA3-3},
\[
\begin{aligned}
& \sum_{x\in\widehat{\mathfrak{D}}_{i}}\rho(x)\widehat{{\bf
g}}(x)\sum_{j\in\llbracket1,\,m\rrbracket\setminus\{i\}}\sum_{y\in\widehat{\mathfrak{D}}_{j}}\widehat{r}(x,\,y)\left[\widehat{{\bf
g}}(x)-\widehat{{\bf g}}(y)\right]\\ 
& =\rho(x_{i})\widehat{{\bf g}}(x_{i})\sum_{j\in\llbracket1,\,m\rrbracket\setminus\{i\}}\sum_{y\in\widehat{\mathfrak{D}}_{j}}\widehat{r}(x_{i},\,y)\left[\widehat{{\bf g}}(x_{i})-\widehat{{\bf g}}(y)\right]\\
 & =\widehat{{\bf g}}(x_{i})\sum_{j\in\llbracket1,\,m\rrbracket\setminus\{i\}}\left[\widehat{{\bf g}}(x_{i})-\widehat{{\bf g}}(x_{j})\right]\rho(x_{i})\sum_{y\in\widehat{\mathfrak{D}}_{j}}\widehat{r}(x_{i},\,y)\,.
\end{aligned}
\]
By \eqref{e_lA3-2}, this is equal to
\[
\widehat{{\bf g}}(x_{i})\sum_{j\in\llbracket1,\,m\rrbracket\setminus\{i\}}\omega_{i,\,j}\left[\widehat{{\bf g}}(x_{i})-\widehat{{\bf g}}(x_{j})\right]\,.
\]
Summing over $i$ gives
\[
\begin{aligned} & \sum_{i\in\llbracket1,\,m\rrbracket}\sum_{x\in\widehat{\mathfrak{D}}_{i}}\rho(x)\widehat{{\bf g}}(x)\sum_{j\in\llbracket1,\,m\rrbracket\setminus\{i\}}\sum_{y\in\widehat{\mathfrak{D}}_{j}}\widehat{r}(x,\,y)\left[\widehat{{\bf g}}(x)-\widehat{{\bf g}}(y)\right]\\
 & =\sum_{i\in\llbracket1,\,m\rrbracket}\widehat{{\bf g}}(x_{i})\sum_{j\in\llbracket1,\,m\rrbracket\setminus\{i\}}\omega_{i,\,j}\left[\widehat{{\bf g}}(x_{i})-\widehat{{\bf g}}(x_{j})\right]\\
 & =\frac{1}{2}\sum_{ i ,\,j\in\llbracket1,\, m \rrbracket}
 \omega_{i,\,j}\left[\widehat{{\bf g}}(x_{j})-\widehat{{\bf g}}(x_{i})\right]^{2}\,.
\end{aligned}
\]
Combining this with \eqref{e_lA3-4} and \eqref{e_lA3-5} yields the desired identity, which completes the proof.
\end{proof}

The next lemma provides the analogue of Lemma \ref{l_A3} in the case $m=1$.

\begin{lem}
\label{l_A4}Fix an equivalence class $\mathfrak{D}\subset\mathscr{V}$
of $\{{\bf y}(t)\}_{t\ge0}$. Suppose that in the
decomposition \eqref{e_lA3-0}, $m=1$, and
\[
\begin{cases}
\sum_{y\in\mathscr{S}\setminus\widehat{\mathfrak{D}}}\widehat{r}(x,\,y)=0\ \text{for}\ x\in\widehat{\mathfrak{D}}\setminus\{x_{1}\}\,,\\
\sum_{y\in\mathscr{S}\setminus\widehat{\mathfrak{D}}}\widehat{r}(x_{1},\,y)>0\,.
\end{cases}
\]
Then, for any ${\bf g}:\mathscr{V}\to\mathbb{R}$ such that ${\bf g}(x)=0$
for all $x\notin\mathfrak{D}$, 
\[
{\bf g}(x_{1})\mathfrak{L}{\bf g}(x_{1})=
- \widehat{{\bf g}}(x_{1})^{2}\sum_{y\in\mathscr{S}\setminus\widehat{\mathfrak{D}}}\widehat{r}(x_{1},\,y)\,.
\]
\end{lem}

\begin{proof}
Let ${\bf g}:\mathscr{V}\to\mathbb{R}$.  Using the same argument as in
the proof of \eqref{e_lA3-3}, one can show that
\begin{equation}
\widehat{{\bf g}}(x)=\widehat{{\bf g}}(x_{1})\ \text{for all}\ x\in\widehat{\mathfrak{D}}\,.\label{e_lA4-2}
\end{equation}
By Lemma
\ref{l_A1},
\[
\begin{aligned}{\bf g}(x_{1})\mathfrak{L}{\bf g}(x_{1}) & =\widehat{{\bf g}}(x_{1})\widehat{\mathfrak{L}}\widehat{{\bf g}}(x_{1})\\
& =\widehat{{\bf g}}(x_{1})\sum_{y\in\mathscr{S}}
\widehat{r}(x_{1},\,y)\left[\widehat{{\bf g}}(y)-\widehat{{\bf g}}(x_1)\right]\,.
\end{aligned}
\]
By \eqref{e_lA4-2}, the terms with $y\in\widehat{\mathfrak{D}}$
vanish, so the sum reduces to 
\[
\widehat{{\bf g}}(x_{1})\sum_{y\in\mathscr{S}\setminus\widehat{\mathfrak{D}}}
\widehat{r}(x_{1},\,y)\left[\widehat{{\bf g}}(y)-\widehat{{\bf g}}(x_1)\right]\,.
\]
Assume now that ${\bf g}(x)=0$ for all $x\notin\mathfrak{D}$.  By
Lemma \ref{l_A2}, $\widehat{{\bf g}}(y)=0$ for
$y\in\mathscr{S}\setminus\widehat{\mathfrak{D}}$ with $\widehat r(x_1,y)>0$, so this equals
\[
-\, (\widehat{{\bf g}}(x_{1}))^2
\sum_{y\in\mathscr{S}\setminus\widehat{\mathfrak{D}}}\widehat{r}(x_{1},\,y)\,,
\]
which completes the proof.
\end{proof}

\subsection{\label{app_DV}Donsker--Varadhan functionals of Markov chains}

In this subsection, we recall some general results on Donsker--Varadhan
large deviation rate functionals for Markov chains. Let ${\color{blue}\mathfrak{J}}:\mathcal{P}(\mathscr{V})\to[0,\,\infty]$ denote
the large deviation rate functional  associated with the chain $\{{\bf y}(t)\}_{t\ge0}$,
defined by
\begin{equation}
\mathfrak{J}(\omega):=\sup_{{\bf u}>0}\,\sum_{x\in\mathscr{V}}-\frac{\mathfrak{L}{\bf u}(x)}{{\bf u}(x)}\omega(x)\,,\label{e_def_J}
\end{equation}
where the supremum is taken over all functions ${\bf u}:\mathscr{V}\to(0,\,\infty)$.

We first evaluate this functional on Dirac measures. For $x_{0}\in\mathscr{V}$,
\begin{equation}
\begin{aligned}\mathfrak{J}(\delta_{x_{0}}) & =\sup_{{\bf u}>0}-\sum_{x\in\mathscr{V}}\frac{\mathfrak{L}{\bf u}(x)}{{\bf u}(x)}\delta_{x_{0}}(x)\\
 & =\sup_{{\bf u}>0}-\sum_{y\in\mathscr{V}\setminus\{x_{0}\}}\frac{r(x_{0},\,y)}{{\bf u}(x_{0})}({\bf u}(y)-{\bf u}(x_{0}))\\
 & =\sum_{y\in\mathscr{V}\setminus\{x_{0}\}}r(x_{0},\,y)-\inf_{{\bf u}>0}\sum_{y\in\mathscr{V}\setminus\{x_{0}\}}\frac{r(x_{0},\,y)}{{\bf u}(x_{0})}{\bf u}(y)\\
 & =\sum_{y\in\mathscr{V}\setminus\{x_{0}\}}r(x_{0},\,y)\,.
\end{aligned}
\label{e_J_delta}
\end{equation}
In the last step, the infimum is zero, as can be seen by taking $u(x_0)=1$ and $u(y)=\eta$ for $y\ne x_0$, and letting $\eta\downarrow0$.

The previous computation extends to general probability measures $\omega\in \mathcal{P}(\mathscr{V})$.

\begin{lem}
\label{l_MC_DV_finite}For any $\omega\in\mathcal{P}(\mathscr{V})$,
$\mathfrak{J}(\omega)<\infty$.
\end{lem}

\begin{proof}
By the convexity of $\mathfrak{J}$ and \eqref{e_J_delta}, for any
$\omega\in\mathcal{P}(\mathscr{V})$, 
\[
\mathfrak{J}(\omega)\le\sum_{x\in\mathscr{V}}\omega(x)\mathfrak{J}(\delta_{x})<\infty\,.
\]
\end{proof}

Recall that for any equivalence class $\mathfrak{D}$ of $\{{\bf y}(t)\}_{t\ge0}$,
$\{{\bf y}_{\mathfrak{D}}(t)\}_{t\ge0}$ is
the Markov chain $\{{\bf y}(t)\}_{t\ge0}$ reflected at $\mathfrak{D}$.
The reflected chain $\{{\bf y}_{\mathfrak{D}}(t)\}_{t\ge0}$
is irreducible, and hence has a unique stationary distribution, denoted by
\textcolor{blue}{$\nu_{\mathfrak{D}}$}. Let
$\mathfrak{n}$ be the number of closed irreducible classes of the original
chain $\{{\bf y}(t)\}_{t\ge0}$, and denote them by \textcolor{blue}{$\mathscr{R}_{1},\,\dots,\,\mathscr{R}_{\mathfrak{n}}$}.

Since every stationary distribution of $\{{\bf y}(t)\}_{t\ge0}$
is a convex combination of $\nu_{\mathscr{R}_{a}}$, $a=1,\dots,\mathfrak{n}$, the following characterization holds, as stated in \cite[Lemma A.8]{Landim-Gamma}.
\begin{lem}[{\cite[Lemma A.8]{Landim-Gamma}}]
\label{l_MC_DV_irred}Let $\omega\in\mathcal{P}(\mathscr{V})$. Then,
$\mathfrak{J}(\omega)=0$ if and only if
\[
\omega=\sum_{a=1}^{\mathfrak{n}}\alpha(a)\nu_{\mathscr{R}_{a}}\,,
\]
for some $\alpha\in\mathcal{P}(\llbracket1,\,n\rrbracket)$.
\end{lem}

For an equivalence class $\mathfrak{D}$, denote by \textcolor{blue}{$\mathfrak{J}_{\mathfrak{D}}$}
the Donsker--Varadhan large deviation rate functional of the reflected
chain $\{{\bf y}_{\mathfrak{D}}(t)\}_{t\ge0}$. If $\mathfrak{D}=\{x_{0}\}$
for some $x_{0}\in\mathscr{V}$, then $\mathcal{P}(\mathfrak{D})=\{\delta_{x_{0}}\}$,
and we set $\mathfrak{J}_{\mathfrak{D}}(\delta_{x_{0}})=0$. Furthermore,
for $\omega\in\mathcal{P}(\mathscr{V})$ and $A\subset\mathscr{V}$,
let \textcolor{blue}{$\omega_{A}$} be the conditioned measure of
$\omega$ on $A$.

The following decomposition formula is a special case of \cite[Lemma A.7]{Landim-Gamma}, with $\omega$ supported on an equivalence class $\mathfrak{D}$.
\begin{lem}[{\cite[Lemma A.7]{Landim-Gamma}}]
\label{l_DV_equiv-1}Fix an equivalence class $\mathfrak{D}$. Then,
for all $\omega\in\mathcal{P}(\mathscr{V})$ supported on $\mathfrak{D}$,
\[
\mathfrak{J}(\omega)=\mathfrak{J}_{\mathfrak{D}}(\omega_{\mathfrak{D}})+\sum_{x\in\mathfrak{D}}\omega(x)\sum_{y\notin\mathfrak{D}}r(x,\,y)\,.
\]
\end{lem}

Let $\mathfrak{l}\in\mathbb{N}$ denote the number of equivalence
classes of the chain $\{{\bf y}(t)\}_{t\ge0}$, and denote them by
$\mathfrak{D}_{1},\,\dots,\,\mathfrak{D}_{\mathfrak{l}}$.  Recall that
$\mathfrak{n}$ denotes the number of the closed irreducible
classes so that $\mathfrak{n}\le\mathfrak{l}$. Reorder the equivalence classes so that
$|\mathfrak{D}_{a}|\ge2$ for $1\le a\le\mathfrak{m}$ and
$|\mathfrak{D}_{a}|=1$ for $\mathfrak{m}+1\le a\le\mathfrak{l}$.
Some of the equivalence classes with one element may be absorbing
states; the others consist of a single transient state.

\begin{lem}
\label{l_DV_equiv-2} For any $\omega\in\mathcal{P}(\mathscr{V})$,
\[
\mathfrak{J}(\omega)=\sum_{a\in{\omega_+}}\omega(\mathfrak{D}_{a})\,\mathfrak{J}(\omega_{\mathfrak{D}_{a}})\,,
\]
where $\omega_+ := \{k\in\llbracket1,\,\mathfrak{l}\rrbracket: \omega(\mathfrak{D}_k)>0\}$.
\end{lem}

\begin{proof}
For $a\in\llbracket\mathfrak{m}+1,\,\mathfrak{l}\rrbracket$, let $\mathfrak{D}_{a}=\{x_{a}\}$.
By \cite[Lemma A.7 and display (A.14)]{Landim-Gamma},
\[
\mathfrak{J}(\omega)=\sum_{a=1}^{\mathfrak{m}}\omega(\mathfrak{D}_{a})\mathfrak{J}_{\mathfrak{D}_{a}}(\omega_{\mathfrak{D}_{a}})+\sum_{a=1}^{\mathfrak{m}}\sum_{x\in\mathfrak{D}_{a}}\omega(x)\sum_{y\notin\mathfrak{D}_{a}}r(x,\,y)+\sum_{a=\mathfrak{m}+1}^{\mathfrak{l}}\omega(x_a)\sum_{y\in \mathscr{V}\setminus\{x_{a}\}}r(x_a,\,y)\,.
\]

For $a\in\llbracket\mathfrak{m}+1,\,\mathfrak{l}\rrbracket$, suppose that $\omega(x_{a})>0$. Then, by \eqref{e_J_delta},
\begin{equation}
\omega(\mathfrak{D}_{a})\mathfrak{J}(\omega_{\mathfrak{D}_{a}})=\omega(x_{a})\mathfrak{J}(\delta_{x_{a}})=\omega(x_{a})\sum_{y\in\mathscr{V}\setminus\{x_{a}\}}r(x_{a},\,y)\,.
\label{e_l_A8}
\end{equation}

For $a\in\llbracket1,\,\mathfrak{m}\rrbracket$ such that $\omega(\mathfrak{D}_{a})>0$, by Lemma \ref{l_DV_equiv-1},
\[
\omega(\mathfrak{D}_{a})\mathfrak{J}(\omega_{\mathfrak{D}_{a}})=\omega(\mathfrak{D}_{a})\mathfrak{J}_{\mathfrak{D}_{a}}(\omega_{\mathfrak{D}_{a}})+\sum_{x\in\mathfrak{D}_{a}}\omega(x)\sum_{y\notin\mathfrak{D}_{a}}r(x,\,y)\,,
\]
which, together with \eqref{e_l_A8}, yields the desired decomposition.
\end{proof}

Finally, the following formula is due to Donsker--Varadhan \cite[Theorem 5]{DV}.
\begin{lem}[{\cite[Theorem 5]{DV}}]
\label{l_DV_equiv-3}Let $\mathfrak{D}\subset\mathscr{V}$ be an
equivalence class such that $|\mathfrak{D}|\ge2$. Suppose that the
reflected chain $\{{\bf y}_{\mathfrak{D}}(t)\}_{t\ge0}$ is reversible
with respect to $\nu_{\mathfrak{D}}$. Then, for any $\omega\in\mathcal{P}(\mathfrak{D})$,
\[
\mathfrak{J}_{\mathfrak{D}}(\omega)=-\sum_{x\in\mathfrak{D}}\nu_{\mathfrak{D}}(x){\bf f}(x)\mathfrak{L}_{\mathfrak{D}}{\bf f}(x)\,,
\]
where
\[
{\bf f}(x):=\sqrt{\frac{\omega(x)}{\nu_{\mathfrak{D}}(x)}}\,.
\]
\end{lem}

\section{\label{sec_dom}Domain of generators}

Recall that the operator $\mathscr{L}_{\epsilon}:D(\mathscr{L}_{\epsilon})\subset L^{2}(d\pi_{\epsilon})\to L^{2}(d\pi_{\epsilon})$,
defined as the extension of \eqref{e_generator}, is the infinitesimal
generator of the process $\{\bm{x}_{\epsilon}(t)\}_{t\ge0}$ governed
by the SDE \eqref{e: SDE}. Define
\[
{\color{blue}C^{2}(\mathscr{L}_{\epsilon}}\mathclose{\color{blue})}:=\{f\in C^{2}(\mathbb{R}^{d}):f,\,-\nabla U\cdot\nabla f+\epsilon\Delta f\in L^{2}(d\pi_{\epsilon})\}\,.
\]

\begin{prop}
\label{p_gen}The infinitesimal generator $\mathscr{L}_{\epsilon}:D(\mathscr{L}_{\epsilon})\subset L^{2}(d\pi_{\epsilon})\to L^{2}(d\pi_{\epsilon})$
satisfies the following.
\begin{enumerate}
\item For every $\lambda>0$ and $g\in L^{2}(d\pi_{\epsilon})$, there exists
a unique solution $f\in D(\mathscr{L}_{\epsilon})$ to the resolvent
equation
\[
(\lambda-\mathscr{L}_{\epsilon})f=g\,.
\]

\item $C^{2}(\mathscr{L}_{\epsilon})\subset D(\mathscr{L}_{\epsilon})$,
and for all $f\in C^{2}(\mathscr{L}_{\epsilon})$,
\[
\mathscr{L}_{\epsilon}f=-\nabla U\cdot\nabla f+\epsilon\Delta f \,.
\] 
\end{enumerate}
\end{prop}

\begin{proof}
The first assertion is a direct consequence of the Hille--Yosida theorem.

We turn to the second assertion. Let $f\in C^{2}(\mathscr{L}_{\epsilon})$.
For $n\in\mathbb{N}$, let $(\xi_{n})_{n\ge1}$ be a sequence of smooth cutoff
functions such that
\[
\xi_{n}(x)=\begin{cases}
1 & |x|\le n\,,\\
0 & |x|\ge n+1\,,
\end{cases}
\]
and
\[
\sup_{n\in\mathbb{N}}\sup_{1\le j\le d}\left\Vert \frac{\partial\xi_{n}}{\partial x_{j}}\right\Vert _{L^{\infty}(\mathbb{R}^{d})}\,,\ \sup_{n\in\mathbb{N}}\sup_{1\le j,\,k\le d}\left\Vert \frac{\partial^{2}\xi_{n}}{\partial x_{j}\partial x_{k}}\right\Vert _{L^{\infty}(\mathbb{R}^{d})}<\infty\,.
\]
Then $\xi_{n}f\in C_{c}^{2}(\mathbb{R}^{d})\subset D(\mathscr{L}_{\epsilon})$
for all $n\in\mathbb{N}$. By elementary calculus, $\xi_{n}f\to f$
and $\mathscr{L}_{\epsilon}(\xi_{n}f)\to\epsilon\Delta f-\nabla U\cdot\nabla f$
in $L^{2}(d\pi_{\epsilon})$. Since $\mathscr{L}_{\epsilon}$ is closed
by the Hille--Yosida theorem, it follows that $f\in D(\mathscr{L}_{\epsilon})$
and $\mathscr{L}_{\epsilon}f=\epsilon\Delta f-\nabla U\cdot\nabla f$.
\end{proof}
For any matrix $\mathbb{M}$, define the matrix norm by
\[
{\color{blue}\|\mathbb{M}\|}:=\sup_{|\bm{y}|=1}|\mathbb{M}\bm{y}|\,.
\]

The following lemma shows that the assumption \eqref{e_L2} is not
restrictive. Note that the condition \eqref{e_Delta_nabla} appears
in \cite[Assumption 2]{LM}.

\begin{lem}
\label{l_assuL2}Suppose that $U$ satisfies \eqref{e_U_L1}. Assume
further that there exist $C>0$ and a compact set $\mathcal{K}\subset\mathbb{R}^{d}$
such that
\begin{equation}
\|\nabla^{2}U(\bm{x})\|\le C|\nabla U(\bm{x})|^{2}\ \ \text{for all}\ \bm{x}\notin\mathcal{K}\,.\label{e_Delta_nabla}
\end{equation}
Then there exists $\epsilon_{0}>0$ such that $|\nabla U|^2,\,\Delta U\in L^{2}(d\pi_{\epsilon})$
for $\epsilon\in(0,\,\epsilon_{0})$.
\end{lem}

\begin{proof}
By \eqref{e_Delta_nabla}, for $\bm{x}\notin\mathcal{K}$,
\[
|\Delta U(\bm{x})|=|{\rm Tr}\nabla^{2}U(\bm{x})|\le d\|\nabla^{2}U(\bm{x})\|\le dC|\nabla U(\bm{x})|^{2}.
\]
Therefore, it suffices to prove $|\nabla U|^{2}\in
L^{2}(d\pi_{\epsilon})$.

Fix $H>0$ large enough so that $\{U<H-1\}$ contains all critical
points of $U$ and $\mathcal{K}$, and $\{U<K\}$ is connected for all
$K\ge H-1$. Fix $\bm{x}\in\mathbb{R}^{d}$ such that $U(\bm{x})\ge
H$. Define the trajectory $\phi:[0,\,\infty)\to\mathbb{R}^{d}$ by
\[
\phi(0)=\bm{x}\,,\ \dot{\phi}(t)=-\nabla U(\phi(t))\,.
\]
Let
\[
T_{\bm{x}}:=\inf\{t>0:\phi(t)\in\{U\le H\}\}\,.
\]
By continuity, $U(\phi(T_{\bm{x}}))=H$.
Define the reversed path
\[
\psi(t)=\phi(T_{\bm{x}}-t)\ ;\ t \in [0,T_{\bm x}]\,,
\]
so that
\[
U(\psi(0))=H\,,\ \psi(T_{\bm{x}})=\bm{x}\,,\ \dot{\psi}(t)=\nabla U(\psi(t))\,.
\]
Differentiating yields
\[
\frac{d}{dt}(|\nabla U(\psi(t))|^{2}e^{-U(\psi(t))/a})=e^{-U(\psi(t))/a}\nabla U(\psi(t))^{\dagger}(2\nabla^{2}U(\psi(t))-\frac{1}{a}|\nabla U(\psi(t))|^{2}\mathbb{I}_{d})\nabla U(\psi(t))\,.
\]

Since $U(\psi(t))\ge H$ for all $t \in [0,T_{\bm x}]$, $\psi(t)\notin\mathcal{K}$.
If $a\in(0,\,(2C)^{-1})$, then by \eqref{e_Delta_nabla}, the matrix inside parentheses is negative definite, so the derivative above is strictly negative.
Thus, for $a\in(0,\,(2C)^{-1})$,
\begin{equation}
|\nabla U(\psi(0))|^{2}e^{-U(\psi(0))/a}\ge|\nabla U(\psi(T_{\bm{x}}))|^{2}e^{-U(\psi(T_{\bm{x}}))/a}=|\nabla U(\bm{x})|^{2}e^{-U(\bm{x})/a}\,.\label{e_l_L2}
\end{equation}

Define 
\[
M_{H}:=\sup_{\bm{x}\in\{U\le H\}}|\nabla U(\bm{x})|^{2}e^{-U(\bm{x})/a}\,.
\]
Then, for all $\bm{x}\notin\{U\le H\}$, the inequality \eqref{e_l_L2}
yields
\[
|\nabla U(\bm{x})|^{4}e^{-2U(\bm{x})/a}\le(M_{H})^{2} \,.
\]
Hence,
for $\epsilon\in(0,\,a/2)$ and $\bm{x}\notin\{U\le H\}$,
\[
\begin{aligned}
|\nabla U(\bm{x})|^{4}e^{-U(\bm{x})/\epsilon} &=|\nabla U(\bm{x})|^{4}e^{-2U(\bm{x})/a}e^{-(a-2\epsilon)U(\bm{x})/(a\epsilon)}\\
 & \le(M_{H})^{2}e^{-(a-2\epsilon)U(\bm{x})/(a\epsilon)}\,.
\end{aligned}
\]
By \eqref{e_U_L1}, the right-hand side is integrable. Therefore, $|\nabla U|^{2}\in L^{2}(d\pi_{\epsilon})$
for $\epsilon\in(0,\,a/2)$, completing the proof.
\end{proof}

\section{The energy landscape}

In this appendix, we recall several results on the energy landscape
from \cite{LLS-1st,LLS-2nd} which are used throughout the article.

\subsection{Landscape of potential $U$}

In this subsection, we summarize general properties of the landscape
of the potential $U$. The first result corresponds to \cite[Lemma A.4]{LLS-2nd}.
\begin{lem}
\label{l_pot1}Fix $H\in\mathbb{R}$. Let $\mathcal{V}\subset\mathbb{R}^{d}$
be a connected component of $\{U<H\}$. Let $\mathcal{M}\subset\mathcal{M}_{0}\cap\mathcal{V}$
and $\mathcal{M}'\subset\mathcal{M}_{0}\setminus\mathcal{M}$.
\begin{enumerate}
\item If $\mathcal{M}'\subset\mathcal{V}$, then $\Theta(\mathcal{M},\,\mathcal{M}')<H$.
Equivalently, if $\Theta(\mathcal{M},\,\bm{m})\ge H$ for all $\bm{m}\in\mathcal{M}'$,
then $\mathcal{M}'\subset\mathbb{R}^{d}\setminus\mathcal{V}$.
\item If $\mathcal{M}'\subset\mathbb{R}^{d}\setminus\mathcal{V}$, then
$\Theta(\mathcal{M},\,\mathcal{M}')\ge H$. Equivalently, if $\Theta(\mathcal{M},\,\bm{m})<H$
for all $\bm{m}\in\mathcal{M}'$, then $\mathcal{M}'\subset\mathcal{V}$.
\end{enumerate}
\end{lem}

The next lemma corresponds to \cite[Lemma 5.6-(1)]{LLS-2nd}.
\begin{lem}
\label{l_exit}Fix $p\in\llbracket1,\,\mathfrak{q}\rrbracket$ and
$H\in\mathbb{R}$. Let $\mathcal{V}$ be a connected component of
$\{U<H\}$ which does not separate $(p)$-states, and let $\mathcal{M}\in\mathscr{S}^{(p)}(\mathcal{V})$.
If $\mathcal{M}\to\mathcal{M}'$ for some $\mathcal{M}'\in\mathscr{S}^{(p)}(\mathcal{V}^{c})$,
then $\mathcal{M}=\mathcal{M}^{*}(\mathcal{V})$.
\end{lem}

\subsection{\label{subsec_valley}Metastable valleys}

In this subsection, we define the modulus ${\color{blue}r_{0}}>0$
associated with the metastable valleys \eqref{e_Em}. Following \cite[conditions (a)-(e) at the paragraph preceding (2.12)]{LLS-1st},
choose $r_{0}>0$ sufficiently small so that, for all $\bm{m}\in\mathcal{M}_{0}$,
the following hold.
\begin{itemize}
\item[(a)] $\overline{\mathcal{W}^{2r_{0}}(\bm{m})}\setminus\{\bm{m}\}$ does
not contain critical points of $U$.
\item[(b)]  For all $\bm{x}\in\mathcal{W}^{2r_{0}}(\bm{m})$ the trajectory $\phi:[0,\,\infty)\to\mathbb{R}^{d}$ defined by
\[
\phi(0)=\bm{x}\,,\ \dot{\phi}(t)=-\nabla U(\phi(t))\,,
\]
converges to $\bm{m}$.
\item[(c)] $-\nabla U(\bm{x})\cdot\bm{n}(\bm{x})<0$ for all $\bm{x}\in\partial\mathcal{W}^{2r_{0}}(\bm{m})$,
where $\bm{n}(\cdot)$ is the unit exterior normal vector of the boundary
of $\mathcal{W}^{2r_{0}}(\bm{m})$.
\item[(d)] $\mathcal{W}^{3r_{0}}(\bm{m})\subset B_{r_{5}(\bm{m})}(\bm{m})$.
\item[(e)] $\mathcal{W}^{2r_{0}}(\bm{m})\subset\mathcal{D}_{r_{4}(\bm{m})}^{\bm{m}}$.
\end{itemize}

It remains to present the definitions of $r_{4}(\bm{m}),\,r_{5}(\bm{m})>0$,
which are given in \cite[Section 3]{LLS-1st} and \cite[Appendix B]{LLS-1st},
respectively. For $\bm{m}\in\mathcal{M}_{0}$, let ${\color{blue}\mathbb{H}^{\bm{m}}}:=\nabla^{2}U(\bm{m})$ denote the Hessian of $U$ at $\bm{m}$.
By the Taylor expansion,
\[
\nabla U(\bm{x})\cdot\mathbb{H}^{\bm{m}}(\bm{x}-\bm{m})=\left[\mathbb{H}^{\bm{m}}(\bm{x}-\bm{m})+O(|\bm{x}-\bm{m}|^{2})\right]\cdot\mathbb{H}^{\bm{m}}(\bm{x}-\bm{m})=\left|\mathbb{H}^{\bm{m}}(\bm{x}-\bm{m})\right|^{2}+O(|\bm{x}-\bm{m}|^{3})\,,
\]
so that there exists ${\color{blue}r_{5}(\bm{m}}\mathclose{\color{blue})}>0$
such that
\[
\nabla U(\bm{x})\cdot\mathbb{H}^{\bm{m}}(\bm{x}-\bm{m})\ge\frac{1}{2}\left|\mathbb{H}^{\bm{m}}(\bm{x}-\bm{m})\right|^{2}\ \ \text{for all}\ \ \bm{x}\in B_{r_{5}(\bm{m})}(\bm{m})\,.
\]
For $\bm{x}\notin B_{r_{5}(\bm{m})}(\bm{m})$, define the projection
\[
\bm{r}^{\bm{m}}(\bm{x}):=\frac{r_{5}(\mb{m})}{|\bm{x}-\bm{m}|}(\bm{x}-\bm{m})+\bm{m}\in\partial B_{r_{5}(\bm{m})}(\bm{m})\,.
\]
Then, define a vector field $\bm{b}_{0}^{\bm{m}}:\mathbb{R}^{d}\to\mathbb{R}^{d}$
by
\[
\bm{b}_{0}^{\bm{m}}(\bm{x})=\begin{cases}
-\nabla U(\bm{x}) & \bm{x}\in B_{r_{5}(\bm{m})}(\bm{m})\,,\\
-\nabla U(\bm{r}^{\bm{m}}(\bm{x}))-\nabla^{2}U(\bm{r}^{\bm{m}}(\bm{x}))\,(\bm{x}-\bm{r}^{\bm{m}}(\bm{x})) & \bm{x}\in(B_{r_{5}(\bm{m})}(\bm{m}))^{c}\,.
\end{cases}
\]
By \cite[Proposition B.1]{LLS-1st}, this vector field $\bm{b}_{0}^{\bm{m}}$
satisfies the hypotheses of \cite[Section 3]{LLS-1st}.

As shown in \cite[Section 3]{LLS-1st}, for each $\bm{m}\in\mathcal{M}_0$, there exists a positive definite
matrix $\mathbb{K}^{\bm{m}}$ such that
	\[
\mathbb{H}^{\bm{m}}\mathbb{K}^{\bm{m}}+\mathbb{K}^{\bm{m}}\mathbb{H}^{\bm{m}}=\mathbb{I}\,,
\]
where $\mathbb{I}$ denotes the identity. Then, there exists $r_{4}'(\bm{m})>0$
such that 
\[
\left\Vert \left(D\bm{b}_{0}^{\bm{m}}(\bm{x})+\mathbb{H}^{\bm{m}}\right)^{\dagger}\mathbb{K}^{\bm{m}}+\mathbb{K}^{\bm{m}}\left(D\bm{b}_{0}^{\bm{m}}(\bm{x})+\mathbb{H}^{\bm{m}}\right)\right\Vert \le\frac{1}{2}\ \text{for all}\ \bm{x}\in B_{r_{4}'(\bm{m})}(\bm{m})\,.
\]
For $\bm{m}\in\mathcal{M}_{0}$ and $r>0$, define
\[
{\color{blue}\mathcal{D}_{r}^{\bm{m}}}:=\left\{ \bm{x}\in\mathbb{R}^{d}:(\bm{x}-\bm{m})\cdot\mathbb{H}^{\bm{m}}(\bm{x}-\bm{m})\le r^{2}\right\} \,.
\]
Then, there exists a sufficiently small ${\color{blue}r_{4}(\bm{m}}\mathclose{\color{blue})}>0$
such that $\mathcal{D}_{2r_{4}(\bm{m})}^{\bm{m}}\subset B_{\min\{r_{4}'(\bm{m}),\,r_{5}(\bm{m})\}}(\bm{m})$.

\begin{acknowledgement*}
C. L. has been partially supported by FAPERJ CNE E-26/201.117/2021,
by CNPq, Brazil Bolsa de Produtividade em Pesquisa PQ 305779/2022-2,
and by the French National Research Agency (ANR) via Project ANR-25-CE40-6875-01 (ANR DySLoS).
J. L. has been supported by the KIAS Individual Grant (HP093101) at Korea Institute for Advanced Study,
the Basic Science Research Program through the National Research Foundation of Korea(NRF) funded by the Ministry of Science and ICT (RS-2026-25599088),
and INHA UNIVERSITY Research Grant.
J. L. is grateful for the invitation from IMPA, where a fruitful discussion took place.
\end{acknowledgement*}

\end{document}